\documentclass[a4paper,11pt]{article}
\usepackage[utf8]{inputenc}
\usepackage{a4wide}
\usepackage{amsmath, amsthm}
\usepackage{amssymb}
\usepackage{array}
\usepackage{graphicx}
\usepackage{ifpdf}
\usepackage{hyperref}
\usepackage{comment}
\usepackage{mathtools}
\usepackage{arydshln}
\usepackage{multicol}
\usepackage{color}
\usepackage{enumerate}
\usepackage{stmaryrd}
\usepackage{cases}
\usepackage{multirow}

\usepackage{geometry}

\allowdisplaybreaks

\newcommand\restr[2]{{
  \left.\kern-\nulldelimiterspace 
  #1 
  \vphantom{\big|} 
  \right|_{#2} 
  }}

\usepackage{tkz-graph}
\tikzset{
  VertexStyle/.append style = {shape=circle,draw, fill=black, minimum size=9pt, inner sep=-1pt},
  EdgeStyle/.append style = {-, thick},
  LoopStyle/.append style = {-}}
\usetikzlibrary{arrows.meta}
\usepackage{tikz}
\usetikzlibrary{arrows,shapes,decorations.markings}
\usetikzlibrary{positioning}
\tikzset{>={Latex[width=2.5mm,length=2.5mm]}}

\newcommand{\Z}{\mathbb{Z}}

\newcommand{\R}{\mathbb{R}}

\newcommand{\T}{\mathcal{T}}
\newcommand{\x}{\mathbf{x}}
\newcommand{\y}{\mathbf{y}}

\newcommand{\rr}{\mathfrak{r}}
\newcommand{\HH}{\mathcal{H}}

\theoremstyle{plain}
\newtheorem{theorem}{Theorem}[section]
\newtheorem{lemma}[theorem]{Lemma}
\newtheorem{proposition}[theorem]{Proposition}

\newtheorem{corollary}[theorem]{Corollary}

\theoremstyle{definition}
\newtheorem{example}[theorem]{Example}
\newtheorem{definition}[theorem]{Definition}

\theoremstyle{remark}
\newtheorem{remark}[theorem]{Remark}

\title{Tutte's dichromate for signed graphs}
\author{Andrew Goodall\thanks{Corresponding author. Charles University, Prague, Czech Republic. Email: \texttt{andrew@iuuk.mff.cuni.cz}. Supported by Project ERCCZ LL1201 Cores and Czech Science Foundation GA \v{C}R 19-21082S.} \and Bart Litjens\thanks{University of Amsterdam, Netherlands. Email: \texttt{bart\_litjens@hotmail.com}. Supported by the European Research Council under the European Union's Seventh Framework Programme (FP7/2007-2013) / ERC grant agreement n$\mbox{}^{\circ}$ 339109.} \and Guus Regts\thanks{University of Amsterdam,  Netherlands. Email: \texttt{guusregts@gmail.com}. Supported by a NWO Veni grant.}\and Llu\'is Vena\thanks{University of Amsterdam, Netherlands. Email: \texttt{lluis.vena@gmail.com}. Supported by the European Research Council under the European Union's Seventh Framework Programme (FP7/2007-2013) / ERC grant agreement n$\mbox{}^{\circ}$ 339109.}}

\begin{document}
\maketitle

\begin{abstract}
\noindent We introduce the {\em trivariate Tutte polynomial} of a signed graph as an invariant of signed graphs up to vertex switching that contains among its evaluations the number of proper colorings and the number of nowhere-zero flows. In this, it parallels the Tutte polynomial of a graph, which contains the chromatic polynomial and flow polynomial as specializations. The number of nowhere-zero tensions (for signed graphs they are not simply related to proper colorings as they are for graphs) is given in terms of evaluations of the trivariate Tutte polynomial at two distinct points. Interestingly, the bivariate dichromatic polynomial of a biased graph, shown by Zaslavsky to share many similar properties with the Tutte polynomial of a graph, does not in general yield the number of nowhere-zero flows of a signed graph. Therefore the ``dichromate" for signed graphs (our trivariate Tutte polynomial) differs from the dichromatic polynomial (the rank-size generating function).  

The trivariate Tutte polynomial of a signed graph can be extended to an invariant of ordered pairs of matroids on a common ground set -- for a signed graph, the cycle matroid of its underlying graph and its frame matroid form the relevant pair of matroids. This invariant is the canonically defined Tutte polynomial of matroid pairs on a common ground set in the sense of a recent paper of Krajewski, Moffatt and Tanasa, and was first studied by Welsh and Kayibi as a four-variable linking polynomial of a matroid pair on a common ground set. 
\end{abstract}



\section{Introduction}
Signed graphs, introduced by Harary~\cite{har53}, are graphs (loops and multiple edges allowed) in which each edge is given a positive or negative sign. A large literature has accumulated on signed graphs~\cite{zas96}. Notably, 
Zaslavsky~\cite{zas82b,zas82a,zas82,zas92} developed the theory of signed graphs with respect to colorings, orientations and matroids associated with signed graphs. 
Just as colorings, flows and orientations of a graph may be defined with reference solely to the underlying cycle matroid of the graph, so may colorings, flows and orientations of a signed graph be defined in terms of the underlying frame matroid of the signed graph. 

For a finite additive abelian group $G$,  a (nowhere-zero) $G$-flow of a graph is defined with reference to an arbitrary orientation of the graph as an assignment of (non-zero) elements of $G$ to its edges such that Kirchhoff's law is satisfied, i.e. for each vertex the sum of values on its incoming edges is equal to the sum of values on its outgoing edges.  
The Tutte polynomial of a graph includes as specializations the chromatic polynomial, counting proper vertex colorings, and the flow polynomial, counting nowhere-zero $G$-flows (the number of which is independent of the orientation of the graph needed to define a flow, and, as Tutte showed~\cite{tutte54}, depends only on $|G|$).  To a (proper) vertex coloring of a graph using at most $n$ colors corresponds, by taking the ``potential difference'' at the ends of an edge, a (nowhere-zero) $G$-tension of the graph, where $G$ is an additive abelian group or order $n$; conversely, for a connected graph there are exactly $|G|$ (proper) vertex colorings for any given (nowhere-zero) $G$-tension. (See the paragraph on tensions and colorings below for how $G$-tensions are defined.) 

Tutte~\cite{tutte54} showed how the \emph{dichromatic polynomial} of a graph (Whitney rank generating function) is equivalent to the \emph{dichromate} of the graph (specializing to the chromatic polynomial and, dually, to the flow polynomial). Zaslavsky~\cite{zas95} defines the dichromatic polynomial of a biased graph (of which a signed graph is a special case) and develops analogous properties to the dichromatic of a graph. In contrast to the case of graphs, however,  Zaslavsky's bivariate dichromatic polynomial of a signed graph does not qualify as being the dichromate of a signed graph since  it does not in general yield the number of nowhere-zero flows as an evaluation. 
In order to enumerate both flows and colorings of signed graphs the dichromatic is not sufficient: the \emph{trivariate Tutte polynomial} of this paper is needed, which contains the dichromatic polynomial as a specialization (see Section~\ref{sec:signed_tutte} below). 

Our main results are that, similarly to the Tutte polynomial of a graph, the trivariate Tutte polynomial of a signed graph  includes among its evaluations the number of (nowhere-zero) $G$-flows and the number of proper $G$-colorings for a given finite additive abelian group~$G$. (The points of evaluation depend on $|G|$ and $|2G|$, where $2G=\{x+x:x\in G\}$.)  
Furthermore, we show that the number of nowhere-zero $G$-tensions of a signed graph can be obtained from its trivariate Tutte polynomial evaluated at two different points (dependent on $|G|$ and $|2G|$). For signed graphs, the number of (nowhere-zero) $G$-tensions is not simply related to the number of (proper) $G$-colorings: there is a group homomorphism from $G$-colorings to $G$-tensions, as for graphs, but in the case of signed graphs this homomorphism is not in general surjective except when $|G|$ is odd.  




\paragraph{Flows}
Flows are defined for signed graphs in a similar way to graphs~\cite{bouchet83} via Kirchhoff's law. 
The number of nowhere-zero $G$-flows of a signed graph, however, depends on the number of elements of order $2$ in $G$ (the order of its 2-torsion subgroup). For a finite additive abelian group $G$ with 2-torsion subgroup of order $2^d$, Beck and Zaslavsky~\cite{beck06} showed that when $d=0$ (so $|G|$ is odd), the number of nowhere-zero $G$-flows of a given signed graph is a polynomial in $|G|$, given as an evaluation of the Tutte polynomial of the underlying frame matroid.  DeVos et al.~\cite{devos17}, by establishing a deletion-contraction recurrence for the number of nowhere-zero $G$-flows reducing its evaluation to single-vertex signed graphs consisting solely of loops, showed that it is a polynomial in $2^d$ and $|G|/2^d$. In particular, this number is not an evaluation of the Tutte polynomial of the underlying frame matroid (unless $|G|$ is odd or the signed graph is balanced).
 In this paper we show that the number of nowhere-zero $G$-flows is an evaluation of the trivariate Tutte polynomial, thereby establishing a subset expansion for this number (Corollary~\ref{cor:flowcount_signed}). While this paper was in preparation, this expansion was found independently and using different methods by Qian and Ren~\cite[Corollary~4]{qian18}.

\paragraph{Tensions and colorings}
 In Section~\ref{sec:colorings} we begin by reviewing Zaslavsky's enumeration of proper $n$-colorings and proper non-zero $n$-colorings of a signed graph, and show that these are evaluations of our signed graph Tutte polynomial: the number of proper $n$-colorings is an evaluation of the Tutte polynomial of the underlying frame matroid, while the number of proper non-zero $n$-colorings is not such an evaluation. 
Zaslavsky's definition of signed graph proper $n$-colorings, in which colors are elements of $\{0,\pm 1,\dots, \pm n\}$, has a natural generalization to {\em proper $(X,\iota)$-colorings}, in which colors are taken from a finite set $X$ equipped with an involution $\iota$: we enumerate these colorings, showing that they are evaluations of the trivariate Tutte polynomial at points dependent on $|X|$ and the number of elements of $X$ that are fixed by $\iota$ (Theorem~\ref{thm:gcolors}). 
When $X$ is the set of elements of a finite additive abelian group $G$, we take the involution $\iota$ to be negation: fixed points of the involution are elements of $G$ of order 1 or 2, and a proper $G$-coloring is one that for each edge assigns colors $a$ and $b$ to its endpoints so that $a\neq b$ ($a\neq -b$) when the edge is positive (negative). 

For graphs, $G$-tensions are assignments of elements of an additive abelian group $G$ to edges with the property that the cumulative sum of edge values encountered when traversing a closed walk is zero, where edge values are negated when the direction of the walk opposes the orientation of the edge. This definition may be more compactly phrased in terms of oriented circuits of the underlying graphic matroid.  For signed graphs, a similar definition applies, with an orientation of an edge consisting in orienting its two half-edges, and only closed walks traversing an even number of negative edges being considered (see Figure~\ref{fig:def_ten}). Likewise, the definition of a signed graph tension may be more compactly formulated in terms of oriented circuits of the underlying frame matroid, as is done in~\cite{chen09}.

The relation of signed graph tensions to signed graph vertex colorings (see~\cite{ zas82b,zas82a}) is not as straightforward as that of graph tensions to graph vertex colorings. 
In a similar way to graphs, each (proper) $G$-coloring of a signed graph yields a (nowhere-zero) $G$-tension, in which the value on an edge is given by taking the difference (sum) of the colors of its endpoints when the edge is positive (negative). In general, however, not every nowhere-zero $G$-tension arises from a $G$-coloring in this way: tensions that do arise from $G$-colorings will be called {\em $G$-potential differences}. (For graphs, potential differences and tensions coincide.) In Section~\ref{sec:tensions} we give an equivalent definition of $G$-potential differences that refers only to the underlying frame matroid of the signed graph, i.e. independent of coloring vertices. Given the correspondence between (nowhere-zero) $G$-potential differences and (proper) $G$-colorings, we can use the previous enumeration of (proper) $n$-colorings to deduce directly that the number of (nowhere-zero) $G$-potential differences is given by an evaluation of the trivariate Tutte polynomial.   
 In Section~\ref{sec:tensions} we also establish that while the number of $G$-tensions (edges allowed to take value $0$) is an evaluation of the trivariate Tutte polynomial, the number of nowhere-zero $G$-tensions is generally not given by an evaluation of the trivariate Tutte polynomial at a single point. For a connected signed graph, the number of nowhere-zero $G$-tensions is, however, up to a sign depending on the nullity of the underlying graph, equal to a linear combination of evaluations of the trivariate Tutte polynomial at two different points (dependent on $|G|$ and $|2G|$), where the prefactors of these two evaluations depend only on $|2G|$ and on whether the signed graph is balanced or unbalanced. 


\paragraph{Matroids}
Much as the Tutte polynomial of a graph depends only on its underlying cycle matroid, the trivariate Tutte polynomial of a signed graph depends only on its frame matroid and the cycle matroid of its underlying graph. The number of nowhere-zero flows and nowhere-zero tensions of a signed graph in general depend both on its underlying cycle matroid and on its frame matroid, and it is for this reason the Tutte polynomial of the frame matroid (equivalent to the dichromatic polynomial~\cite{zas95}) fails to give them. 
The trivariate Tutte polynomial of a signed graph extends to a polynomial invariant of arbitrary pairs of matroids on a common ground set, 
which includes the Las Vergnas polynomial~\cite{vergnas80} in the special case of matroid perspectives and is equivalent to the four-variable {\em linking polynomial} of a pair of matroids of Welsh and Kayibi~\cite{WK04}, as explained in Appendix~\ref{app:matroids}. 

Krajewski et al.~\cite{krajewski18}, using a Hopf algebra framework, unify many of the existing extensions of the Tutte polynomial of a graph to other combinatorial objects (for example, those of Las Vergnas~\cite{vergnas80}, Bollob\'as and Riordan~\cite{bollobas01,bollobas02}, Krushkal~\cite{krushkal11}, and Butler~\cite{butler12} to embedded graphs, that of Crapo to matroids~\cite{crapo69}, and that of $\Delta$-matroids by Chun et al.~\cite{CMNR16}). Dupont et al.~\cite{dupont18} extend this approach to bialgebras more generally.
For a given graded class of combinatorial objects with a notion of deletion and contraction, there is a canonically defined invariant that shares with the Tutte polynomial of a graph the property of satisfying a deletion-contraction recurrence terminating in trivial objects, being universal for deletion-contraction invariants on the class, having a subset sum expansion, and satisfying a duality formula. This canonical invariant accordingly merits the designation of being the Tutte polynomial for the class. 
The trivariate Tutte polynomial defined in this paper is in this sense the canonically defined Tutte polynomial for the class of equivalence classes of signed graphs under switching (flipping signs of edges incident with a common vertex, a loop preserving its sign under this operation). 



\paragraph{Summary of enumerative results}
Table~\ref{table:summary} below summarizes the main enumerative results of this paper given by evaluations of the trivariate Tutte polynomial, placed alongside the comparable evaluations of the Tutte polynomial for graphs. In it, $\Gamma=(V,E)$ is a graph and $\Sigma=(\Gamma,\sigma)$ is a signed graph with signature $\sigma$, $G$ is a finite additive abelian group, and $2G=\{2x:x\in G\}$. 


\begin{table}[ht]
\centering
\begin{tabular}{|@{\:\:}p{4cm}@{\:}|c|@{\:\:}p{4.5cm}@{\:}|c|}
\hline
 & Tutte polynomial   & \centering{trivariate Tutte polynomial}  &  \\  &
$T_\Gamma(X,Y)$, \eqref{eq:Tutte_subset}  &  \centering{$T_\Sigma(X,Y,Z)$, Def.\ref{defn:signedtutte}} & \\
\hline
nowhere-zero $G$-flows & $(0,1-|G|)$ & \centering{$\Big(0,1-|G|,1-\frac{|G|}{|2G|}\Big)$}        & Theorem~\ref{thm:nz_flows} \\ \hline
proper $G$-colorings	&    $(1-|G|,0)$  &   \centering{$\Big(1-|G|,0,1-\frac{1}{|2G|}\Big)$}       & Corollary \ref{thm:gcolors} \\ \hline

	nowhere-zero \newline $G$-tensions & $(1-|G|,0)$ &    \centering{$\Big(1-|G|,0,1-\frac{1}{|2G|}\Big)$} \newline \centering{and $(1-|G|,0,1)$}        & Theorem~\ref{theorem:nz-tensions} \\ \cline{1-1} \cline{3-4}
\hline
\end{tabular}

\caption{Evaluation points of the trivariate Tutte polynomial involved in enumerations of proper colorings, nowhere-zero flows and nowhere-zero tensions of signed graphs compared with the analogous enumerations for graphs. In the second column the values are those given to $(X,Y)$, and in the third column the values are those given to $(X,Y,Z)$. }\label{table:summary}
\end{table}

\paragraph{Organization}
In the opening Section~\ref{sec:matroids} we briefly introduce signed graphs and their underlying graphic and frame matroid,  singling out those properties required in the sequel.
In Section~\ref{section:ST} we introduce the trivariate Tutte polynomial of a signed graph via its spanning subgraph expansion; in Appendix~\ref{app:matroids} we explain how this invariant of signed graphs is a special case of the trivariate Tutte polynomial of a pair of matroids on a common ground set, obtained by taking the matroids to be the underlying graphic matroid and frame matroid of the signed graph. 
We end Section~\ref{section:ST} by connecting the trivariate Tutte polynomial of a signed graph to other related polynomials defined in the literature.
(In Appendix~\ref{app:dictionary} we give an overview of many of these polynomials.) In Section~\ref{sec:del_cont_recipe} we show how the trivariate Tutte polynomial behaves under deletion-contraction, and give a ``Recipe Theorem" for the polynomial for signed graphs, which is the key result for proving the enumerations in
Sections~\ref{section:flows}--\ref{sec:tensions} given as evaluations of the trivariate Tutte polynomial (see Table~\ref{table:summary} above). 


\section{Signed graphs and their matroids}\label{sec:matroids}

\subsection{Signed graphs}\label{sec:signedgraphs}

A \emph{signed graph} 
 is a pair $\Sigma=(\Gamma,\sigma)$, where $\Gamma = (V,E)$ is a finite undirected graph possibly with loops and multiple edges, called the \emph{underlying graph} of $\Sigma$, and $\sigma$ is a function $\sigma:E\to \{-1,1\}$  that associates a sign to each edge of $\Gamma$, called the {\em signature} of $\Sigma$. 
A cycle $C=(v_1,e_1,v_2,\ldots,v_k,e_k,v_{1})$ in $\Gamma$ is called \emph{balanced} in $\Sigma$ if $\prod_{i=1}^k \sigma(e_i)=1$ and \emph{unbalanced} otherwise. 
The signed graph $\Sigma=(\Gamma,\sigma)$ is itself called \emph{balanced} if each cycle of $\Gamma$ is balanced in $\Sigma$ and \emph{unbalanced} otherwise. 

For a signed graph $\Sigma = (\Gamma,\sigma)$, we define $k(\Sigma) := k(\Gamma)$, the number of connected components of the underlying graph $\Gamma$. The number of balanced and unbalanced connected components of $\Sigma$ are denoted by $k_b(\Sigma)$ and by $k_u(\Sigma)$, respectively. Thus $k(\Sigma) = k_b(\Sigma) + k_u(\Sigma)$.

\emph{Switching} at a vertex $v$ means negating the sign of every edge that is incident with $v$, while keeping the sign of each loop attached to $v$. 
We say that two signed graphs $\Sigma_1 = (\Gamma_1,\sigma_1)$ and $\Sigma_2 = (\Gamma_2,\sigma_2)$ are \emph{equivalent} if the graph $\Gamma_1$ is isomorphic to the graph $\Gamma_2$, and if, under such an isomorphism, the signature $\sigma_1$ can be obtained from $\sigma_2$ by a sequence of switchings at vertices.
The property of being balanced or unbalanced is constant on equivalence classes of signed graphs, as being balanced or unbalanced is invariant under switching at a vertex.

The \emph{deletion} of an edge $e$ in $\Sigma=(\Gamma,\sigma)$ yields the signed graph $(\Gamma\backslash e, \sigma')$, where $\sigma'$ is the restriction of $\sigma$ to $E\backslash  \{e\}$ and where $\Gamma\backslash e$ is the graph obtained from $\Gamma$ by deleting $e$ as a graph edge. This signed graph obtained from $\Sigma$ be deleting $e$ is denoted by $\Sigma\backslash e$, and the signed graph obtained by deleting all the edges in $A\subseteq E$ is denoted by $\Sigma\backslash A$. For $A\subseteq E$ we write $A^c$ for the complementary subset $E\backslash A$; the signed graph $\Sigma\backslash A^c$ is restriction of $\Sigma$ to $A$. 

The \emph{contraction} of a non-loop edge $e$ of $\Gamma$ that has positive sign in $\Sigma=(\Gamma,\sigma)$ yields the signed graph $(\Gamma/e,\sigma')$,  where $\sigma'$ is the restriction of $\sigma$ to $E\backslash  \{e\}$ and where $\Gamma/e$ is the graph obtained from $\Gamma$ by contracting $e$ as a graph edge. The signed graph obtained from $\Sigma$ by contracting $e$ is denoted by $\Sigma/e$. Note that by switching we can always ensure that the sign of a non-loop edge is positive. When $e$ is a loop with positive sign in $\Sigma$ we set $\Sigma/e=\Sigma\backslash e$. In order to define contraction of negative loops, Zaslavsky~\cite{zas82a} enlarges the definition of signed graphs to include half-arcs and free loops. In our case, as we can avoid contracting negative edges and loops, we do not define the contraction $\Sigma/A$ by an arbitrary subset of edges $A$, only for subsets of positive edges (after possible switching -- Zaslavsky~\cite{zas82} shows that the order in which the edges are contracted does not affect the outcome).

\subsection{Matroids}\label{subsection:biasedgraphs}

For further background on matroids see~\cite{oxley06}. Here we highlight what is needed in the sequel. 

A matroid may -- among many ``cryptomorphic" axiomatizations -- be defined in terms of its collection of \emph{independent sets}, by its collection of \emph{bases} (independent sets of maximum size), by its \emph{circuits} (minimal dependent sets), or by its rank function (size of a maximal independent subset).
 
A {\em coloop} of a matroid is defined by the property that it belongs to no circuit; equivalently, a coloop belongs to every basis. 
A {\em loop} of a matroid has the defining property that it belongs to no independent set of edges, and in particular to no basis. 
An element of a matroid that is neither a loop nor a coloop is~{\em ordinary}.

\subsubsection{The cycle matroid of a graph}

To a graph $\Gamma=(V,E)$ there corresponds a matroid $M(\Gamma)$ on ground set $E$ with rank function
defined for $A \subseteq E$ by
$r_M(A)=|V|-k(\Gamma\backslash A^c)$ (equal to the size of a maximal spanning forest of $\Gamma\backslash A^c$).  
For a connected graph $\Gamma$, the bases of $M(\Gamma)$ are the edge sets of spanning trees of $\Gamma$, the independent sets are the edge sets of spanning forests of $\Gamma$, and the dependent sets the edge sets of spanning subgraphs containing a cycle of $\Gamma$. The circuits (minimal dependent sets) are edge sets of spanning subgraphs minimal with respect to containing a cycle in $\Gamma$, called {\em circles} in~\cite{zas82}. 

A coloop in the matroid $M(\Gamma)$ is a bridge of $\Gamma$ (deleting a bridge increases the number of connected components by one). A loop in $M(\Gamma)$ is an edge $e=uv$ of $\Gamma$ with $u=v$.

\subsubsection{The frame matroid of a signed graph}\label{sec:frame_matroid}

The material in this section is drawn from Chapter~6.10 of~\cite{oxley06}.

A subdivision of the left-hand graph in Figure~\ref{fig:M2} is a \textit{tight handcuff}, and a subdivision of the right-hand graph is a \textit{loose handcuff}. 
A loose handcuff or a tight handcuff in $\Sigma$ is \textit{unbalanced} if both its cycles are unbalanced in $\Sigma$.
\begin{figure}[h!]
\centering
\includegraphics{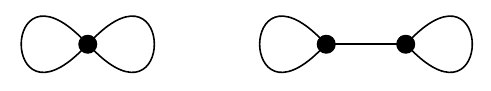}
\vspace{-4mm}
\caption{Tight handcuff and loose handcuff.}\label{fig:M2}
\end{figure}

%
%

To a signed graph $\Sigma=(\Gamma = (V,E),\sigma)$ is associated its \textit{frame matroid} $F(\Sigma)$ on ground set $E$ with rank function defined for $A \subseteq E$ by
\begin{equation}\label{equation:ranksgm}
r_F(A) = |V|-k_b(\Sigma\backslash  A^c).
\end{equation}


When $\Sigma$ is connected, the bases of $F(\Sigma)$ correspond to spanning trees of $\Gamma$ when $\Sigma$ is balanced, and in that case the two matroids coincide $F(\Sigma)=M(\Gamma)$; otherwise, when $\Sigma$ is unbalanced, a basis of $F(\Sigma)$ is the union of edges in a spanning forest $S$ of $\Gamma$ (formed by the trees $\{T_1,\ldots,T_{k(S)}\}$), and edges $\{e_i\}_{i\in[k(S)]}$, $e_i\not\in S$, such that the unique cycle of $\Gamma$ in $T_i \cup \{e_i\}$ is unbalanced in $\Sigma$, for each $i\in[k(S)]$. In particular, if $\Sigma$ is connected and balanced, each basis element contains $|V(\Gamma)|-1$ edges, whereas if $\Sigma$ is connected and unbalanced, it contains $|V(\Gamma)|$ edges.
For a not necessarily connected signed graph $\Sigma$, a basis of $F(\Sigma)$ is the union of bases of the frame matroids of the connected components of $\Sigma$.

 The circuits of $F(\Sigma)$ are the balanced cycles, unbalanced loose handcuffs and unbalanced tight handcuffs. 

A loop of $F(\Sigma)$ is a loop of $\Gamma$ with positive sign in $\Sigma$. A loop of $\Gamma$ with negative sign in $\Sigma$ is either contained in no circuit of $F(\Sigma)$ (when there are no other unbalanced cycles in the same connected component), in which case it is a coloop of $F(\Sigma)$, or it belongs to some circuit and is thereby an ordinary edge of $F(\Sigma)$. 
A coloop of $M(\Gamma)$ (a bridge of $\Gamma$) is also a coloop of $\Sigma$, except when it is a circuit path edge (an edge on the path joining the two unbalanced cycles of a loose handcuff), in which case it is ordinary in $\Sigma$.

\paragraph{Loops and coloops in signed graphs} The relationships between the notions of loop, coloop and ordinary edge for the matroids $M(\Gamma)$ and $F(\Sigma)$ are summarized in Table~\ref{fig.1}.

\begin{table}[ht!]
\centering
{\small
\begin{tabular}{|c|p{10cm}|c|c|}
	\hline
	$\Gamma$ &\centering $\Sigma$ & $M(\Gamma)$ & $F(\Sigma)$ \\
	\hline
	
	ordinary  &  edge not a loop or bridge in $\Gamma$ that when removed does not change the number of balanced connected components & ordinary & ordinary \\
	\hline
	
ordinary & edge not a loop in $\Gamma$ that belongs to every unbalanced cycle in its connected component (of which there is at least one) & ordinary & coloop \\
	\hline
	
	bridge  & bridge of $\Gamma$ in an unbalanced loose handcuff  & coloop & ordinary \\
	\hline
	bridge  & bridge of $\Gamma$ in no unbalanced loose handcuff & coloop & coloop \\
	\hline

	loop  & loop of $\Gamma$ with negative sign, other unbalanced cycles in its connected component & loop & ordinary\\
	\hline
	
	loop & loop of $\Gamma$ with negative sign, no other unbalanced cycles in its connected component & loop & coloop \\
	\hline
	
	loop  & loop of $\Gamma$ with positive sign & loop & loop \\
	\hline
\end{tabular} }
\caption{Loops, coloops and ordinary edges in a signed graph $\Sigma$ with underlying graph~$\Gamma$.} \label{fig.1}
\end{table}

\begin{figure}[ht!]
	\centering
\begin{tikzpicture}[scale=0.48,decoration={
	markings,
	mark=at position 0.5 with {\arrow{>}}}]
\tikzstyle{vertex}=[circle,fill=black!100,minimum size=7pt,inner sep=0pt]

\tikzstyle{vertex2}=[circle,fill=black!100,minimum size=6pt,inner sep=0pt]
%


\clip (-11.5,-6.5) rectangle (16, 5);	

\node[vertex2,label={\footnotesize{$v_1$}}] (v1) at (-10,2) {};
\node[vertex2,label={\footnotesize{$v_2$}}] (v2) at (-8,2) {};
\node[vertex2,label=270:{\footnotesize{$v_3$}}] (v3) at (-10,0) {};
\node[vertex2,label=270:{\footnotesize{$v_4$}}] (v4) at (-6,0) {};
\node[vertex2,label={\footnotesize{$v_5$}}] (v5) at (-4,2) {};
\node[vertex2,label={\footnotesize{$v_6$}}] (v6) at (-6,2) {};
\node[vertex2,label=270:{\footnotesize{$v_7$}}] (v7) at (-2,0) {};
\node[vertex2,label={\footnotesize{$v_8$}},label={[label distance=0.4cm]0:\footnotesize{$-$}}] (v8) at (2,2) {};
\node[vertex2,label={\footnotesize{$v_9$}},label={[label distance=0.4cm]90:\footnotesize{$+$}}] (v9) at (-2,2) {};

\draw[line width=2pt] (v1.center) -- (v3.center) node [midway,label={[label distance=-0.2cm]180:\footnotesize{$+$}}] (v13) {};

\draw[line width=2pt] (v1.center) -- (v2.center) node [midway,label={[label distance=-0.2cm]90:\footnotesize{$+$}}] (v12) {};
\draw[line width=2pt] (v2.center) -- (v3.center) node [midway,label={[label distance=-0.1cm]0:\footnotesize{$-$}}] (v23) {};
\draw[line width=2pt] (v3.center) -- (v4.center) node [midway,label={[label distance=-0.1cm]270:\footnotesize{$-$}}] (v34) {};
\draw[line width=2pt] (v4.center) -- (v6.center) node [midway,label={[label distance=-0.2cm]180:\footnotesize{$+$}}] (v46) {};
\draw[line width=2pt] (v5.center) -- (v6.center) node [midway,label={[label distance=-0.2cm]90:\footnotesize{$-$}}] (v56) {};
\draw[line width=2pt] (v4.center) -- (v5.center) node [midway,label={[label distance=-0.3cm]350:\footnotesize{$+$}}] (v45) {};
\draw[line width=2pt] (v5.center) -- (v7.center) node [midway,label={[label distance=-0.35cm]225:\footnotesize{$+$}}] (v57) {};
\draw[line width=2pt] (v7.center) -- (v8.center) node [midway,label={[label distance=-0.3cm]315:\footnotesize{$+$}}] (v78) {};
\draw[line width=2pt] (v8.center) -- (v9.center) node [midway,label={[label distance=-0.2cm]90:\footnotesize{$+$}}] (v89) {};
\draw[line width=2pt] (v7.center) -- (v9.center) node [midway,label={[label distance=-0.2cm]0:\footnotesize{$+$}}] (v79) {};

\draw[line width=2pt,dashed] (v8) to [out=45,in=315,looseness=30] (v8);

\draw[line width=2pt,dashed] (v9) to [out=135,in=45,looseness=30] (v9);

\node (tltext) [below right=0.5cm and -0.7cm of v3] [text width=6cm,align=left] {\footnotesize{$v_3v_4$ coloop/ ordinary \newline $v_5v_7$ coloop/ coloop. \newline 
$v_9v_9$ loop/ loop \newline $v_8v_8$ loop/ ordinary\newline other edges: ordinary/ ordinary\newline}};


\node[vertex2,label=270:{\footnotesize{$v_{10}$}}] (v10) at (5,0) {};
\node[vertex2,label=270:{\footnotesize{$v_{11}$}}] (v11) at (7,0) {};
\node[vertex2,label=90:{\footnotesize{$v_{12}$}}] (v12) at (7,2) {};

\draw[line width=2pt] (v10.center) -- (v11.center) node [midway,label={[label distance=-0.2cm]270:\footnotesize{$+$}}] (v1011) {};
\draw[line width=2pt] (v10.center) -- (v12.center) node [midway,label={[label distance=-0.1cm]180:\footnotesize{$-$}}] (v1012) {};
\draw[line width=2pt] (v11.center) -- (v12.center) node [midway,label={[label distance=-0.2cm]0:\footnotesize{$+$}}] (v1112) {};

\node (tltext) [below right=0.5cm and -1.5cm of v10] [text width=3.5cm,align=left] {\footnotesize{all edges: \newline ordinary/ coloop}};


\node[vertex2,label=270:{\footnotesize{$v_{13}$}}] (v13) at (11,0) {};
\node[vertex2,label=270:{\footnotesize{$v_{14}$}}] (v14) at (13,0) {};
\node[vertex2,label=90:{\footnotesize{$v_{15}$}},label={[label distance=0.4cm]90:\footnotesize{$-$}}] (v15) at (13,2) {};

\draw[line width=2pt] (v13.center) -- (v14.center) node [midway,label={[label distance=-0.2cm]270:\footnotesize{$+$}}] (v1314) {};
\draw[line width=2pt] (v13.center) -- (v15.center) node [midway,label={[label distance=-0.1cm]180:\footnotesize{$+$}}] (v1315) {};
\draw[line width=2pt] (v14.center) -- (v15.center) node [midway,label={[label distance=-0.2cm]0:\footnotesize{$+$}}] (v1415) {};
\draw[line width=2pt,dashed] (v15) to [out=135, in=45,looseness=30] (v15);

\node (tltext) [below right=0.5cm and -0.8cm of v13] [text width=3.5cm,align=left] {\footnotesize{$v_{15}v_{15}$: loop/ coloop}};

\end{tikzpicture}
	
	\caption{Signed graph edges of the seven types described in Table~\ref{table:summary} (edge type in $M(\Gamma)$/ edge type in $F(\Sigma)$).} \label{fig:M3}
\end{figure}
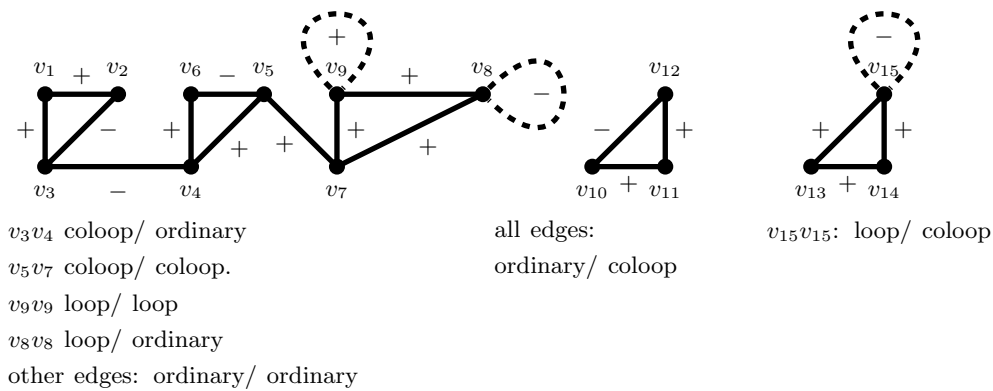




\section{Tutte polynomials for matroid pairs and for signed graphs}\label{section:ST}

\subsection{The Tutte polynomial of a graph and of a matroid}\label{section:ST matroid}

The Tutte polynomial of a graph $\Gamma=(V,E)$ has subset expansion
\begin{equation}\label{eq:Tutte_subset}
T_{\Gamma}(X,Y)= \sum_{A \subseteq E}(X-1)^{k(\Gamma\backslash A^c)-k(\Gamma)}(Y-1)^{|A|-|V|+k(\Gamma\backslash A^c)},\end{equation}
and may alternatively be defined by the recurrence
\begin{equation}\label{eq:Tutte_del_con}T_{\Gamma}(X,Y)=\begin{cases} T_{\Gamma/e}(X,Y)+T_{\Gamma\backslash e}(X,Y) & \mbox{if $e$ is an ordinary edge of $\Gamma$,}\\
X T_{\Gamma/e}(X,Y) & \mbox{if $e$ is a bridge of $\Gamma$,}\\
Y T_{\Gamma\backslash e}(X,Y) & \mbox{if $e$ is loop of $\Gamma$,}\end{cases}\end{equation}
and $T_\Gamma(X,Y)=1$ if $\Gamma$ has no edges. 
Among its many evaluations with combinatorial interpretations are the following, in which $\mathbb Z_n$ denotes the additive cyclic group on $n$ elements:
\begin{itemize}
\item $(-1)^{|V|-k(\Gamma)}n^{k(\Gamma)}T_{\Gamma}(1-n,0)$ is the number of proper vertex colorings of $\Gamma$ with $n$ colors,
\item  $(-1)^{|V|-k(\Gamma)}T_{\Gamma}(1-n,0)$ is the number of nowhere-zero $\mathbb Z_n$-tensions of $\Gamma$ (in one-to-$n^{k(\Gamma)}$ correspondence with proper $n$-colorings),
\item $(-1)^{|E|-|V|+k(\Gamma)}T_{\Gamma}(0,1-n)$ is the number of nowhere-zero $\mathbb Z_n$-flows of $\Gamma$.
\end{itemize}

The Tutte polynomial of a matroid $M=(E,r)$ with ground set $E$ and rank function $r$ is defined by
\begin{equation}
\label{eq:tutte_matroid}
T_{M}(X,Y)= \sum_{A \subseteq E}(X-1)^{r(E)-r(A)}(Y-1)^{|A|-r(A)}.
\end{equation}

When $M=M(\Gamma)=(E,r_M)$ is the cycle matroid of $\Gamma=(V,E)$ we have 
$r_M(E)-r_M(A)=k(\Gamma\backslash A^c)-k(\Gamma)$ and $|A|-r_M(A)=|A|-|V|+k(\Gamma\backslash A^c)$, so this subset expansion coincides with that given in~\eqref{eq:Tutte_subset} to define the Tutte polynomial of a graph. 
The Tutte polynomial of a matroid satisifies mutatis mutandis the same deletion-contraction recurrence given in~\eqref{eq:Tutte_del_con} for the Tutte polynomial of a graph.

\subsection{The trivariate Tutte polynomial of a signed graph} \label{sec:signed_tutte}

\begin{definition}\label{defn:signedtutte}
The \textit{trivariate Tutte polynomial} of a signed graph $\Sigma=(\Gamma,\sigma)$ with underlying graph $\Gamma=(V,E)$ is defined by
\begin{equation}\label{eq:signed_Tutte_subset}
T_{\Sigma}(X,Y,Z) := \sum_{A \subseteq E}(X-1)^{k(\Sigma\backslash A^c)-k(\Sigma)}(Y-1)^{|A|-|V|+k_{b}(\Sigma\backslash A^c)}(Z-1)^{k_{u}(\Sigma\backslash A^c)}.
\end{equation}
\end{definition}

It is easy to verify directly that $T_{\Sigma_1}=T_{\Sigma_2}$ when $\Sigma_1$ is switching equivalent to $\Sigma_2$, and that $T_{\Sigma_1\sqcup \Sigma_2}=T_{\Sigma_1}T_{\Sigma_2}$. 

The Tutte polynomial of a signed graph can, just as the Tutte polynomial of a graph extends to matroids, be extended to pairs of matroids on a common ground set, as explained in Appendix~\ref{app:matroids}. 
Given a signed graph $\Sigma=(\Gamma = (V,E),\sigma)$, let $M(\Gamma)=(E,r_M)$ be the cycle matroid of $\Gamma$ with rank function $r_M$, and let $F(\Sigma)=(E,r_F)$ be the frame matroid of $\Sigma$ with rank function $r_F$.
Since $r_F(A)-r_M(A)=k_u(\Sigma\backslash A^c)\geq 0$ for each $A\subseteq E$, the polynomial $S_{M(\Gamma), F(\Sigma)}$ of Definition~\ref{definition:extendedstp} is divisible by $(Z-1)^{r_M(E)}$, as can be seen by inspecting its subset expansion~\eqref{equation:sm1m2_2}.
We have then for signed graph $\Sigma$ with underlying graph $\Gamma=(V,E)$, 
\begin{equation}\label{def_tutte_from_mat}
T_{\Sigma}(X,Y,Z) = (Z-1)^{-r_M(E)}S_{M(\Gamma), F(\Sigma)}(X,Y,Z).
\end{equation}
%


Identities~\eqref{eq:Tutte_M1} and~\eqref{eq:Tutte_M2} in Appendix~\ref{app:matroids} then imply that the trivariate Tutte polynomial \eqref{eq:signed_Tutte_subset} contains as a specialization combinatorial invariants of a signed graph $\Sigma=(\Gamma,\sigma)$ that can be obtained as a specialization of the Tutte polynomial of the underlying frame matroid of $\Sigma$ or the Tutte polynomial of the cycle matroid of $\Gamma$. For example, the number of proper colorings of $\Sigma$ is an evaluation of the Tutte polynomial of $F(\Sigma)$, expressed as an evaluation of the trivariate Tutte polynomial of $\Sigma$ in Corollary~\ref{corollary:ncolors} below. More generally, the generating function for colorings of $\Sigma$ according to the number of improperly colored edges or, equivalently, the dichromatic polynomial $Q_\Sigma(u,v)$ defined by Zaslavsky~\cite[Section 3]{zas95}, is given by 
\begin{equation}\label{eq:Q} Q_\Sigma(u,v)=u^{k(\Sigma)}T_\Sigma\Big(u+1,v+1,\frac{1}{u}+1\Big).\end{equation}
(In a similar way to how the dichromatic polynomial of a graph $\Gamma$ is up to a prefactor equivalent to the Tutte polynomial of $M(\Gamma)$, the dichromatic polynomial of $\Sigma$ is up to a prefactor the Tutte polynomial of $F(\Sigma)$: equation~\eqref{eq:Q} is equation~\eqref{eq:Tutte_M2} in disguise. Zaslavsky's dichromatic polynomial is defined more widely for biased graphs, of which signed graphs form a subclass.)
By contrast, the number of nowhere-zero flows of $\Sigma$ in general depends on both $M(\Gamma)$ and $F(\Sigma)$ and is not given by an evaluation of the Tutte polynomial of either matroid: to enumerate nowhere-zero flows we require the trivariate Tutte polynomial (see Theorem~\ref{thm:nz_flows} below).  

If $\Sigma$ is balanced, the polynomial $T_{\Sigma}$ coincides with the Tutte polynomial of the underlying graph $\Gamma$ (as being balanced is hereditary, and so $k_u(\Sigma\backslash  A^c)\equiv0$).

\paragraph{Related work on Tutte polynomials for signed graphs}
The Tutte polynomial for signed graphs that we have just defined can be seen as a special case of the huge Tutte polynomial of weighted gain graphs of Forge and Zaslavsky~\cite{FZ}, taking all weights equal to $1$ and the gain group equal to $\{-1,1\}$.

In~\cite{kauf89} Kauffman defines a trivariate polynomial $Q(A,B,d)$ for signed graphs that for balanced signed graphs also reduces to the Tutte polynomial of the underlying graph. Godsil and Royle~\cite[Chapter 15]{GR01} define a signed rank polynomial for matroids on a signed ground set, specializing to the Kaufmann bracket for links. The polynomial of Definition~\ref{defn:signedtutte} differs from the polynomial of Kauffman since it is invariant under switchings, while the polynomial $Q(A,B,d)$ generally is not. For instance, for the graph $K_2$ in which the edge carries a positive sign the polynomial of Kauffman equals $A+Bd$, while for $K_2$ with a negative sign on the edge it equals $Ad+B$. These and other signed graph polynomials described in~\cite[Section 3.2]{Chm} are specializations of a signed version of the Bollob\'as-Riordan polynomial~\cite{bollobas99}, which is not invariant under switching.

The trivariate Tutte polynomial may be obtained as a specialization of the ``surface Tutte polynomial" $\T(M)$ of a map $M$ (graph embedded in a compact surface), introduced in~\cite{goodall2020tutte}, and whose expression can found in Appendix~\ref{app:dictionary}, 
 Table~\ref{tab.1}.  
An embedding of a graph as a map is commonly represented as a ribbon graph (see e.g.~\cite{EMM13}), edges being bands whose two ends are glued along the boundaries of disks representing vertices. With this representation, a sign can be associated with each of the edges of the graph with respect to this embedding, in which an edge receives positive sign when it is untwisted and negative sign if it is twisted.
The surface Tutte polynomial $\T(M)$ of a map $M$ contains $T_{\Sigma}$ as a specialization as follows.
For an arbitrary embedding of the signed graph $\Sigma$ into a compact surface (non-orientable precisely when $\Sigma$ is unbalanced, cf.~\cite{mohar01}) as a map $M$,
\begin{equation}\label{eq.signed_eval_surf}
T_{\Sigma}(X,Y,Z)=(X-1)^{-k(M)}\T(M;\x,\y),
\end{equation}
in which $\x$ and $\y$ are set equal to the following values: $x=1$, $y=Y-1$, $x_g=1$ for all $g\in \Z$, $y_g=X-1$ if $g\geq 0$ and $y_g=(X-1)(Z-1)/(Y-1)$ if $g\leq -1$.


 A variant of the polynomial $T_{\Sigma}$, in which the exponent $k(\Sigma\backslash A^c)-k(\Sigma)$ of $X-1$ for an $A \subseteq E$ in the subgraph expansion is replaced by $k_b(\Sigma\backslash A^c)-k_b(\Sigma)$, appears in the slides of a presentation by Krieger and O'Connor in 2013~\cite{kriocon13}. They show that a suitable renormalization of this polynomial equals the Euler characteristic of a chain complex of trigraded modules. This builds upon earlier work on the categorification of, in chronological order, the Jones polynomial by Khovanov~\cite{khovanov00}, the chromatic polynomial by Helme-Guizon and Rong~\cite{helmeguizon05} and the Tutte polynomial (for graphs) by Jasso-Hernandez and Rong~\cite{jassohernandez06}.

\section{Deletion-contraction invariants} \label{sec:del_cont_recipe}

To define contraction of negative edges in a signed graph requires enlarging the domain of signed graphs by allowing half-arcs and free loops~\cite{zas82}. To avoid doing this, in giving a recurrence for the trivariate Tutte polynomial we only allow contraction of positive edges.  
As a non-loop can always be made positive by vertex switching, this gives a recurrence terminating in signed graphs consisting solely of bouquets of negative loops.

\begin{theorem}\label{thm:S_del_con}
The trivariate Tutte polynomial $T_{\Sigma}=T_{\Sigma}(X,Y,Z)$ of a signed graph $\Sigma=(\Gamma,\sigma)$ with underlying graph $\Gamma=(V,E)$ satisfies, for a positive edge $e$,

\begin{numcases}
{T_{\Sigma}=} T_{\Sigma/e}+T_{\Sigma\backslash e} & \mbox{if $e$ is an ordinary edge of $\Gamma$,}\label{case.4}\\
T_{\Sigma/e}+(X\!-\!1)T_{\Sigma\backslash e} & \mbox{if $e$ is a bridge of $\Gamma$ and circuit path edge of $\Sigma$,}\label{case.2}\\
XT_{\Sigma/e} &
 \mbox{if $e$ is a bridge of $\Gamma$ not a circuit path edge of $\Sigma$,} \label{case.1}\\
YT_{\Sigma\backslash e} & \mbox{if $e$ is a loop of $\Gamma$ positive in $\Sigma$,}\label{case.3}
\end{numcases}

and if $\Sigma$ is the signed graph consisting of $\ell\geq 1$ negative loops on a single vertex then 
$$T_{\Sigma}=1+(Z-1)\big[1+Y+\cdots + Y^{\ell-1}\big],$$
and $T_{\Sigma}=1$ if $\Sigma$ has no edges. 
\end{theorem}

Theorem~\ref{thm:S_del_con} can be obtained as a corollary of the framework of~\cite{krajewski18} applied to the polynomial for matroid pairs, as deletion and contraction of positive edges (and deletion of negative edges) in $\Sigma=(\Gamma,\sigma)$ is compatible with deletion and contraction in the matroids $M(\Gamma)$ and $F(\Sigma)$ as usually defined (see \cite{oxley06}).  However, we give an independent proof of the recurrence formula for $T_{\Sigma}(X,Y,Z)$ as it is key to many of our results.

\begin{proof}[Proof of Theorem~\ref{thm:S_del_con}]
Given a signed graph $\Sigma=(\Gamma,\sigma)$ with underlying graph $\Gamma=(V,E)$ and an edge $e\in E$ with $\sigma(e)=+1$, we split the subset expansion of $T_{\Sigma}$ into two parts
\begin{align}\label{bridge1}
T_{\Sigma} = T_{\Sigma}' + T_{\Sigma}'',
\end{align}
according to whether a subset $A \subseteq E$ contains $e$, in which case the corresponding term is contained in $T_{\Sigma}'$, or does not contain $e$, in which case the corresponding term is contained in $T_{\Sigma}''$. Let $\Sigma/e$ have underlying graph $(V',E\backslash\{e\})$ and $\Sigma\backslash e$ underlying graph $(V,E\backslash  \{e\})$. We have a bijection $\{A \subseteq E\backslash \{e\}\} \rightarrow \{A \subseteq E : e \in A\}$ defined by $A \mapsto A \cup \{e\}$. 

For $A\subseteq E\backslash\{e\}$, 
\begin{equation}\label{3}
k_u((\Sigma/e)\backslash A^c) = k_u(\Sigma\backslash (A\cup\{e\})^c)\: \text{ and }\: k_b((\Sigma/e)\backslash A^c) = k_b(\Sigma\backslash (A\cup\{e\})^c),
\end{equation}
and 
\begin{equation}\label{6}
k_b((\Sigma\backslash e)\backslash A^c) =k_b(\Sigma\backslash A^c)\: \text{ and }\: k_u((\Sigma\backslash e)\backslash A^c) = k_u(\Sigma\backslash A^c).
\end{equation}
Thus
\begin{equation}
k((\Sigma/e)\backslash A^c)=k(\Sigma\backslash (A\cup\{e\})^c)\: \text{ and }\:
k((\Sigma\backslash e)\backslash A^c) =k(\Sigma\backslash A^c)
\end{equation}
and
\begin{equation}
k((\Sigma/e)\backslash A^c) - k(\Sigma/e) = k(\Sigma\backslash (A\cup\{e\})^c) - k(\Sigma),  \label{1}
\end{equation}
as $k(\Sigma/e)=k(\Sigma)$. Moreover, 
\begin{equation}
|A|-|V'| = \begin{cases} (|A\cup\{e\}|-1)-(|V|-1) = |A\cup\{e\}|-|V| \hspace{1.9mm} \text{ if }e \text{ is not a loop in $\Gamma$,}\\ |A\cup\{e\}|-|V|-1 \hspace{45mm} \text{ if }e \text{ is a loop in $\Gamma$,}\end{cases}\label{2}
\end{equation} 
where we have used that $\Gamma/e = \Gamma\backslash e$ when $e$ is a loop in $\Gamma$.

Also, \begin{align}
&k((\Sigma\backslash e)\backslash A^c) - k(\Sigma\backslash e) = \begin{cases} k(\Sigma\backslash A^c) - k(\Sigma) \hspace{10.6mm} \text{if }e \text{ is not a bridge in $\Gamma$,}\\k(\Sigma\backslash A^c) - k(\Sigma)-1 \hspace{4mm} \text{if }e \text{ is a bridge in $\Gamma$.}\end{cases}\label{4}
\end{align}

After these preparations, we may now prove each case of the recurrence, starting with
the case of ordinary edges~\eqref{case.4}. Take an edge $e$ not a loop or bridge such that $\sigma(e) = +1$. Then $T_{\Sigma}' = T_{\Sigma/e}$ by equations~\eqref{1}, \eqref{2} and~\eqref{3}. Equations~\eqref{4} and~\eqref{6} imply that $T_{\Sigma}'' = T_{\Sigma\backslash e}$, so that $T_{\Sigma} = T_{\Sigma/e}+ T_{\Sigma\backslash e}$. 

For case~\eqref{case.2}, if $e$ is a bridge of $\Gamma$ then $T_{\Sigma}' = T_{\Sigma/e}$ by equations $(\ref{1}), (\ref{2})$ and $(\ref{3})$, and $T_{\Sigma}'' = (X-1)\,T_{\Sigma\backslash e}$ by equations $(\ref{4})$ and $(\ref{6})$. Thus~\eqref{bridge1} in this case becomes $T_{\Sigma} = T_{\Sigma/e}+(X-1)\,T_{\Sigma\backslash e}$. 

Case \eqref{case.1} follows from~\eqref{case.2}, for if $e$ is a bridge but not a circuit path edge of $\Sigma$ then $T_{\Sigma/e}=T_{\Sigma\backslash e}$. Indeed, for $A\subseteq E\backslash \{e\}$, 
$k((\Sigma\backslash e) \backslash  A^c)=k((\Sigma/ e) \backslash A^c)+1$ (since $e$ is a bridge of $\Gamma$) 
and $|V|=|V'|+1$. Also,  for every $A\subseteq E\backslash\{e\}$, $k_b((\Sigma\backslash e) \backslash  A^c)=k_b((\Sigma/ e) \backslash A^c)+1$ since in $\Sigma\backslash e$ at least one endpoint of $e$ is contained in a balanced connected component of $\Sigma\backslash A^c$ (as $e$ is not a circuit path edge of $\Sigma$). 
This means that in $(\Sigma/ e)\backslash A^c$ there is one less balanced component than in $(\Sigma\backslash e)\backslash A^c$. 
The same reasoning establishes the equality $k_u((\Sigma\backslash e)\backslash A^c)=k_u((\Sigma/ e) \backslash A^c)$.

For case~\eqref{case.3}, assume now that $e$ is a loop with $\sigma(e) = +1$. Then $T_{\Sigma}' = (Y-1)\,T_{\Sigma/e}$ by equations~\eqref{1}, \eqref{2} and~\eqref{3}, and $T_{\Sigma}'' = T_{\Sigma\backslash e}$ by equations~\eqref{4} and~\eqref{6}. As $\Sigma\backslash  e=\Sigma/e$ for a positive loop $e$, we have $T_{\Sigma/e} = T_{\Sigma\backslash e}$ and subsituting into~\eqref{bridge1} yields $T_\Sigma=YT_{\Sigma\backslash e}$.

Finally, assume that $\Sigma$ is the one-vertex signed graph with $\ell \geq 1$ negative loops. The contribution of $A = \emptyset$ to $T_{\Sigma}$ is $1$. Let $A \subseteq E$ be a subset of size $i > 0$. Then 
\begin{displaymath}
(X-1)^{k(\Sigma\backslash A^c)-k(\Sigma)}(Y-1)^{|A|-|V|+k_b(\Sigma\backslash A^c)}(Z-1)^{k_u(\Sigma\backslash  A^c)} = (Y-1)^{i-1}(Z-1).
\end{displaymath}
Hence,
\begin{displaymath}
T_{\Sigma} = 1 + (Z-1)\cdot\sum_{i=1}^{\ell}\binom{\ell}{i}(Y-1)^{i-1} = 1 + (Z-1)\,(1+Y+ \ldots + Y^{\ell-1}).
\end{displaymath}
\end{proof}

A balanced signed graph $\Sigma$, switching equivalent to a signed graph with all edges positive, can be identified with its underlying graph $\Gamma$. In this case, the recurrence for $T_\Sigma$ in Theorem~\ref{thm:S_del_con} reduces to that of the Tutte polynomial of $\Gamma$ as there are no circuit path edges or negative loops. 

The Tutte polynomial $T_\Gamma(X,Y)$ of a graph as a polynomial in $X$ and $Y$ has non-negative coefficients (evident from its deletion-contraction recurrence, but not so evident from its subset expansion, which involves terms of the form $(X-1)^r(Y-1)^s$); furthermore, for a connected graph $\Gamma$, the coefficient of $X^iY^j$ in $T_\Gamma(X,Y)$ has a combinatorial interpretation as the number of spanning trees of internal activity $i$ and external activity $j$. 
The recurrence of Theorem~\ref{thm:S_del_con} indicates that 
as a polynomial in $X,Y$ and $Z$ the trivariate Tutte polynomial in general has negative coefficients, as confirmed by the following example.

\begin{example}\label{example:bridge}
Let $\Sigma$ be the signed graph on two vertices having a unique (positive) bridge $e$ and two negative loops, one at each vertex (see Figure~\ref{fig:M1}).
\begin{figure}[h!]
\centering
\begin{tikzpicture}[scale=0.8]
\SetGraphUnit{1.4}
\Vertex{a}
\EA(a){b}
\draw[style={-,thick,color=black}] (a)--(b)node[pos=0.5,anchor=south]{$e$};
\draw[style={-,thick,color=black}] (a) edge[in=135,out=225,loop,min distance=20mm] node[above,pos=0.37] {\vspace{10mm}\hspace{-9mm}$-$} (a);
\draw[style={-,thick,color=black}] (b) edge[in=315,out=45,loop,min distance=20mm] node[above,pos=0.63] {\vspace{5mm}\hspace{9mm}$-$} (b);
\end{tikzpicture}\caption{Unbalanced loose handcuff $\Sigma$} \label{fig:M1}
\end{figure}
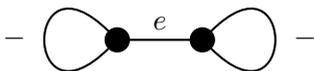
Then, taking $\ell=2$ in Theorem~\ref{thm:S_del_con}, we have $T_{\Sigma/e} = YZ+Z-Y$, while $T_{\Sigma\backslash e} = Z^2$ (by multiplicativity over disjoint unions and taking $\ell=1$ in Theorem~\ref{thm:S_del_con}). 
By the recurrence~\eqref{case.2} we have $T_\Sigma=(Y-Z)(Z-1)+XZ^2$. \end{example}

\subsection{The Recipe Theorem}
The Tutte polynomial is universal for deletion-contraction graph invariants in the sense that if $U$ is a graph invariant multiplicative over disjoint unions and satisfying 
$$U_{\Gamma}=\begin{cases} \alpha\, U_{\Gamma/e}+\beta U_{\Gamma\backslash e} & \mbox{if $e$ is an ordinary edge of $\Gamma$,}\\
x U_{\Gamma/e} & \mbox{if $e$ is a bridge of $\Gamma$,}
\end{cases}$$
and  $U_\Gamma=\gamma y^\ell$ if $\Gamma$ is a bouquet of $\ell\geq 0$ loops, 
then
$$U_\Gamma=\alpha^{r(\Gamma)}\beta^{|E|-r(\Gamma)}\gamma^{k(\Gamma)}T_\Gamma\left(\frac{x}{\alpha},\frac{y}{\beta}\right).$$
(See e.g.~\cite{bryl92}.) Despite the divisions by $\alpha$ and $\beta$, the formula for $U_\Gamma$ also holds for $\alpha=0$ and $\beta=0$ as the right-hand side upon expanding by the subset expansion for the Tutte polynomial is a polynomial in $x,y,\alpha$ and $\beta$.

 A similar ``Recipe Theorem" holds for the trivariate Tutte polynomial, which we shall apply to obtain combinatorial interpretations of its evaluations at various points in Sections~\ref{section:flows}--\ref{sec:tensions}. 

\begin{theorem}[Recipe Theorem]\label{theorem:recipe2}
	Let $R$ be an invariant of signed graphs invariant under switching and multiplicative over disjoint unions. 
	Suppose that there are constants $\alpha,\beta,\gamma,x,y$ and $z$, with $\gamma\neq 0$,  such that, for a signed graph $\Sigma=(\Gamma,\sigma)$ with underlying graph $\Gamma=(V,E)$ and positive edge $e\in E$,
$$
R_{\Sigma}=\begin{cases} 
\alpha\, R_{\Sigma/e}+\beta R_{\Sigma\backslash e} & \mbox{if $e$ is ordinary in $\Gamma$ and in $\Sigma$,}\\
\alpha\, R_{\Sigma/e}+\gamma R_{\Sigma\backslash e} & \mbox{if $e$ is ordinary in $\Gamma$ and  $k_u(\Sigma\backslash e)<k_u(\Sigma)$,}\\
\alpha R_{\Sigma/e}+\frac{\beta(x-\alpha)}{\gamma}R_{\Sigma\backslash e} & \mbox{if $e$ is a bridge in $\Gamma$ and a circuit path edge in $\Sigma$,}\\
x R_{\Sigma/e} & \mbox{if $e$ is a bridge in $\Gamma$ that is not a circuit path edge in $\Sigma$,}\\
y\, R_{\Sigma\backslash e} & \mbox{if $e$ is a loop in $\Gamma$ and in $\Sigma$,}\\	
\end{cases}
	$$
	while if $\Sigma$ is a bouquet of $\ell\geq 1$ negative loops then 
	$$R_{\Sigma}=\beta^{\ell-1}\gamma+(z-\gamma)\sum_{i=0}^{\ell-1}y^{\ell-1-i}\beta^i,$$
	and $R_{\Sigma}=1$ when $\Sigma$ is a single vertex with no edges.

	
	Then, 
	\begin{equation}\label{eq:R} R_{\Sigma}=\alpha^{r_M(E)}\beta^{|E|-r_F(E)}\gamma^{r_F(E)-r_M(E)}T_{\Sigma}\left(\frac{x}{\alpha},\frac{y}{\beta},\frac{z}{\gamma}\right),\end{equation}
a polynomial in $\alpha,\beta,x,y$ and $z$ over $\mathbb Z[\gamma,\gamma^{-1}]$. 

If $\alpha=0$ or $\beta=0$ then we use the subset expansion of the right-hand side of~\eqref{eq:R}:
\begin{align}
R_\Sigma=\sum_{A\subseteq E}&\alpha^{r_M(A)}\beta^{|E|-|A|+r_F(A)-r_F(E)}\gamma^{r_F(E)-r_F(A)-[r_M(E)-r_M(A)]}\cdot\nonumber \\
&\qquad \cdot(x-\alpha)^{r_M(E)-r_M(A)}(y-\beta)^{|A|-r_F(A)}(z-\gamma)^{r_F(A)-r_M(A)}.
\label{eq:alphabeta0}
	\end{align}

\end{theorem}


\begin{proof}
	Consider $R_{\Sigma}=a^{r_M(E)}b^{r_F(E)}c^{|E|}T_{\Sigma}(X,Y,Z).$ By the recurrence for $T_\Sigma$, we have
	$$
	R_{\Sigma}=\begin{cases} abc\, R_{\Sigma/e}+cR_{\Sigma\backslash e} & \mbox{if $e$ is ordinary in $\Gamma$ and in $\Sigma$}\\
	abc\, R_{\Sigma/e}+bcR_{\Sigma\backslash e} & \mbox{if $e$ is ordinary in $\Gamma$ and $k_u(\Sigma\backslash e)<k_u(\Sigma)$.}\\
ac\left[bR_{\Sigma/e}+(X-1)R_{\Sigma\backslash e}\right] & \mbox{if $e$ is a bridge in $\Gamma$ and circuit path edge of $\Sigma$,}\\
abcX\, R_{\Sigma/e} & \mbox{if $e$ is a bridge in $\Gamma$ that is not a circuit path edge in $\Sigma$,}\\
		cY\, R_{\Sigma\backslash e} & \mbox{if $e$ is a positive loop in $\Sigma$,}\\
	\end{cases}
	$$
with 
	\begin{displaymath}R_{\Sigma}=b c^{\ell} \left(1+(Z-1)\frac{Y^{\ell}-1}{Y-1}\right)\end{displaymath}
	for a bouquet with $\ell\geq 1$ negative loops and  $R_{\Sigma}=1$ for a single vertex edge-less graph.
	There is an additional case in the recurrence with respect to Theorem~\ref{thm:S_del_con} for an edge $e$ that is ordinary in $\Gamma$ and  $k_u(\Sigma\backslash e)<k_u(\Sigma)$: 
 the connected component containing $e$ is unbalanced in $\Sigma$ and balanced in $\Sigma\backslash e$.
	
	
	Introducing the parameters $x=abcX, y=cY, z=bcZ, \alpha=abc, \beta = c, \gamma=bc$ yields the~result.

It is clear from the recurrence formula that $R_\Sigma$ is a polynomial in $\alpha,\beta,x,y$ and $z$ over $\mathbb Z[\gamma,\gamma^{-1}]$. In the given subset expansion~\eqref{eq:alphabeta0}, all the exponents except possibly that of $\gamma$ are nonnegative: for $A\subseteq E$ we have $r_M(A)\geq 0$, $(|E|-r_F(E))-(|A|-r_F(A))\geq 0$, $r_M(E)-r_M(A)\geq 0$, $|A|-r_F(A)\geq 0$ and $r_F(A)-r_M(A)\geq 0$, while $r_F(E)-r_F(A)-[r_M(E)-r_M(A)]=k_u(\Sigma)-k_u(\Sigma\backslash A^c)$ may be negative (for example, when $\Sigma$ is an edge with a negative loop on either endpoint and $A$ comprises the two loops). 
\end{proof}

\section{Flows}\label{section:flows}

Flows on signed graphs taking values in an abelian group are defined in a similar way to flows on graphs. Given a graph $\Gamma=(V,E)$, we call a pair $(v,e)$ with $v\in V$ and $e\in E$ an edge containing $v$ a \emph{half-edge}. (A loop comprises two half-edges.)
A \emph{bidirected graph} is a pair $(\Gamma,\omega)$, where $\Gamma=(V,E)$ is a graph (not necessarily simple) in which every half-edge $(v,e)$ receives an orientation $\omega(v,e)\in \{-1,1\}$  (The two half-edges associated with a loop at a vertex consist of the same vertex-edge pair but receive orientations independently.) 
We call the orientation $\omega$ \emph{compatible} with the signature $\sigma$ of a signed graph $\Sigma=(\Gamma,\sigma)$ if for each edge $e=uv$ we have 
\begin{equation}\label{eq.two_half_orient_cond}
\sigma(e)=-\omega(u,e)\omega(v,e).
\end{equation} 
(In particular, if the sign of a loop is negative then its two half-edges receive the same orientation sign.) A half-edge $(v,e)$ oriented positively points into~$v$, and oriented negatively points out of~$v$; for an edge $e=uv$, when $\sigma(e)=+1$ the half-edges $(u,e)$ and $(v,e)$ are consistently directed, while if $\sigma(e)=-1$ they are oppositely directed.
A vertex switch at $v$ has the effect of changing 
the orientation of all the half-edges $(v,e)$ incident with $v$; thus the new orientation is compatible with the new signature.

Let $G$ be a finite additive abelian group. Considered as a $\mathbb Z$-module, for $x\in G$ we have $(+1)x=x$, $(-1)x=-x$ and $2x=x+x$. The subgroup $2G:=\{2x: x\in G\}$ will play a significant role in the sequel; the quotient group $G/2G$ of cosets $u+2G$ is isomorphic to the subgroup $\{x\in G:2x=0\}$ (in particular, $-u+2G=u+2G$). 
 
Flows on bidirected graphs were introduced by Bouchet~\cite{bouchet83}.
A \emph{$G$-flow of a bidirected graph} $(\Gamma = (V,E),\omega)$ is a function $f:E\to G$ such that at each vertex of $\Gamma$ the Kirchhoff law is satisfied, that is, for each vertex $v$, 
\begin{equation}\label{eq:bi flow}
\sum_{\substack{(v,e)\\v\in e}} \omega(v,e)f(e)=0,
\end{equation}
where the summation runs over half-edges $(v,e)$ incident with $v$, so if $e$ is a positive loop it contributes with two terms to the sum.
A \emph{$G$-flow of a signed graph}  $\Sigma=(\Gamma,\sigma)$ is a function $f: E\to G$ such that $f$ is a $G$-flow for the bidirected graph $(\Gamma,\omega)$, where $\omega$ is an orientation of $\Gamma$ compatible with $\sigma$. 
A $G$-flow is \emph{nowhere-zero} if $f(e)\neq 0$ for all $e\in E$.
We let $q^0_\Sigma(G)$ and $q_\Sigma(G)$ denote the number of $G$-flows and number of nowhere-zero $G$-flows of $\Sigma$, respectively. 

Given $e=uv$, and by considering $-f(e)$ instead of $f(e)$ as the value of a flow at $e$, we see that the number of nowhere-zero $G$-flows does not depend on the exact values of $\omega(v,e)$ and $\omega(u,e)$ but only on the value of their product. 
Hence the notion of nowhere-zero flow on a signed graph is well-defined and the number of nowhere-zero $G$-flows is an invariant of signed graphs as it does not depend on the choosen orientation. 
Furthermore, the number of  nowhere-zero $G$-flows is constant on equivalence classes of signed graphs, as switching a vertex $v$ reverses the orientation of those half-edges incident with $v$ and just replaces the left-hand side of equation~\eqref{eq:bi flow} with its negation. 
If a signed graph is balanced, then the number of (nowhere-zero) $G$-flows only depends on the size of the group $G$~\cite{tutte49}, but if the signed graph is unbalanced, then it also depends on the group structure of $G$, as described by the following theorem.

\begin{theorem}\label{theorem:G-flows}
	The number 
	of $G$-flows of a signed graph $\Sigma=(\Gamma = (V,E),\sigma)$ is equal to $|G|^{|E|-|V|+k_b(\Sigma)}\left(\frac{|G|}{|2G|}\right)^{k_u(\Sigma)}.$
\end{theorem}

\begin{proof}
The number of $G$-flows of $\Sigma$ satisfies the recurrence 
	$$q^0_\Sigma(G)=\begin{cases} q^0_{\Sigma/e}(G) & \text{if $e$ is not a loop in $\Gamma$,}\\
	|G|q_{\Sigma\backslash e}^0(G) & \text{if $e$ is a positive loop in $\Sigma$}, \end{cases}
	$$
since for a non-loop $e$ any given $G$-flow of $\Sigma/e$ by the defining equations~\eqref{eq:bi flow} extends to a $G$-flow of $\Sigma$ by the assignment of a unique value in $G$ to $e$, 
and for a positive loop $e$ any value in $G$ assigned to $e$ contributes zero to equation~\eqref{eq:bi flow}.  
	Furthermore, $q_\Sigma^0(G)=\frac{|G|}{|2G|}\cdot |G|^{\ell-1}$ for a bouquet of $\ell\geq 1$ loops negative in $\Sigma$, as
	\begin{align*}
	\sum_{g_1,\ldots,g_\ell\in G}{\bf 1}_{2g_1+\cdots +2g_\ell=0} &= \sum_{\substack{g \in G:\\ 2g = 0}}\sum_{g_1,\ldots,g_\ell\in G}{\bf 1}_{g_1+\cdots +g_\ell=g}
	=\frac{|G|}{|2G|}\sum_{g_1,\ldots,g_\ell\in G}{\bf 1}_{g_1+\cdots +g_\ell=0}
	=\frac{|G|}{|2G|}|G|^{\ell-1}.
	\end{align*}
	
	The deletion-contraction recurrence for $q_\Sigma^0(G)$ is not of the form given in Theorem~\ref{theorem:recipe2} (which would require taking $\gamma=0$ for it to fit). 
	However, here we can argue directly. Contraction of graph non-loops in $\Gamma$ leaves a graph consisting solely of $|E|-|V|+k(\Gamma)$ graph loops (the nullity of the graph $\Gamma$, the dimension of its cycle space). Among these loops, the positive loops in $\Sigma$ each contribute $|G|$; each negative loop in $\Sigma$ also contributes $|G|$, but a scale factor of $\frac{1}{|2G|}$ is applied for each bouquet containing a negative loop -- the latter correspond to unbalanced connected components of $\Sigma$.   Hence 
	\begin{align*}q_\Sigma^0(G)&=|G|^{|E|-|V|+k(\Gamma)}\left(\frac{1}{|2G|}\right)^{k_u(\Sigma)}
	= |G|^{|E|-|V|+k_b(\Sigma)}\left(\frac{|G|}{|2G|}\right)^{k_u(\Sigma)}\end{align*}
	using $k(\Gamma)=k(\Sigma)=k_b(\Sigma)+k_u(\Sigma)$. 
\end{proof}


Theorem~\ref{theorem:G-flows} and an application of inclusion-exclusion yields 
the following subset expansion formula for the number of nowhere-zero $G$-flows.

\begin{theorem}\label{cor:flowcount_signed}
Let $G$ be a finite additive abelian group and $\Sigma=(\Gamma, \sigma)$ a signed graph with underlying graph $\Gamma=(V,E)$. 
Then the number of nowhere-zero $G$-flows of $\Sigma$ is given by
\begin{equation}\label{equation:flowcountsigned}
q_\Sigma(G) = \sum_{A \subseteq E}(-1)^{|A^c|}|G|^{|A|-|V|+k_b(\Sigma\backslash A^c)}\left(\frac{|G|}{|2G|}\right)^{k_u(\Sigma\backslash A^c)}.
\end{equation}
\end{theorem}

Theorem~\ref{cor:flowcount_signed} is a special case of \cite[Theorem $4.6$]{goodall2020tutte}, namely \cite[Corollary $4.11$]{goodall2020tutte}. However, since we deal here with abelian groups, the technical difficulties involved in the proof of \cite[Theorem $4.6$]{goodall2020tutte} when general finite groups are involved can be avoided (as the proof of Theorem~\ref{theorem:G-flows} shows). Very recently, Theorem~\ref{cor:flowcount_signed} was found independently by Qian~\cite[Theorem~4.3]{qian18}.

Theorem~\ref{cor:flowcount_signed} can also be obtained as an evaluation of the trivariate Tutte polynomial by using the deletion-contraction recurrence established in~\cite{devos17} and applying Theorem~\ref{theorem:recipe2}, and then using the subset expansion of the trivariate Tutte polynomial as given in Definition~\ref{defn:signedtutte}. 

\begin{theorem}\label{thm:nz_flows}
Let $G$ be a finite additive abelian group. 
Then, for a signed graph $\Sigma=(\Gamma = (V,E), \sigma)$, the number of nowhere-zero $G$-flows of $\Sigma$ is given by
\begin{equation}\label{eq.flow-eval-3var}
q_\Sigma(G) = (-1)^{|E|-|V|+k(\Gamma)}T_{\Sigma}\left(0,1-|G|,1-\frac{|G|}{|2G|}\right).
\end{equation}

\end{theorem}
\begin{proof}
As shown in~\cite{devos17}, for a positive edge $e$ (after appropriate switching of vertices, an edge in a signed graph that is not a negative loop can be made positive),
$$ q_\Sigma(G)=
\begin{cases}

q_{\Sigma/e}(G)-q_{\Sigma\backslash e}(G) & \mbox{if $e$ is not a loop of $\Gamma$,}\\
(|G|-1)q_{\Sigma\backslash e}(G) & \mbox{if $e$ is a loop of $\Gamma$ positive in $\Sigma$.}
\end{cases}
$$
To see this note that for a positive loop $e$, any value $x\in G\backslash \{0\}$ assigned to $e$ contributes  $x-x=0$ to the sum~\eqref{eq:bi flow}, from which it follows in this case that $q_\Sigma(G)=(|G|-1)q_{\Sigma\backslash e}(G)$.
Otherwise, if $e$ is not a loop, then for a given nowhere-zero $G$-flow of $\Sigma/e$ there is by the defining equations~\eqref{eq:bi flow} for a flow a unique value we can assign to $e$ to extend this flow of $\Sigma/e$ to a flow of $\Sigma$. Those extensions that take the value $0$ on $e$ are precisely the nowhere-zero $G$-flows of $\Sigma\backslash e$. This establishes the recurrence. 

If $\Sigma$ is a bouquet of $\ell$ negative loops then, as shown in ~\cite{devos17} (simplifying by the binomial expansion the expression given in~\cite[Lemma 2.1]{devos17}),
$$q_{\Sigma}(G)=\frac{1}{|G|}\left[\frac{|G|}{|2G|}(|G|-1)^\ell+(-1)^\ell(|G|-\frac{|G|}{|2G|})\right].$$
The expression of $q_{\Sigma}(G)$ as an evaluation of $T_{\Sigma}$ now follows by taking 
$(x,y,z,\alpha,\beta,\gamma)=(0,|G|-1, \frac{|G|}{|2G|}-1,1, -1,-1)$ in Theorem~\ref{theorem:recipe2}.
\end{proof}

When $2G=G$, i.e. $G$ is of odd order, the number of nowhere-zero $G$-flows of $\Sigma$ given in Theorem~\ref{thm:nz_flows} is the evaluation $(-1)^{|E|-r_F(\Sigma)}T_{F(\Sigma)}(0,1-|G|)$ of the Tutte polynomial of the frame matroid $F(\Sigma)$. This is a consequence of the identity $$T_{F(\Sigma)}(X,Y)=(X-1)^{r(\Sigma)-r(\Gamma)}T_{\Sigma}\Big(X,Y,\frac{X}{X-1}\Big),$$ 
which follows from equation~\eqref{eq:Tutte_M2} and $T_{\Sigma}(X,Y,Z)=(Z-1)^{-r_M(E)}S_{M(\Gamma),F(\Sigma)}(X,Y,Z)$. When $G$ is of even order the number  of nowhere-zero $G$-flows of $\Sigma$ is not an evaluation of the Tutte polynomial of the frame matroid. 

\begin{example}\label{example:quasi}
For a graph $\Gamma=(V,E)$, the number of nowhere-zero $\mathbb Z_n$ flows of $\Gamma$ is equal to $(-1)^{|E|-|V|+k(\Gamma)}T_\Gamma(0,1-n)$, a polynomial in $n$. For a signed graph $\Sigma$, since $2{\mathbb Z_n} = \mathbb Z_n$ when $n$ is odd and $2{\mathbb Z_n} \cong \mathbb Z_{n/2}$ when $n$ is even, the number of nowhere-zero $\mathbb Z_n$-flows $q_\Sigma(\mathbb Z_n)$ is a quasipolynomial in $n$ of period~$2$.
\end{example}

\section{Colorings}\label{sec:colorings}

Zaslavsky~\cite{zas82} introduced a notion of signed graph colorings as follows. 

\begin{definition}\label{definition:ncoloring}
Let $n \geq 1$ be an integer. A  \textit{proper $n$-coloring} of $(\Gamma = (V,E),\sigma)$ is an assignment $f: V \rightarrow \{0,\pm1,...,\pm n\}$ such that for every edge $e = uv$ we have $f(u) \neq \sigma(e)f(v)$. A \textit{proper non-zero $n$-coloring} is a proper $n$-coloring which does not assign the value $0$ to any vertex.
\end{definition}

In~\cite{mrs16}, M{\'a}{\v{c}}ajov{\'a}, Raspaud and {\v{S}}koviera call an $n$-coloring in Zaslavsky's sense a $(2n+1)$-coloring, and call a non-zero $n$-coloring in Zaslavsky's sense a $2n$-coloring, with   
the advantage that the corresponding notion of chromatic number of a signed graph agrees with the (usual) chromatic number of a balanced signed graph (viewed as a graph). We will however use Zaslavsky's terminology.

Let $\Sigma=(\Gamma,\sigma)$ be a signed graph. For $n\geq 0$, define $\chi_{\Sigma}(2n+1)$ to be the number of proper $n$-colorings and $\chi_{\Sigma}^*(2n)$ to be the number of proper non-zero $n$-colorings of $\Sigma$. 
Zaslavsky showed that these are both polynomials in $n\in\mathbb N$,  and established subgraph expansions for them 
as follows.
\begin{theorem}[Theorem $2.4$ in~\cite{zas82a}]\label{thm:chrom_subset}
Let $\Sigma$ be a signed graph. Then 
\begin{displaymath}
\chi_{\Sigma}(t) = \sum_{A \subseteq E}(-1)^{|A|}t^{k_b(\Sigma\backslash A^c)} 
\end{displaymath}
and\begin{displaymath}
\chi_{\Sigma}^*(t) = \sum_{\substack{A \subseteq E:\\ \Sigma\backslash  A^c \:\text{\rm balanced}}}(-1)^{|A|}t^{k_b(\Sigma\backslash A^c)}.
\end{displaymath}
\end{theorem}

Zaslavsky further showed that these chromatic polynomials evaluated at negative integers have interpretations similar to the chromatic polynomial  of a graph evaluated at negative integers in terms of colorings and compatible orientation~\cite[Theorem 3.5]{zas82a}. 

Theorem~\ref{thm:chrom_subset} and the subgraph expansion of the trivariate Tutte polynomial (Definition~\ref{defn:signedtutte}) immediately yield an expression for the number of proper (non-zero) $n$-colorings as an evaluation of the trivariate Tutte polynomial.

\begin{corollary}\label{corollary:ncolors}
Let $\Sigma=(\Gamma,\sigma)$ be a signed graph. Then the number of proper $n$-colorings of $\Sigma$ is given by
\begin{equation}\label{equation:chromatic_poly}
\chi_{\Sigma}(2n\!+\!1) = (-1)^{|V|-k(\Gamma)}(2n\!+\!1)^{k(\Gamma)}T_{\Sigma}\left(-2n,0,\frac{2n}{2n+1}\right) 
\end{equation}
 and the number of proper non-zero $n$-colorings is given by
\begin{displaymath}
\chi_{\Sigma}^*(2n) = (-1)^{|V|-k(\Gamma)}(2n)^{k(\Gamma)}T_{\Sigma}(1-2n,0,1).
\end{displaymath}

\end{corollary}
The number of proper $n$-colorings is an evaluation of the Tutte polynomial of the frame matroid $F(\Sigma)$ at $(-2n,0)$, i.e. similarly to how the number of proper $(2n+1)$-colorings of a graph $\Gamma$ is the evaluation $(-1)^{r(\Gamma)}(2n+1)^{k(\Gamma)}T_\Gamma(-2n,0)$. See Corollary~\ref{thm:gcolors} 
below.
The number of proper non-zero $n$-colorings, however, is not an evaluation of the Tutte polynomial of $F(\Sigma)$. For example, while a balanced complete graph on three vertices and the unbalanced loose handcuff on two vertices (Figure~\ref{fig:M1}) have the same frame matroid (in either case, any pair of edges forms a basis), 
 a balanced complete graph has no non-zero $1$-colorings while the unbalanced loose handcuff has two. 

Zaslavsky's notion of proper (non-zero) $n$-colorings of signed graphs may be generalized as follows to colorings taking values in a finite set $X$ equipped with an involution $\iota$ on $X$.

\begin{definition}\label{definition:Xcoloring}
A \textit{proper $(X,\iota)$-coloring} of a signed graph $\Sigma$ with vertices $V$ is a map $f: V \rightarrow X$ such that, for an edge $e=uv$, we have $f(u) \neq f(v)$ if $e$ is positive and $\iota(f(u)) \neq f(v)$ if $e$ is negative.
\end{definition}

If $X$ is an additive abelian group of order $2n+1$ and if $\iota$ is the involution $\iota(x) = -x$, for $x \in X$, then the definition of a proper $(X,\iota)$-coloring is equivalent with Zaslavsky's definition of a proper $n$-coloring.\\
\indent Let $P_\Sigma(X,\iota)$ denote the number of proper $(X,\iota)$-colorings of $\Sigma$. 
The following theorem establishes that $P_\Sigma(X,\iota)$ is an evaluation of the trivariate Tutte polynomial similar in form to the specialization of the Tutte polynomial of a graph to the classical chromatic polynomial, $$\chi_\Gamma(|X|)=(-1)^{|V|-k(\Gamma)}|X|^{k(\Gamma)}T_{\Gamma}(1-|X|,0),$$ 
counting the number of proper vertex colorings of $\Gamma$ using a finite color set $X$. 
\begin{theorem}\label{theorem:Xcolorings}
For a signed graph $\Sigma = (\Gamma,\sigma)$ with underlying graph $\Gamma = (V,E)$, the number of proper $(X,\iota)$-colorings of $\Sigma$ is given by
\begin{equation}\label{equation:lambdag}
P_{\Sigma}(X,\iota)= (-1)^{|V|-k(\Sigma)}|X|^{k(\Sigma)}T_{\Sigma}\left(1-|X|,0,1-\frac{t}{|X|}\right),
\end{equation}
where $t = |\{x : \iota(x) = x\}|$.
\end{theorem}
\begin{proof}
For a map $f: V \rightarrow X$, define $I_{\Sigma}(f) \subseteq E$ (the set of impropriety of $f$, to use Zaslavsky's term) by
\begin{displaymath}
I_{\Sigma}(f):= \{e = uv : \sigma(e) = 1, f(u) = f(v)\} \cup \{e = uv : \sigma(e) = -1,  \iota(f(u)) = f(v)\}.
\end{displaymath}
We also define
\begin{equation}\label{equation:defgamma}
i(\Sigma) := |\{f: V \rightarrow X : I_{\Sigma}(f) =E\}|, 
\end{equation}
the number of colorings of $\Sigma$ improper on every edge. 
Then 
\begin{equation}\label{equation:gammatel}
i(\Sigma) = t^{k_u(\Sigma)}|X|^{k_b(\Sigma)}.
\end{equation}
To see why equation~(\ref{equation:gammatel}) holds, we may assume that $\Sigma$ is connected. A map $f: V \rightarrow X$ for which $I_{\Sigma}(f) = E$ then is uniquely determined by the value it assigns to a fixed vertex of $\Sigma$. If $\Sigma$ is balanced, there are $|X|$ choices for this value. If $\Sigma$ is unbalanced, then the presence of an unbalanced cycle forces this value $x$ to satisfy $\iota(x) = x$, yielding $t$ choices for $x$.\\
\indent From equation~(\ref{equation:defgamma}) it follows that for $A \subseteq E$ we have 
\begin{displaymath}
i(\Sigma\backslash A^c) = |\{f: V \rightarrow X : I_{\Sigma}(f) \supseteq A\}|.
\end{displaymath}
Using inclusion-exclusion (in the third equality below), we then calculate that
\begin{align*}
P_{\Sigma}(X,\iota) &= |\{f: V \rightarrow X : I_{\Sigma}(f) = \emptyset\}| = \sum_{f: V \rightarrow X}\sum_{A \subseteq I_{\Sigma}(f)}(-1)^{|A|}\\
&= \sum_{A \subseteq E}(-1)^{|A|}i(\Sigma\backslash A^c) = \sum_{A \subseteq E}(-1)^{|A|}t^{k_u(\Sigma\backslash A^c)}|X|^{k_b(\Sigma\backslash A^c)}.
\end{align*}
The results follows upon substituting $X=1-|X|, Y=0$ and $Z=1-\frac{t}{|X|}$ into the subset expansion formula~\eqref{eq:signed_Tutte_subset} defining $T_{\Sigma}(X,Y,Z)$. 
\end{proof}

Alternatively,  Theorem~\ref{theorem:Xcolorings} can be shown by establishing that the number of proper $(X,\iota)$-colorings satisfies the deletion-contraction of Theorem~\ref{theorem:recipe2} with $(x,y,z,\alpha,\beta,\gamma)=(|X|-1,0,1-\frac{t}{|X|},-1,1,1)$; an additional prefactor of $|X|^{k(\Sigma)}$ arises as the number of proper $(X,\iota)$-colorings for a single vertex is $|X|$. 

The focus of Section~\ref{sec:tensions} will be {\em proper $G$-colorings} for a finite additive abelian group $G$, by which we mean proper $(X,\iota)$-colorings in which $X$ is the set of elements of $G$ and $\iota:x\mapsto -x$ is negation (additive inverse). The number of fixed points of $\iota$ in this case is $\frac{|G|}{|2G|}$. 
 Let $P_\Sigma(G)$ denote the number of proper $G$-colorings.  
By Theorem~\ref{theorem:Xcolorings} 
we then have the following:
\begin{corollary}\label{thm:gcolors}
Let $G$ be a finite additive abelian group. 
Then, for a signed graph $\Sigma=(\Gamma = (V,E),\sigma)$, the number of proper $G$-colorings is given by
\begin{displaymath}
P_\Sigma(G) = (-1)^{|V|-k(\Sigma)}|G|^{k(\Sigma)}T_{\Sigma}\left(1-|G|,0,1-\frac{1}{|2G|}\right).
\end{displaymath}
\end{corollary}

\section{Tensions and  potential differences}\label{sec:tensions}

For a finite additive abelian group $G$, 
there is a correspondence between $G$-tensions of a graph $\Gamma$ and colorings of the vertices of $\Gamma$ by elements of $G$: for each such coloring, taking the difference between endpoint colors (the order the difference is taken according to a fixed orientation of $\Gamma$) yields a unique $G$-tension, which is nowhere-zero precisely when the coloring is proper.  Conversely, to each $G$-tension there correspond $|G|^{k(\Gamma)}$ vertex colorings of $\Gamma$.  Therefore tensions of a graph coincide with ``potential differences" of vertex colorings (see e.g.~\cite{biggs97}). 

An analogous notion of a $G$-tension~\cite{chen09} exists for signed graphs; only, unlike for graphs, not every $G$-tension arises from a vertex coloring. 
For signed graphs, then, the notion of a potential difference 
does not coincide with the notion of a tension established in~\cite{chen09}: potential differences of a $G$-coloring of an unbalanced signed graph form a proper subset of $G$-tensions unless $G$ is of odd order.   

We shall define tensions in a different but equivalent way to~\cite{chen09}, and introduce $G$-potential differences as $G$-tensions with an added constraint (the relationship between $G$-potential differences to $G$-colorings will emerge in  Section~\ref{sec:ten_col}). Before giving a formal definition, we require some preliminaries.
 
A walk of a signed graph $\Sigma=(\Gamma,\sigma)$ with underlying graph $\Gamma=(V,E)$, written as a vertex-edge sequence
$$W=(v_1,e_1,v_2,e_2,\dots v_k,e_k,v_{k+1}),$$
is said to be \emph{positive} in $\Sigma=(\Gamma,\sigma)$ if $\prod_{i=1}^k \sigma(e_i)=1$ and \emph{negative} otherwise.  
The walk $W$ is a \emph{closed walk} if $v_{k+1}=v_1$. 
Recall from Section~\ref{sec:frame_matroid} that a {\em circuit} of a signed graph $\Sigma=(\Gamma,\sigma)$ is an edge set forming a balanced cycle, or an edge set forming two unbalanced cycles sharing exactly one common vertex (an {\em unbalanced tight handcuff}), or an edge set forming two vertex-disjoint unbalanced cycles joined by a simple path meeting the cycles exactly in its endpoints (an {\em unbalanced loose handcuff}). In the last case, the edges of the path are called {\em circuit path edges}; otherwise edges of a circuit belong to a cycle of the graph $\Gamma$.
Circuits naturally give rise to positive closed walks in the signed graph with circuit paths edges being used twice and other edges once. We call such  closed walks \emph{circuit walks}.


\begin{definition}\label{def:tensions_2}
	Let $\Sigma=(\Gamma,\sigma)$ be a signed graph, let $\omega$ be an orientation compatible with $\sigma$ and let $G$ be a finite additive abelian group. A map	$f:E\to G$ is a \emph{$G$-tension} of $\Sigma$ with respect to the orientation $\omega$  if and only if, for each circuit walk $W=(v_1,e_1,v_2,e_2,\ldots,v_k,e_k,v_1)$,
	\begin{equation}\label{eq:def_ten_2}
	\sum_{i=1}^k \left(-\omega(v_i,e_i)\prod_{j=1}^{i-1} \sigma(e_j)\right) f(e_i) =0.
	\end{equation}
	
	The map $f$ is said to be a \emph{$G$-potential difference} of $\Sigma$ if and only if $f$ is a $G$-tension such that, for every walk $W=(v_1,e_1,v_2,e_2,\ldots,v_k,e_k,v_1)$ around an unbalanced cycle,
	\begin{equation}\label{eq:def_ten_3}
	\sum_{i=1}^k  f(e_i) \in 2G.
	\end{equation}
\end{definition}

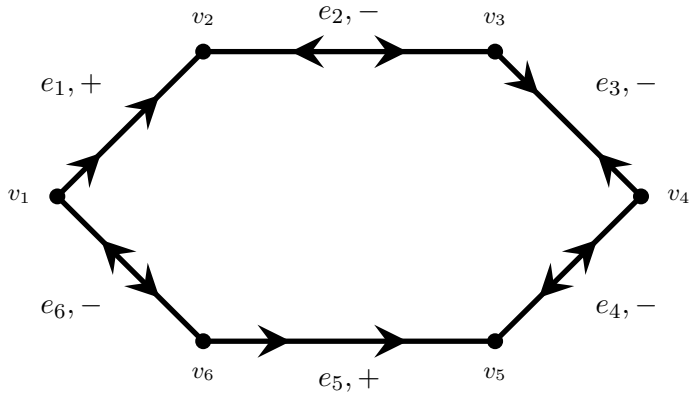
\begin{figure}[h!]
	\centering
	\begin{tikzpicture}[scale=0.48]
	\tikzstyle{vertex}=[circle,fill=black!100,minimum size=7pt,inner sep=0pt]
	
	\tikzstyle{vertex2}=[circle,fill=black!100,minimum size=6pt,inner sep=0pt]
	%
	
	
	\clip (-12,-5.5) rectangle (20, 8);

	\node[vertex2,label={[label distance=0.1cm]180:\footnotesize{$v_1$}}] (v1) at (-8,0) {};
	\node[vertex2,label={[label distance=0.1cm]90:\footnotesize{$v_2$}}] (v2) at (-4,4) {};
	\node[vertex2,label={[label distance=0.1cm]90:\footnotesize{$v_3$}}] (v3) at (4,4) {};
	\node[vertex2,label={[label distance=0.1cm]0:\footnotesize{$v_4$}}] (v4) at (8,0) {};
	\node[vertex2,label={[label distance=0.1cm]270:\footnotesize{$v_5$}}] (v5) at (4,-4) {};
	\node[vertex2,label={[label distance=0.1cm]270:\footnotesize{$v_6$}}] (v6) at (-4,-4) {};

	\draw[line width=2pt,postaction={decorate},
		decoration={markings,mark=at position 0.3 with {\arrow[line width=1.5mm]{stealth}}},
	decoration={markings,mark=at position 0.7 with {\arrow[line width=1.5mm]{stealth }}}] 
	(v1.center) -- 
	node [midway,above left=0.2cm and 0.2cm] {$e_1,+$} 
	(v2.center);

	\draw[line width=2pt,postaction={decorate},
	decoration={markings,mark=at position 0.3 with {\arrowreversed[line width=1.5mm]{stealth}}},
	decoration={markings,mark=at position 0.7 with {\arrow[line width=1.5mm]{stealth}}}] 
	(v2.center) -- 
	node [midway,above=0.2cm] {$e_2,-$} 
	(v3.center);
	
	\draw[line width=2pt,postaction={decorate},
	decoration={markings,mark=at position 0.3 with {\arrow[line width=1.5mm]{stealth}}},
	decoration={markings,mark=at position 0.7 with {\arrowreversed[line width=1.5mm]{stealth}}}] 
	(v3.center) -- 
	node [midway,above right=0.2cm and 0.2cm] {$e_3,-$} 
	(v4.center);
	
		\draw[line width=2pt,postaction={decorate},
	decoration={markings,mark=at position 0.3 with {\arrowreversed[line width=1.5mm]{stealth}}},
	decoration={markings,mark=at position 0.7 with {\arrow[line width=1.5mm]{stealth}}}] 
	(v4.center) -- 
	node [midway,below right=0.2cm and 0.2cm] {$e_4,-$} 
	(v5.center);
	
	\draw[line width=2pt,postaction={decorate},
	decoration={markings,mark=at position 0.3 with {\arrowreversed[line width=1.5mm]{stealth}}},
	decoration={markings,mark=at position 0.7 with {\arrowreversed[line width=1.5mm]{stealth}}}] 
	(v5.center) -- 
	node [midway,below=0.2cm] {$e_5,+$} 
	(v6.center);
	
	\draw[line width=2pt,postaction={decorate},
	decoration={markings,mark=at position 0.3 with {\arrowreversed[line width=1.5mm]{stealth}}},
	decoration={markings,mark=at position 0.7 with {\arrow[line width=1.5mm]{stealth}}}] 
	(v6.center) -- 
	node [midway,below left=0.2cm and 0.2cm] {$e_6,-$} 
	(v1.center);
	
	\node[] (label) at (6,6) {\footnotesize{$W=(v_1,e_1,v_2,e_2,v_3,e_3,v_4,e_4,v_5,e_5,v_6,e_6,v_1)$}};
	\node[] (label) at (4,7) {\footnotesize{$
			-(-1)(+1)f(e_1)-
			(+1)(+1)f(e_2)-
			(-1)(-1)f(e_3)-
			(+1)(+1)f(e_4)-
			(+1)(-1)f(e_5)-
			(+1)(-1)f(e_6)=0$}};

	\end{tikzpicture}
	
	\caption{Tension equation \eqref{def:tensions_2} for the circuit walk $W$. The first sign in brackets corresponds to $\omega(v_i,e_i)$ (negative if the half-edge is goes out of and positive if it goes into $v_i$), and the second sign in brackets corresponds to the product of signs on the previous edges.  Alternatively, the signing $-\omega(v_i,e_i)\prod_{j=1}^{i-1}\sigma(e_j)$ is $+1$ if the half-edge $(v_i,e_i)$ has the same orientation as the walk (from $v_i$ to $v_{i+1}$, the first half-edge oriented outwards), and $-1$ otherwise. If the graph is balanced and all edges are positively signed, then the coefficient $-\omega(v_i,e_i)\prod_{j=1}^{i-1}\sigma(e_j)$ is the usual sign for a graph tension: $+1$ if the walk traverses the edge in the same direction as the edge orientation, and $-1$ if the walk traverses the edge in the opposite direction.} \label{fig:def_ten}
\end{figure}

 See Figure~\ref{fig:def_ten} for an example of \eqref{eq:def_ten_2}.
When $G$ has odd order, equation~\eqref{eq:def_ten_3} is always satisfied, so that $G$-potential differences coincide with $G$-tensions in this case.
\begin{remark}
	Our definition of tension is equivalent to the definition of tensions in~\cite[(4.4)]{chen09}, as $[\omega,\omega_W](e_i)$ (in the notation of~\cite{chen09}) is equal to $-\omega(v_i,e_i)\prod_{j=1}^{i-1} \sigma(e_j)$ when the orientation of the walk $W$ has the half-edge $(v_1,e_1)$ going out of  $v_1$. See Appendix~\ref{sec:app_fig_lemma} for more details. 
	
\end{remark}

\begin{remark}\label{rmk:wd}
Although Definition~\ref{def:tensions_2} depends on the choice of orientation $\omega$ of $\Sigma=(\Gamma,\sigma)$ compatible with $\sigma$ and involves a choice of starting vertex for the walk $W$, the number of $G$-tensions of $\Sigma$ is independent of both choices:  changing the direction of an edge $e$ in $\omega$ corresponds to negating the value of each $G$-tension on $e$; starting the walk at a different vertex, say $v_i$, has the effect of multiplying equation \eqref{eq:def_ten_2} by $\prod_{j=1}^{i-1} \sigma(e_j)^{-1}$. 
\end{remark}

Switching at a vertex $v$ in $\Sigma=(\Gamma, \sigma)$ flips the sign function $\sigma$ on edges incident with $v$, making an equivalent signed graph $\Sigma'=(\Gamma,\sigma')$. Orientations $\omega$ of $\Sigma$ compatible with $\sigma$ are transformed into orientations $\omega'$ of $\Sigma'$ compatible with $\sigma'$ by switching the sign of $\omega$ on each half-edge $(v,e)$ incident with $v$.  The proof of the following proposition consists in showing that the equations~\eqref{eq:def_ten_2} for circuit walks $W$ that define a tension of $\Sigma$ are preserved under vertex switching. 

\begin{proposition}\label{prop:tension_switch}
The set of $G$-tensions and the set of $G$-potential differences of a signed graph are invariant under switching. 
\end{proposition}

\begin{proof} We just need to show that the set of $G$-tensions is invariant under switching, as the condition~\eqref{eq:def_ten_3} for a $G$-potential difference is independent of edge signing and orientation. 

The effect on $\omega$ of switching at a vertex $v$ is to multiply its value on each half-edge incident with $v$ by $-1$; the signs of edges incident with $v$ are likewise flipped (except for loops on $v$: both half-edges of a loop at $v$ are multiplied by $-1$, thus preserving the sign of the loop).

By Remark~\ref{rmk:wd}, it suffices to consider circuit walks $W=(v,e_1,v_2,e_2,\ldots, v_k,e_{k},v)$ starting and finishing at a given vertex~$v$. 
The coefficient $-\omega(v,e_1)$ of $f(e_1)$ in equation~\eqref{eq:def_ten_2} has its sign flipped. We now describe the effect on the coefficients of $f(e_i)$ for $i\in\{2,\dots, k\}$.
As $W$ is a circuit walk, there are just two cases to consider: there is no $i\in\{2,\ldots, k\}$ such that $v=v_i$, or there is exactly one $j\in\{2,\ldots, k\}$ such that $v=v_j$.

Consider the first case where $v_i\neq v$ for  $i\in\{2,\ldots, k\}$, and the walk $W$ does not simply traverse a loop on $v$. Then, for each $i$, the coefficient $-\omega(v_i,e_i)\linebreak[0]\sigma(e_1)\linebreak[0]\cdots\linebreak[0]\sigma(e_{i-1})$  of $f(e_i)$ in equation~\eqref{eq:def_ten_2} has its sign flipped, as $\sigma(e_1)$ is the only sign that is changed (along with $\sigma(e_k)$, which does not feature).
Therefore, after switching at $v$, equation~\eqref{eq:def_ten_2} is replaced by its negation, thus preserving the solutions for $f(e_i)$ to the equation. When there is a loop $e$ on $v$, the only possibility in this case is the walk $W=(v,e,v)$, for which the coefficient of $f(e)$ in equation~\eqref{eq:def_ten_2} is equal to $-\omega(v,e)$, which is flipped in sign under switching at~$v$. 

Consider now the second case where $v_j=v$ for exactly one  $j\in\{2,\ldots, k\}$, i.e., $W=(v,\linebreak[0]e_1,\linebreak[0]\ldots,\linebreak[0] e_{j-1},\linebreak[0]v,\linebreak[0]e_j,\linebreak[0]\ldots,\linebreak[0] e_k,\linebreak[0]v)$. First we assume $j\in\{3,\dots, k-1\}$, i.e. neither $e_1$ nor $e_k$ are loops on $v$. 
The same argument as for the first case shows that the sign of the coefficient of $f(e_i)$ is flipped for each $i\in\{1,\dots, j-1\}$. The coefficient $-\omega(v,e_j)\linebreak[0]\sigma(e_1)\linebreak[0]\cdots\linebreak[0]\sigma(e_{j-1})$  of $f(e_j)$ has its sign flipped, as $\omega(v,e_j)$, $\sigma(e_1)$ and $\sigma(e_{j-1})$ are negated and the remaining signs are preserved. 
For $i\in\{j+1,\dots, k\}$ the coefficient $-\omega(v_i,e_i)\linebreak[0]\sigma(e_1)\linebreak[0]\cdots\linebreak[0]\sigma(e_{i-1})$ of $f(e_i)$ in equation~\eqref{eq:def_ten_2} has its sign flipped as $\sigma(e_1)$, $\sigma(e_{j-1})$ and $\sigma(e_j)$ are negated and the remaining signs are preserved. Again, the overall effect is to negate the whole of  equation~\eqref{eq:def_ten_2}.

When there is a loop $e$ on $v$, we may assume $W=(v,e,v,e_2, \dots, v_k, e_k,v)$, where $e_2=e_k$ is not a loop if $k>2$  (so the underlying circuit is an unbalanced loose handcuff), while if $k=2$ then $W=(v,e,v,e_2,v)$ where $e_2$ is another loop on $v$ (so the underlying circuit is an unbalanced tight handcuff). For the case of a loose handcuff, the coefficient of $f(e)$ is $-\omega(v,e)$, which is flipped in sign under switching at $v$. The coefficient of $f(e_2)$ is $-\omega(v,e_2)\sigma(e)$, in which $\omega(v,e_2)$ has signed flipped under switching at $v$ (while the sign of $\sigma(e)$ is preserved since $e$ is a loop). 
For $i\in\{3,\dots, k\}$, the coefficient of $f(e_i)$ is $-\omega(v_i,e_i)\sigma(e)\sigma(e_2)\cdots \sigma(e_{i-1})$, which is flipped in sign as $\sigma(e_2)$ has its sign flipped. 
For the case of the tight handcuff $W=(v,e,v,e_2,v)$, both coefficients $-\omega(v,e_1)$ and $-\omega(v,e_2)\sigma(e_1)$ are flipped in sign (the sign of $\sigma(e)$ is preserved while both $\omega(v,e)$ and $\omega(v,e_2)$ have flipped signs). Hence again the overall effect is to negate the whole of equation~\eqref{eq:def_ten_2}.
\end{proof}

Proposition~\ref{prop:tension_switch} implies that if $\Sigma$ is balanced, and hence switching equivalent to an all-positive signed graph, then a $G$-tension of $\Sigma$ corresponds exactly with a $G$-tension of its underlying graph $\Gamma$  (see, for instance, ``global $G$-tension'' in the terminology of~\cite{goodall16,litjens18}). 
It follows that if $\Sigma$ is balanced then every $G$-tension of $\Sigma$ is a $G$-potential difference.

Tensions are equivalently defined by stipulating that the net sum around any positive closed walk (not just a circuit walk) is zero, as expressed by the following proposition. 
\begin{proposition} \label{prop.technical_circuits}
With the same notation as in Definition~\ref{def:tensions_2}. 
\begin{itemize}
	\item The map $f:E\to G$ is a $G$-tension if and only if \eqref{eq:def_ten_2} is satisfied for every positive closed walk $W$. 
	\item The map $f:E\to G$ is a $G$-potential difference if and only if $f$ is a tension and 
 \eqref{eq:def_ten_3} is satisfied for every negative closed walk $W$.
\end{itemize} 
\end{proposition}
	
\begin{proof}
Let $G$ be a finite abelian group, $\Sigma=(\Gamma,\sigma)$ a signed graph, $\omega$ an orientation compatible with $\sigma$, and $f$ a $G$-tension of $\Sigma$, i.e. satisfying equation~\eqref{eq:def_ten_2} for each circuit walk $W$.  
In order to show that equation~\eqref{eq:def_ten_2} is satisfied for every positive closed walk~$W$, we argue by contradiction, showing that any counterexample must contain a subwalk of smaller nonzero length that is also a counterexample, or implies that a circuit walk is also a counterexample. As neither option is possible, we conclude that no counterexample exists. 
 
We introduce some terminology and notation needed in the proof.	
The {\em concatenation} of walks $W=(u, X, v)$ and $Z =(v,Y,w)$ is defined by 
$W\ast Z=(u,X,v,Y,w)$. A closed walk $W=(v_1,e_1,\dots, v_k,e_k,v_1)$ can be expressed as a concatenation $X*Y$ of two walks $X$ and $Y$ 
in which $X$ starts at an arbitrary vertex of the walk $W$, and $Y$ starts at the last vertex of $X$ and finishes at the first vertex of $X$.  
Given a walk $W=(v_1,e_1,v_2,e_2,\ldots v_k,e_k,v_{k+1})$, its {\em length} is $k$, and its {\em interior vertices} are  the vertices $v_2,\dots, v_k$ (which may include the first or last vertex if revisited by the walk). A walk is \emph{nontrivial} if it has non-zero length, and a subwalk of a walk $W$ is {\em proper} is it is neither trivial nor the whole walk $W$. 
Let $\sigma(W):=\prod_{i=1}^k\sigma(e_i)$, and, for $i=1,\ldots,k+1$, let $W^{(i)}=(v_1,e_1,\ldots,e_{i-1},v_{i})$. 
The \emph{height} of $f$ on $W$ is then defined as
$$
f(W)=\sum_{i=1}^k -\sigma(W^{(i)})\omega(v_i,e_i)f(e_i).
$$
Definition~\ref{def:tensions_2} says that $f$ is a $G$-tension if and only if $f(W)=0$ for all circuit walks $W$. Our goal is to show that $f(W)=0$ for any positive closed walk. 

We begin by collecting some properties of the height function that we shall use.
For $W=W_1*W_2*\cdots *W_\ell$, a concatenation of walks,
\begin{equation}\label{eq:concat}
f(W_1*W_2*\cdots*W_\ell)=\sum_{i=1}^\ell \left(\prod_{j=1}^{\ell-1} \sigma(W_j)\right)f(W_i).
\end{equation}
If $W=X*Y$ is a positive closed walk then we have $\sigma(X)=\sigma(Y)$, and thus
\begin{equation}\label{eq:reorder}
f(X*Y)=f(X)+\sigma(X)f(Y)=\sigma(X)(f(Y)+\sigma(Y)f(X))=\sigma(X)f(Y*X).
\end{equation}
We may therefore cyclically permute a positive closed walk $W$ to bring any of its subwalks as its initial subwalk without affecting the property that $f(W)=0$. (This extends the observation made in Remark~\ref{rmk:wd} that the defining equations~\eqref{eq:def_ten_2} for a $G$-tension do not depend on the starting vertex chosen for each circuit walk $W$.)
The {\em reverse} walk $\overline{W}$ of a walk $W=(v_1,\linebreak[0]e_1,\linebreak[0]v_2,\linebreak[0]e_2,\linebreak[0]\dots,\linebreak[0] v_k,\linebreak[0]e_k,\linebreak[0]v_{k+1})$ is the walk $\overline{W}=(v_{k+1},\linebreak[0]e_k,\linebreak[0]v_k,\linebreak[0]e_{k-1},\linebreak[0]v_{k-1},\linebreak[0]\dots,\linebreak[0] e_1,\linebreak[0]v_1).$ 	
We have $\sigma(W)=\sigma(\overline{W})$ and more generally $\sigma(\overline{W}^{(i)})=\sigma(W)\sigma(W^{(k+2-i)})$ for $i=1,\ldots,k+1$.
We observe that
\begin{align}
	f(\overline{W})&=\sum_{i=1}^k - \sigma(W)\sigma(W^{(k+2-i)})\cdot(-\sigma(e_{k+1-i}))\omega(v_{k+1-i},e_{k+1-i})\cdot f(e_{k+1-i})\nonumber
	\\
	&=-\sigma(W)\sum_{i=1}^k - \sigma(W^{(k+1-i)})\omega(v_{k+1-i},e_{k+1-i})f(e_{k+1-i})\nonumber
	\\
	&=-\sigma(W)f(W),\label{eq:reverse}
	\end{align}
where in the first line we use that $\omega(v_{k+2-i},e_{k+1-i)})=-\sigma(e_{k+1-i})\omega(v_{k+1-i},e_{k+1-i})$ by compatibility of $\omega$ with~$\sigma$. 
As a consequence, for walks $X,Y$ and $Z$ for which the concatenation $X*Z*\overline{Z}*Y$ is defined, 
\begin{equation}\label{eq:goback}
f(X*Z*\overline{Z}*Y)=f(X*Y).
\end{equation}
Indeed, using \eqref{eq:concat} and \eqref{eq:reverse},
\begin{align*} f(X*Z*\overline{Z}*Y)&=f(X)+\sigma(X)f(Z)+\sigma(X*Z)f(\overline{Z})+\sigma(X*Z*\overline{Z})f(Y)\\
& = f(X)+\sigma(X)f(Z)-\sigma(X)\sigma(Z)^2f(Z)+\sigma(X)f(Y)\\
& = f(X)+\sigma(X)f(Y) = f(X*Y).
\end{align*}


Assume now that $W$ is a positive closed walk of minimum nonzero length such that $f(W)\neq 0$. 
We establish a series of claims which together imply that $W$ cannot exist. 

If $W=X*Y$ for nontrivial closed walks $X$ and $Y$ then both $X$ and $Y$ must be negative for otherwise $0\neq f(W)=f(X)+\sigma(X)f(Y)$ by~\eqref{eq:concat} and one of $X$ and $Y$ is of smaller length than $W$. 
Thus,
\begin{equation}\label{eq:no pos subwalk}
\text{No proper closed subwalk of $W$ is positive.}
\end{equation}

By the number of times a vertex $v$ occurs in a closed walk $(v_1,e_1,v_2,\ldots,e_k,v_{k+1})$ we mean the number of indices $i\leq k$ such that $v_i=v$. 

\begin{equation}\label{eq:vertex twice}
\text{No vertex occurs more than twice in $W$.}
\end{equation}
To see this, suppose $W=X*Y*Z$ for nontrivial closed walks $X,Y,Z$ each starting and ending at vertex~$v$. (By~\eqref{eq:reorder} we may assume $W$ starts at~$v$.) As $W$ is positive, not all of $X,Y,Z$ can be negative. By applying~\eqref{eq:reorder} if necessary, we may take $X$ to be positive. By~\eqref{eq:concat} we have $f(W)=f(X)+f(Y*Z)$.
Since $f(W)\neq 0$, either $f(X)\neq 0$ or $f(Y*Z)\neq 0$. But $X$ and $Y*Z$ are positive closed walks of smaller length than $W$. 
This contradicts minimality of $W$ and establishes~\eqref{eq:vertex twice}. 


We say that a nontrivial subwalk $Z$ of closed walk $W$ appears more than once in $W$ if $W=Z*X*Z*Y$ (starting $W$ at the first vertex of $Z$ in one of its appearances as a subwalk). 
\begin{equation}\label{eq:no subwalk}
\text{No nontrivial subwalk appears more than once in $W$}.
\end{equation}
Suppose to the contrary that 
$W=Z*X*Z*Y$. 
The walk $Z*X$ is closed as the last vertex of $X$ is the first vertex of $Z$, and hence by~\eqref{eq:no pos subwalk} cannot be positive. 
Since $W$ is closed and positive, the walk $X*\overline{Y}$  is closed and positive.
By~\eqref{eq:concat} we have, since $\sigma(Z*X)=-1$,
\[
f(W)=f(Z)+\sigma(Z)f(X)-f(Z)+\sigma(Z*X*Z)f(Y)=\sigma(Z)(f(X)+\sigma(X)f(\overline{Y})),
\]
where the last equality follows from~\eqref{eq:reverse}.
This implies that $f(X*\overline{Y})\neq 0$. But $X*\overline{Y}$ is a positive closed walk of smaller length than $W$, contradicting minimality of $W$. This establishes~\eqref{eq:no subwalk}.
\begin{equation}\label{eq:no interior}
\text{No interior vertex of a proper closed subwalk of $W$ occurs elsewhere in $W$.}
\end{equation}
Suppose on the contrary that $W=C*X*Y$ for proper closed subwalk $C$ and subwalks $X,Y$ such that $v$ is an interior vertex of $C$ that appears again as the last vertex of $X$ and first vertex of $Y$.  
Let $C=A*B$ where $A$ ends in $v$ and $B$ begins with $v$ ($A$ and $B$ are nontrivial as $v$ is an interior vertex of $C$). 
By~\eqref{eq:no pos subwalk}, $C$ must be negative, and then $X*Y$ is also negative. Thus $A$ and $B$ have opposite signs and $X$ and $Y$ have opposite signs. Since $B*X$ and $Y*A$ are closed subwalks of $W=A*B*X*Y$ too, they are negative by~\eqref{eq:no pos subwalk}. This forces $\sigma(A)=\sigma(X)=-\sigma(B)=-\sigma(Y)$. 
Since $A*\overline{X}$ and $B*\overline{Y}$ are positive closed walks of smaller length than $W$ we have 
$$0=f(A*\overline{X})=f(A)+\sigma(A)f(\overline{X})=f(A)-\sigma(A)\sigma(X)f(X)=f(A)-f(X),$$
and similarly $f(B)-f(Y)=0$. 
But, using $\sigma(A)\sigma(B)=-1$ and $\sigma(A)=\sigma(X)$, 
$$0\neq f(W)=f(A)+\sigma(A)f(B)-f(X)-\sigma(X)f(Y)=f(A)-f(X)+\sigma(A)[f(B)-f(Y)]=0,$$
a contradiction. This establishes~\eqref{eq:no interior}.

\begin{equation}\label{eq:basic}
\text{$W$ contains at most two proper closed subwalks}.
\end{equation}
To see this, suppose on the contrary that $W=A*X*B*Y*C*Z$ for nontrivial closed subwalks $A,B,C$ and walks $X,Y,Z$ (we may assume $W$ begins with one of its closed subwalks by~\eqref{eq:reorder}, and by~\eqref{eq:no interior} two proper closed subwalks share no interior vertices). 
By minimality of $W$, each of $A, B$ and $C$ must be negative. By using~\eqref{eq:reorder} if necessary, we may assume $X$ has minimum length among $X,Y,Z$. 
Let $W_1=A*X*B*\overline{X}$ and $W_2=X*Y*C*Z$. The walk $W_1$ is closed (since $X$ starts at the end of $A$ and $\overline{X}$ ends at the start of $A$), positive (since $\sigma(A)=-1=\sigma(B)$ and $\sigma(X)=\sigma(\overline{X})$), and shorter than $W$ (since $\overline{X}$ has strictly smaller length than $Y*C*Z$ by assumption on $X$). Thus $W_2$ is also closed, positive, and of smaller length than $W_1$. 
By~\eqref{eq:goback} and~\eqref{eq:concat}, $f(W)=f(W_1*W_2)=f(W_1)+f(W_2)$, which leads to the contradiction that one of the shorter closed positive walks $W_1$ and $W_2$ has nonzero height. This establishes~\eqref{eq:basic}.  



Properties~\eqref{eq:no pos subwalk}, \eqref{eq:vertex twice}, \eqref{eq:no subwalk}, \eqref{eq:no interior} and \eqref{eq:basic} together imply that $W=A*X*B*Y$ for simple closed negative walks $A$ and $B$ and paths $X$ and $Y$ such that $X*Y$ is a positive closed walk, and, by minimality of $W$, $f(X*Y)=f(X)+\sigma(X)f(Y)=0$. 
We have 
\begin{align*}0\neq f(W)=f(A*X*B*Y)&=f(A*X*B)+\sigma(X)f(Y)\\
&=f(A*X*B)-f(X)\\
&=f(A*X*B)+\sigma(\overline{X})f(\overline{X})=f(A*X*B*\overline{X}),\end{align*}
and since $A*X*B*\overline{X}$ is a circuit walk, we must then by~\eqref{eq:def_ten_2} have $f(W)=0$. This final contradiction 
establishes that there is no counterexample $W$ to the first part of the proposition. 

To prove the second part of the proposition characterizing $G$-potential differences, we begin with the observation that for an arbitrary walk $W=(v_1,e_1,v_2,e_2,\dots, v_k,e_k,v_{k+1})$ the height
$f(W)$ belongs to the same coset of $2G$ as $\sum_{i=1}^kf(e_i)$, since  $-u+2G=u+2G$ for any $u\in G$.
Thus the constraint~\eqref{eq:def_ten_3} defining a $G$-potential difference is equivalent to $f(W)\in 2G$ for every walk $W$ around an unbalanced cycle.
A negative closed walk $Z$ must pass through a vertex $v$ belonging to a walk $W$ traversing an unbalanced cycle. (By vertex switching we may assume all edges not in cycles traversed by $Z$ are positive; then there must be a negative cycle that $Z$ meets.) 
Making $v$ the start and end vertex of $Z$ (permitted by~\eqref{eq:reorder}), the walk $Z*W$ is positive and hence $f(Z)-f(W)=0$. By~\eqref{eq:def_ten_3}, $f(W)\in 2G$, whence $f(Z)\in 2G$, and the result follows.  
\end{proof}

The set of $G$-tensions of a connected signed graph can be partitioned into $\frac{|G|}{|2G|}$ subsets according to the height of walks traversing unbalanced cycles modulo the subgroup $2G$ (zero for $G$-potential differences as per equation~\eqref{eq:def_ten_3}): 
\begin{proposition}\label{fact.1}
	
	If $\Sigma$ is a connected unbalanced signed graph and $f$ is a $G$-tension, then 
there is $u\in G$ such that for every closed walk $W=(v_1,e_1,\ldots,v_k,e_k,v_1)$ around an unbalanced cycle of~$\Sigma$
	\begin{displaymath}
	\sum_{i=1}^k f(e_i)\in u+2G.
	\end{displaymath}
\end{proposition}
%

\begin{proof}
We use notation introduced at the beginning of the proof of Proposition~\ref{prop.technical_circuits}. If $X=(v_1,e_1,\dots, v_k,e_k,v_1)$ and $X'=(v_1',e_1',\dots, v_\ell',e_\ell',v_{1}')$ are walks traversing two unbalanced cycles in the same connected component of $\Sigma$ and $Y$ a path from the end $v_1$ of $X$ to the start $v_1'$ of $X'$, then the concatenation $W=X*Y*X'*\overline{Y}$ is a positive closed walk. By Proposition~\ref{prop.technical_circuits}, the $G$-tension $f$ then satisfies $f(W)=0$ (equation~\eqref{eq:def_ten_2}).
	Using~\eqref{eq:reverse}
\begin{align*}0=f(W)&=f(X)-f(Y)-\sigma(Y)f(X')+\sigma(Y)f(\overline{Y})\\
& =f(X)-\sigma(Y)f(X')-2f(Y),\end{align*}
which, as $-u+2G=u+2G$ for any $u\in G$,  implies
$$\sum_{i=1}^kf(e_i)-\sum_{j=1}^\ell f(e_j')\in 2G.$$ 
This in turn implies $\sum_{i=1}^kf(e_i)$ belongs to the same coset of $2G$ as $\sum_{j=1}^\ell f(e_j')$.
\end{proof}

\subsection{Enumeration of (nowhere-zero) $G$-tensions and (nowhere-zero) $G$-potential differences}\label{sec:num_tensions_pot_diff}

We let $t_\Sigma^0(G)$, $t_\Sigma(G)$, $p_\Sigma^0(G)$ and  $p_\Sigma(G)$ denote the number of $G$-tensions, nowhere-zero $G$-tensions, $G$-potential differences and nowhere-zero $G$-potential differences, respectively.

Much as a $G$-tension of a connected graph $\Gamma$ is uniquely determined by its values on a spanning tree of $\Gamma$, so that there are clearly $|G|^{|V|-k(\Gamma)}$ $G$-tensions of $\Gamma$, an analogous statement holds for $G$-tensions and $G$-potential differences of a connected signed graph, proved as Theorem~\ref{theorem:tensionscircuit} below, with the notion of a connected basis of a signed graph replacing that of spanning tree of a graph.


\begin{definition}Let $\Sigma=(\Gamma,\sigma)$ be a connected signed graph. 
	A {\em connected basis} of $\Sigma$ is a spanning tree of $\Gamma$ when $\Sigma$ is balanced, and, when $\Sigma$ is unbalanced, the union of a spanning tree $T$ of $\Gamma$ and an edge $e$ such that the unique cycle in the subgraph $T \cup \{e\}$ is unbalanced in $\Sigma$.
		A \emph{connected basis} of the disjoint union of connected signed graphs is the union of connected bases of the connected components.  
\end{definition}

For a balanced connected signed graph, all bases (maximal subsets of edges containing no circuit) are connected so the qualifier ``connected" in ``connected basis" is redundant in this case. However, for unbalanced signed graphs, maximal subsets of edges containing no circuit may be disconnected. In fact, any forest of connected bases of unbalanced induced subgraphs of $\Sigma$ that together cover all the vertices of $\Sigma$ forms a basis as the addition of any edge forms either a balanced cycle or a handcuff. 

\begin{lemma}\label{lemma:specialbasis_2}
	A $G$-tension of a connected signed graph $\Sigma$ is uniquely determined by its values on the edges of a connected basis (any choice of values on the basis determines a unique tension on the whole of $\Sigma$).
\end{lemma}

Lemma~\ref{lemma:specialbasis_2} can be proven using results in~\cite{chen09}; we discuss this connection and give an alternative proof in Appendix~\ref{sec:app_fig_lemma}.

\begin{theorem}\label{theorem:tensionscircuit}
The number 
of $G$-tensions of a signed graph $\Sigma=(\Gamma = (V,E),\sigma)$ is given by $$t_\Sigma^0(G)=|G|^{|V|-k_b(\Sigma)}.$$ The number of $G$-potential differences is given by $$p_\Sigma^0(G)=|G|^{|V|-k(\Sigma)}|2G|^{k_u(\Sigma)}.$$
\end{theorem}

\begin{proof}
By multiplicativity of the number of $G$-tensions and $G$-potential differences of a signed graph over disjoint unions, it suffices to prove that for a connected signed graph $\Sigma$ on vertex set $V$ the number of $G$-tensions is given by
\begin{displaymath}t_\Sigma^0(G)=\left\{\begin{array}{ll}|G|^{|V|-1} & \text{if $\Sigma$ is balanced,}\\|G|^{|V|} & \text{if $\Sigma$ is unbalanced}\end{array}\right.
\end{displaymath} 
and the number of $G$-potential differences is given by
\begin{displaymath}p_\Sigma^0(G)=\left\{\begin{array}{ll}|G|^{|V|-1} & \text{if $\Sigma$ is balanced,}\\|G|^{|V|-1}|2G| & \text{if $\Sigma$ is unbalanced}.\end{array}\right.
\end{displaymath} 
By Lemma~\ref{lemma:specialbasis_2}, there is a bijection between $G$-tensions of $\Sigma$ and $G$-valued functions on a connected basis. If a connected signed graph $\Sigma=(\Gamma,\sigma)$ with underlying graph $\Gamma=(V,E)$ is balanced, then a basis has $|V|-1$ edges, while if it is unbalanced, then a basis has $|V|$ edges. The additional coset condition implies the $|2G|$ factor in the $G$-potential differences instead of the $|G|$ factor for the $G$-tensions since a basis contains an unbalanced cycle $X$. 
While Lemma~\ref{lemma:specialbasis_2} does not directly give that this tension will automatically be a potential difference, it does follow from it. Indeed if there is another unbalanced cycle $X'$ in $\Sigma$ then we can traverse $X$ and $X'$ by walks and appeal to Proposition~\ref{fact.1} to show that the sum of the tension values on the two unbalanced cycles belongs to $2G$. 
\end{proof}

Although Theorem~\ref{theorem:tensionscircuit} gives a formula for the number of $G$-tensions, the usual route of inclusion-exclusion in order to obtain a subset expansion for the number of nowhere-zero $G$-tensions is not available to us as contraction of negative edges has not been defined for signed graphs.
\footnote{Recall that for a graph $\Gamma=(V,E)$ the number of $G$-tensions of  $\Gamma$ is equal to $|G|^{r(\Gamma)}$ so that by inclusion-exclusion the number of nowhere-zero $G$-tensions of $\Gamma$ is 
$$\sum_{A\subseteq E}(-1)^{|A|}|G|^{r(\Gamma/A)}=(-1)^{r(\Gamma)}T_\Gamma(1-|G|,0).$$
Here we use the fact that $r(\Gamma/A)=|A|-r(\Gamma\backslash A^c)$ and that a tension of $\Gamma/A$ corresponds to a unique tension of $\Gamma$ that is zero on $A$ (dual to the fact that a flow of $\Gamma\backslash A$ corresponds to a unique flow of $\Gamma$ that is zero on $A$).} 
For nowhere $G$-flows of signed graphs we were able to use inclusion-exclusion -- see Theorem~\ref{theorem:G-flows}; for proper $G$-colorings too, we relied on the inclusion-exclusion formula used to derive Theorem~\ref{theorem:Xcolorings} to obtain the number of proper $G$-colorings~\ref{thm:gcolors} as an evaluation of the trivariate Tutte polynomial. (Once we have explained the connection between colorings and tensions, we shall see that this in fact gives the number of nowhere-zero $G$-potential differences -- see Corollary~\ref{cor:num_c_tens} below.)

In order to enumerate nowhere-zero $G$-tensions we shall establish a deletion-contraction recurrence, but in order to do so we need to partition tensions into classes. 
Let
 $p_{\Sigma}(G;u)$ denote the number of nowhere-zero $G$-tensions such that for every walk $W=(v_1,e_1,\ldots,e_k,v_1)$ around an unbalanced cycle of $\Sigma$ 
\begin{displaymath}
	\sum_{i=1}^k f(e_i)\in u+2G.
	\end{displaymath}
By Proposition~\ref{fact.1}, for a connected signed graph $\Sigma$ each  nowhere-zero $G$-tension falls into one of these classes, $p_{\Sigma}(G;u)=p_\Sigma(G)$ when $u\in 2G$, while $\sum_{u\in G}p_{\Sigma}(G;u)=|2G|t_\Sigma(G)$ (restricting the range of $u$ in the sum to a transversal of cosets of $2G$ gives $t_\Sigma(G)$).

\begin{theorem}\label{theorem:nz-tensions}
Let $\Sigma=(\Gamma,\sigma)$ be a signed graph and $G$ a finite additive abelian group. 
Then the number of nowhere-zero $G$-potential differences of $\Sigma$ is given by
\begin{equation}
\label{equation:potential_differences}
p_\Sigma(G)=
(-1)^{r(\Gamma)}|2G|^{k_u(\Sigma)}T_{\Sigma}\Big(1-|G|,0,1-\frac{1}{|2G|}\Big),
\end{equation}
and, for $u\not\in 2G$, 
$$p_\Sigma(G;u)=(-1)^{r(\Gamma)}|2G|^{k_u(\Sigma)}T_{\Sigma}(1-|G|,0,1). 
$$

In particular, 
when $\Sigma$ is connected and unbalanced, the number of nowhere-zero $G$-tensions of $\Sigma$ is given by
$$t_\Sigma(G)=(-1)^{r(\Gamma)}|2G|\left[T_{\Sigma}\Big(1-|G|,0,1-\frac{1}{|2G|}\Big)+\left(\frac{|G|}{|2G|}-1\right)T_{\Sigma}(1-|G|,0,1)\right].$$
\end{theorem}

\begin{remark} When $\Sigma$ is balanced, $G$-tensions of $\Sigma$ are $G$-potential differences of $\Sigma$ and correspond to $G$-tensions of $\Gamma$; when $\Sigma$ is balanced $T_{\Sigma}(X,Y,Z)=T_{\Gamma}(X,Y)$ when $Z\neq 1$ and $T_{\Sigma}(X,Y,1)=0$ (as can be verified from the subset expansion for the trivariate Tutte polynomial).  So $t_\Sigma(G)=p_\Sigma(G)=(-1)^{r(\Gamma)}T_{\Gamma}(1-|G|,0)$ when $\Sigma$ is balanced.
	\end{remark} 

\begin{proof}
	We claim that for each $u\in G$ the invariant $p_{\Sigma}(G;u)$ satisfies the following deletion-contraction recurrence: For positive edge~$e$,
\begin{displaymath}
p_{\Sigma}(G;u)=
\left\{
\begin{array}{lp{6cm}}
p_{\Sigma\backslash e}(G;u)- p_{\Sigma/ e}(G;u) & \parbox[t]{0.45\textwidth}{if $e$ is ordinary in $\Gamma$ and in $\Sigma$,}\\
|2G| p_{\Sigma\backslash e}(G;u) - p_{\Sigma/ e}(G;u)  & \parbox[t]{0.6\textwidth}{if $e$ is ordinary in $\Gamma$ and  $k_u(\Sigma\backslash e)<k_u(\Sigma)$,}\\
\frac{|G|}{|2G|}p_{\Sigma\backslash e}(G;u) - p_{\Sigma/ e}(G;u)
& \parbox[t]{0.6\textwidth}{if $e$ is a bridge in $\Gamma$ and a circuit path edge in $\Sigma$,}\\
(|G|-1) p_{\Sigma\backslash e}(G;u) & \parbox[t]{0.5\textwidth}{if $e$ is a bridge in $\Gamma$ that is not a circuit path edge in $\Sigma$,}\\
0  & \parbox[t]{0.6\textwidth}{if $e$ is a loop in $\Gamma$ and in $\Sigma$ (positive loop),}\\
\end{array}\right.
\end{displaymath}
and for a vertex with $\ell\geq 1$ negative loops we have
\begin{displaymath}
p_{\Sigma}(G;u)|=\begin{cases}
|2G|-1 & \text{ $u\in 2G$}\\
|2G| & \text{ $u\not\in  2G$}.\\
\end{cases}
\end{displaymath}
Then the result will follow by Theorem~\ref{theorem:recipe2} with $(\alpha,\beta,\gamma,x,y,z)=(-1,1,|2G|,|G|-1,0,|2G|(-1))$.

It suffices to consider connected $\Sigma$.
We take the cases in turn, letting $f':E\backslash\{e\}\to G\backslash 0$ be a nowhere-zero $G$-tension of $\Sigma\backslash e$ with sum in $u+2G$ on unbalanced cycles, and counting how many ways $f'$ can be extended to a nowhere-zero $G$-tension $f:E\to G\backslash 0$ of $\Sigma$ also with sum in $u+2G$ on unbalanced cycles.  

When $e$ is ordinary in $\Gamma$ and in $\Sigma$, the edge $e$ either appears in some balanced cycle $Y$ (a circuit of both $\Sigma$ and $\Gamma$) or in an unbalanced cycle of a handcuff $Y'$ (the cycle is a circuit of $\Sigma$ and the handcuff, loose or tight, is a circuit of $\Gamma$). The value of $f(e)$ is by Lemma~\ref{lemma:specialbasis_2} uniquely determined by the partial tension $f'$ on $\Sigma\backslash  e$. 
Those tensions $f$ for which $f(e)=0$ correspond to tensions $f'':E\backslash\{e\}\to G$ of $\Sigma /e$. 
Hence $p_\Sigma(G;u)=p_{\Sigma\backslash e}(G;u)-p_{\Sigma/e}(G;u)$ in this case. 

When $e$ is ordinary in $\Gamma$ and  $k_u(\Sigma\backslash e)<k_u(\Sigma)$ it belongs to each unbalanced cycle of $\Sigma$ (of which there is at least one). For such an unbalanced cycle $X$, the condition $\sum_{e'\in X}f(e')\in u+2G$ determines $|2G|$ possible values of $f(e)$; choices $f(e)=0$ correspond to tensions of $\Sigma/e$. These conditions are compatible for all unbalanced cycles as two such cycles together form a positive closed walk.
Since there are no unbalanced cycles in $\Sigma\backslash e$, the only circuits of $\Sigma$ are balanced cycles. The positive edge $e$ cannot both belong to an unbalanced cycle and to a balanced cycle, for otherwise there would be an unbalanced cycle not containing $e$. Therefore there are no further constraints on $f(e)$ beyond the sum of values on the cycle $X$ being in $u+2G$.  Subtracting for when the choice $f(e)=0$ is available (corresponding to a tension of $\Sigma/e$), we thus obtain $p_{\Sigma}(G;u)=|2G|p_{\Sigma\backslash e}(G;u)-p_{\Sigma/e}(G;u)$ in this case.

When $e$ is a bridge of $\Gamma$ and circuit path edge, then there is a circuit $Y$ of $\Sigma$ containing $e$ and equation~\eqref{eq:def_ten_2} for a circuit walk around $Y$ determines $2f(e)$, and -- this is where the condition that $f$ summed on each unbalanced cycle belongs to a fixed coset $u+2G$ is required -- this makes $\frac{|G|}{|2G|}$ choices for $f(e)$. (If two unbalanced cycles contained in a loose handcuff had incongruent sums modulo $2G$ then equation~\eqref{eq:def_ten_2} could not be satisfied for the circuit walk  around this loose handcuff.) Adjusting again for the possibility $f(e)=0$, this yields $p_{\Sigma}(G;u)= \frac{|G|}{|2G|}p_{\Sigma\backslash e}(G;u) - p_{\Sigma/ e}(G;u)$ in this case.

When $e$ is a bridge of $\Gamma$ that is not a circuit path edge, there are no circuits of $\Sigma$ containing $e$ and thus there are $|G|-1$ non-zero choices for $f(e)$ to extend the tension $f'$ of $\Sigma\backslash e$ to a tension of $\Sigma$ (all the while keeping the sum on any unbalanced cycles an element of $u+2G$). Thus $p_\Sigma(G;u)=(|G|-1)p_{\Sigma\backslash e}(G;u)$ in this case.

When $e$ is a positive loop it forms a circuit of $\Sigma$ and the equation~\eqref{eq:def_ten_2} for a circuit walk around it determines that $f(e)=0$. Hence $p_{\Sigma}(G;u)=0$ in this case.

Finally, when $\Sigma$ consists of a vertex with $\ell\geq 1$ negative loops the value on one loop determines all the others to have the same value by equation~\eqref{eq:def_ten_2} for circuit walks around pairs of loops. The value on a loop must also belong to $u+2G$. This implies $p_{\Sigma}(G;u)=|2G|$ when $u\not\in 2G$ and $p_{\Sigma}(G;u)=|2G|-1$ when $u\in 2G$.

This completes the proof of the recurrence and the Recipe Theorem yields the result.
\end{proof}

\begin{remark}
For graphs, (nowhere-zero) $G$-tensions are dual to (nowhere-zero) $G$-flows in the sense that tensions of the oriented cycle matroid $M(\Gamma)$ of $\Gamma$ are flows of the dual oriented cycle matroid $M(\Gamma)^*$ (isomorphic to $M(\Gamma^*)$ for planar $\Gamma$). This is reflected in the fact that the number of nowhere-zero $G$-tensions is equal to $(-1)^{r(M(\Gamma))}T_{M(\Gamma)}(1-|G|,0)$ and the number of nowhere-zero $G$-flows is equal to $(-1)^{|E|-r(M(\Gamma))}T_{M(\Gamma)}(0,1-|G|)$, which by the duality formula
 for the Tutte polynomial is equal to $(-1)^{r(M(\Gamma)^*)}T_{M(\Gamma)^*}(1-|G|,0)$.

In a similar way, the number of nowhere-zero $G$-potential differences given in Theorem~\ref{theorem:nz-tensions} is equal to 
$$(-1)^{r(\Gamma)}|2G|^{k_u(\Sigma)}T_{\Sigma}\left(1-|G|,0,1-\frac{1}{|2G|}\right)=|2G|^{r_F(E)}S_{M, F}\Big(1-|G|,0,1-\frac{1}{|2G|}\Big)$$
where $M=M(\Gamma)$ and $F=F(\Sigma)$, and the number of nowhere-zero $G$-flows of $\Sigma$ is equal to
\begin{align*}(-1)^{|E|-r(\Gamma)}T_{\Sigma}\left(0,1-|G|,1-\frac{|G|}{|2G|}\right)&=(-1)^{|E|-r_M(E)}|2G|^{r_M(E)}S_{M,F}\left(0,1-|G|,1-\frac{|G|}{|2G|}\right)\\
  &=|2G|^{|E|-r_{F}(E)}S_{M^*, F^*}\left(1-|G|,0,1-\frac{1}{|2G|}\right)\end{align*}
the last line by the duality formula for $S_{M(\Gamma),F(\Sigma)}(X,Y,Z)$ given by Equation~\eqref{eq:dualityS} in Appendix~\ref{app:matroids}.

What remains unclear, though, is how to define $G$-flows and $G$-potential differences for the dual of frame matroids. Tensions and flows of (the cycle matroid of) a graph have a smooth translation to cographic matroids; but for frame matroids it is not apparent, for example, how the condition on $G$-potential differences that sums on unbalanced cycles belong to $2G$ translates to the dual setting.  

\end{remark}
While $t_\Sigma(G)$ for connected $\Sigma$ is a sum of evaluations of the trivariate Tutte polynomial at two different points, and is multiplicative over disjoint unions, 
the invariant $t_\Sigma(G)$ does not itself satisfy a deletion-contraction recurrence. 
To see this, it suffices to consider the case of edge $e$ being a bridge in $\Gamma$ that belongs to a loose handcuff of $\Sigma$. 
We introduce notation for a number of signed graphs: $\Sigma_0$, consisting of a single vertex with one negative loop, and $\Sigma_0'$ a single vertex with two negative loops;
$\Sigma_1$,  consisting of two vertices connected by an edge with one negative loop on each vertex, and $\Sigma_1'$ two vertices connected by an edge, one negative loop on one, and two negative loops on the other; $\Sigma_2$, consisting of a path on three vertices with a negative loop on each vertex.
Supposing there exist $\lambda$ and $\mu$ (possibly depending on $G$) such that $t_{\Sigma}(G)= \lambda t_{\Sigma/e}(G)+\mu t_{\Sigma\backslash  e}(G)$, 
		then we have 
	\begin{displaymath}
	t_{\Sigma_1}(G)=\lambda t_{\Sigma_0'}(G)+\mu t_{\Sigma_0}(G)^2=\lambda t_{\Sigma_0}(G)+\mu t_{\Sigma_0}(G)^2
	\end{displaymath}
	and
	\begin{displaymath}
	t_{\Sigma_2}(G)=\lambda t_{\Sigma_1'}(G)+\mu t_{\Sigma_1}(G)\:t_{\Sigma_0}(G)=\lambda t_{\Sigma_1}(G)+\mu t_{\Sigma_1}(G)\:t_{\Sigma_0}(G)\: .
	\end{displaymath}
From these two equations it follows that 
	\begin{equation}\label{eq.equality}
	t_{\Sigma_1}(G)^2=t_{\Sigma_0}(G)\:t_{\Sigma_2}(G),
	\end{equation}
	but a direct computation shows
	\begin{displaymath}
	\begin{cases}
	t_{\Sigma_0}(G)=|G|-1 \\
	t_{\Sigma_1}(G)=(|G|-|2G|)(|G|-1) + (|2G|-1)\left(|G|-\frac{|G|}{|2G|}-1\right)\\
	t_{\Sigma_2}(G)=(|G|-|2G|)(|G|-1)^2 + (|2G|-1)\left(|G|-\frac{|G|}{|2G|}-1\right)^2\\
	\end{cases}
	\end{displaymath}
	and for $|2G|=2$ and $\frac{|G|}{|2G|}=2^5$ (for example)  equality~\eqref{eq.equality} does not hold.

\subsection{Tensions and colorings} \label{sec:ten_col}


\indent We describe a relation between tensions and colorings, thereby extending Theorem~5.1 in~\cite{chen09} (where only abelian groups of odd order are considered). As usual, $G$ is a finite additive abelian group, and $\Sigma=(\Gamma,\sigma)$ is a signed graph whose underlying graph $\Gamma=(V,E)$ has been given an orientation  $\omega$ compatible with $\sigma$. Define the \emph{difference operator} $\delta: G^V \to G^E$ with respect to the orientation $\omega$ (sometimes we will make the dependence explicit by writing $\delta=\delta_\omega$)  for $g\in G^V$ by
\begin{equation}\label{equation:delta}(\delta g)(e) := \omega(v,e)g(v) + \omega(u,e)g(u).\end{equation}


Changing the orientation of $\omega$ on an edge $e$ has the effect of replacing $\delta g(e)$ by $- \delta g(e)$.
The image of the operator $\delta$ is unchanged under vertex switching, and the kernel of $\delta$ is unchanged up to isomorphism. To see this, it suffices to consider the effect of switching $\Sigma$ at a single vertex $u\in V$ to make equivalent signed graph $\Sigma'=(\Gamma,\sigma')$. Under switching at $u$, an orientation $\omega'$ of $\Sigma'$ compatible with $\sigma'$ may be obtained by changing the sign of the half-edges incident with $u$, i.e. by setting $\omega'(u,e)=-\omega(u,e)$ for each edge $e$ with $u$ as an endpoint, and $\omega'(v,e)=\omega(v,e)$ otherwise.  For given $g\in G^V$, if we define $g'$ by $g'(u)=-g(u)$ while setting $g'(v)=g(v)$ for $v\neq u$, then, for each edge $e=uv$ incident with $u$, 
\begin{align*}
(\delta_\omega g)(e)&=\omega(v,e)g(v)+\omega(u,e)g(u)\\
 & = \omega(v,e)g(v)-\omega(u,e)(-\!g(u))\\
 & =\omega'(v,e)g'(v)+\omega'(u,e)g'(u)  = (\delta_{\omega'} g')(e).
\end{align*} 
For edges not incident with $u$ it is clear that $(\delta_\omega g)(e)=(\delta_{\omega'} g')(e)$. Hence the image of $\delta$ under the orientation $\omega$ of $\Sigma$ is equal to the image of $\delta$ under the orientation $\omega'$ of $\Sigma'$;
likewise, $g\in\ker \delta$ under the given orientation $\omega$ of $\Sigma$ if and only if $g'\in \ker\delta$ under the orientation $\omega'$ of $\Sigma'$, and $g\mapsto g'$ is an isomorphism of $G^V$. 


\begin{theorem}\label{theorem:kernel}
Let $G$ be a finite additive abelian group and $\Sigma=(\Gamma,\sigma)$ a signed graph with underlying graph $\Gamma=(V,E)$. Then $\delta: G^V \rightarrow G^E$ is a group homomorphism under pointwise addition, whose image is contained in the group of $G$-tensions of $\Sigma$, and
\begin{displaymath}
\mathrm{ker}\hspace{0.5mm}\delta \cong G^{k_b(\Sigma)}\times G_2^{k_u(\Sigma)},
\end{displaymath}
where $G_2=\{x\in G: 2x=0\}$. 
\end{theorem} 
\begin{proof}
It is not difficult to see that $\delta$ is a group homomorphism and that $\delta g$ is a $G$-tension for every $g \in G^V$ (see also the first part of the proof of Theorem~5.1 in~\cite{chen09}, which holds for any finite abelian group $G$).\\
\indent We now determine the kernel of $\delta$. It suffices to prove that, when $\Sigma$ is connected,
\begin{equation}\label{equation:kerdelta}
\textrm{ker}\hspace{0.5mm}\delta \cong \left\{\begin{array}{ll} G & \text{if $\Sigma$ is balanced,}\\ G_2 & \text{if $\Sigma$ is unbalanced.}\end{array}\right.
\end{equation}
By Lemma~\ref{lemma:specialbasis_2}, the value of a tension $f$ is determined by its values on a connected basis. 

Suppose that $\Sigma$ is balanced. Then, a connected basis is a spannning tree. Let $T$ be such a tree. Let $g\in \textrm{ker}\hspace{0.5mm}\delta$. Consider a vertex $u$. Let $e=uv$ be an edge in $T$ with $\sigma(e)=1$. Then the vertex $v$ receives the value $g(u)$ by \eqref{equation:delta}. If $\sigma(e)=-1$, then $v$ receives the value $-g(u)$. In general a vertex $w$ receives the value $g(u)$ or $-g(u)$ depending on the parity of the number of negative edges on the unique path in $T$ from $u$ to $w$. Thus, each $g$ depends on the value on a given vertex, and such value can always be freely chosen, and the values on the other vertices is propagated through a  spanning tree. As the component is balanced, the assignment is consistent.


It remains to prove~\eqref{equation:kerdelta} when $\Sigma$ is unbalanced. 

In one direction,  take $g \in G^V$ and for $x\in G_2$ define $g_x\in G^V$ by $g_x(v) := g(v)+x$. Then $\delta g_x = \delta g$, showing that $\text{ker}\hspace{0.5mm}\delta$ contains a subgroup isomorphic to $G_2$.

In the other direction, let $B = T\cup\{e'\}$ be a connected basis of $\Sigma$, in which $T$ is a spanning tree of the underlying graph $\Gamma$. By Lemma~\ref{lemma:specialbasis_2}, we know that a tension is uniquely determined by its values on $B$. 
Assume, without loss of generality, that $\sigma(e')=-1$ (take $e'$ to be an edge signed $-1$ in the unique unbalanced cycle of $B$). 
Take $g\in \textrm{ker}\hspace{0.5mm}\delta$ and consider the value of $g$ on a vertex $v$. As in the balanced case, if $e\in T$, $e=uv$ and $\sigma(e)=+1$, then $g(v)=g(u)$, while if $\sigma(e)=-1$, then $g(v)=-g(u)$, due to (\ref{equation:delta}) and the assumption $g\in \textrm{ker}\hspace{0.5mm}\delta$. In particular, the value of $g$ at each vertex is either $g(u)$ or $-g(u)$ depending on the number of edges signed $-1$ on the unique path in $T$ from $u$ to $v$. Now consider the unique unbalanced cycle containing $e'=u'v'$. Since $\sigma(e')=-1$ and the cycle is unbalanced, the path from $u'$ to $v'$ has an even number of edges signed $-1$. Thus $g(u')=g(v')$.
Equation~\eqref{equation:delta} then states that
$0 = g(u')+g(v') = g(u')+g(u')$. Therefore, $g(u')\in G_2$, and the same is true of $g(u)=\pm g(u')$, and we are done.
%
%
%
\end{proof}

\noindent Equation~(\ref{equation:delta}) implies that if $g \in G^V$ is a $G$-coloring of $\Sigma$, then $\delta g$ is a nowhere-zero $G$-tension of~$\Sigma$.\\
\indent For a balanced signed graph $\Sigma$, the difference operator $\delta$ is surjective (\cite[Theorem 5.1]{chen09}). If $\Sigma$ is unbalanced, the image of $\delta$ has the following characterization. 

\begin{theorem}\label{cor:tension_coloring}
Let $\Sigma=(\Gamma,\sigma)$ be an unbalanced signed graph. Let $f$ be a $G$-tension of $\Sigma$. 
Then $f$ is in the image of $\delta$ 
if and only if for each unbalanced cycle $X$ of $\Sigma$ 
$$\sum_{e \in E(X)}f(e) \in 2G.$$
Thus, $f$ is in the image of $\delta$ if and only if it is a $G$-potential difference.
\end{theorem}
\begin{proof}
Suppose first that $f=\delta g$ for some $g \in G^V$. The terms of the sum
$$\sum_{e \in E(X)}(\delta g)(e)=\sum_{e=uv\in E(X)}\left[\omega(v,e)g(v)+\omega(u,e)g(u)\right]$$ contain for each $v\in V(X)$ two appearances of $g(v)$ with coefficient $\pm 1$, and therefore in total $g(v)$ has coefficient in $\{0,\pm 2\}$, and the overall sum is a multiple of $2$.
This proves that if $f$ is in the image of $\delta$, then the sum of its values around an unbalanced cycle is a multiple of $2$ in~$G$. 

Let us now prove the converse. For balanced connected components the statement follows from (\cite[Theorem 5.1]{chen09}). 
Consider an unbalanced connected component $\Sigma_0$ and let $X$ be an unbalanced cycle in this component. Let $B=X\cup B'$ be a connected basis for $\Sigma_0$ containing $X$ as its unbalanced cycle.
By vertex switching we may assume that there is a single edge $e'$ of $B$ that is negative, and that such an edge belongs to $X$, that is, $\sigma(e)=+1$ for each $e\in E(B)\backslash \{e'\}$ and $e'=u'v'$ has $\sigma(e') = -1$ and $e'\in E(X)$. 

Traversing $X$ as a cycle of $\Gamma$ in either direction gives a cyclic order on $V(X)$: fix one of these.  
Set the orientation $\omega$ of $X$ compatible with $\sigma$ so that 
$\omega(v,e)=-1, \omega(u,e)=+1$ for $e=uv\in E(X)\backslash\{e'\}$ in which $v$ follows $u$ when traversing $X$, and $\omega(v',e')=-1=\omega(u',e')$. 

For the given $G$-tension $f$ of $\Sigma$ we  
have by assumption $x\in G$ such that $\sum_{e\in E(X)}f(e) = 2x$. We construct $g \in G^V$ such that $\delta g  = f$. 
To start the vertex coloring, we set $g(u')=x$. We then use the tension $f$ to give each vertex of $V(X)$ a color as follows.  For $e'=u'v'$, with $v'$ following $u'$, we set $g(v')=-g(u')+f(e')=-x+f(e')$; after this, for $e=uv\in E(X)\backslash\{u'v'\}$, with $v$ following $u$, and $u$ colored, we set $g(v)=g(u)+f(e)$.
This coloring rule is consistent since $-x+\sum_{e\in E(X)}f(e)=-x+2x=x=g(u')$.  To color the remaining vertices, we use the fact that each vertex in $V(\Sigma)\setminus V(X)$ is connected to $V(X)$ by a unique path in $B$; directing these paths outwards from $V(X)$, for each edge $e=uv\in E(B)\setminus \{u'v'\}$ with $v$ following $u$, $u$ colored, we set $g(v)=g(u)+f(e)$.
Then $(\delta g)(e)=g(v)-g(u)=f(e)$ for $e=uv\in E(B)\backslash\{e'\}$, and $(\delta g)(e')=g(v')+g(u')=f(e')$.

The vertex coloring $g:V(\Sigma)\to G$ then has $\delta g=f$ on the edges of $B$; by Lemma~\ref{lemma:specialbasis_2} the (extended) tension $\delta g$ coincides with $f$ on the whole of $E$. The procedure is now repeated for each unbalanced connected component and thus a coloring for the whole $V$ is found.
\end{proof}

	Using Theorem~\ref{cor:tension_coloring} and Theorem~\ref{theorem:kernel},  from which the number of proper $G$-colorings is given by $P_{\Sigma}(G)=|G|^{k_b(\Sigma)}\left(\frac{|G|}{|2G|}\right)^{k_u(\Sigma)}p_\Sigma(G)$, we obtain from Corollary~\ref{thm:gcolors} an alternative derivation for the number of nowhere-zero $G$-potential differences to that obtained by deletion-contraction in Theorem~\ref{theorem:nz-tensions}. 

\begin{corollary}\label{cor:num_c_tens}
	Let $\Sigma$ be a signed graph and $G$ a finite additive abelian group.
Then the number of nowhere-zero $G$-potential differences is given by
	\begin{equation} \label{eq:num_nwz_tensions_from_colorings_2}
	p_\Sigma(G)=(-1)^{r(\Gamma)}|2G|^{k_u(\Sigma)}T_\Sigma\left(1-|G|,0,1-\frac{1}{|2G|}\right) \; .
	\end{equation}
\end{corollary}

\section*{Acknowledgements}
We thank Sergei Chmutov for alerting us to~\cite{FZ} and \cite{Chm}, Cl\'ement Dupont for elucidating to us the relevant parts of~\cite{dupont18}, Steve Noble for pointing us to~\cite{WK04}, Lex Schrijver for leading us to Theorem~\ref{theorem:Xcolorings}, Thomas Zaslavsky for comments on a preprint version of this paper, and the two anonymous referees for their comments on the submitted manuscript.

\appendix
\section*{Appendices}
\renewcommand{\thesubsection}{\Alph{subsection}}



\section{{The trivariate Tutte polynomial of a pair of matroids on a common ground~set}
}\label{app:matroids}
     \setcounter{theorem}{0}
     \renewcommand{\thelemma}{\Alph{section}\arabic{theorem}}


\begin{definition}\label{definition:extendedstp}
For matroids $M_1=(E,r_1)$ and $M_2=(E,r_2)$ on a common ground set~$E$, 
The {\em trivariate Tutte polynomial} of the ordered pair $(M_1,M_2)$ is defined by
\begin{equation}\label{equation:sm1m2_2}
S_{M_1,M_2}(X,Y,Z) := \sum_{A\subseteq E}(X-1)^{r_1(E)-r_1(A)}(Y-1)^{|A|-r_2(A)}(Z-1)^{r_2(A)+r_1(E)-r_1(A)}.
\end{equation}
\end{definition}

\noindent
Given a signed graph $\Sigma=(\Gamma = (V,E),\sigma)$, let $M(\Gamma)=(E,r_M)$ be the cycle matroid of $\Gamma$ with rank function $r_M$, and let $F(\Sigma)=(E,r_F)$ be the frame matroid of $\Sigma$ with rank function $r_F$.
Since $r_F(A)-r_M(A)=k_u(\Sigma\backslash A^c)\geq 0$ for each $A\subseteq E$, we see from the subset expansion~\eqref{equation:sm1m2_2} that the polynomial $S_{M(\Gamma), F(\Sigma)}$ is divisible by $(Z-1)^{r_M(E)}$.
The trivariate Tutte polynomial of a signed graph $\Sigma$ with underlying graph $\Gamma=(V,E)$, as defined in Definition~\ref{defn:signedtutte}, is given by
\begin{equation}\label{def_tutte_from_mat_2}
T_{\Sigma}(X,Y,Z) = (Z-1)^{-r_M(E)}S_{M(\Gamma), F(\Sigma)}(X,Y,Z).
\end{equation}

In order to establish properties of  $S_{M(\Gamma), F(\Sigma)}(X,Y,Z)$ we shall use an auxiliary four-variable polynomial invariant of matroid pairs first introduced by Welsh and Kayibi. 
\begin{definition}[\cite{WK04}]
Let $M_1=(E,r_1)$ and $M_2=(E,r_2)$ be two matroids on the same ground set. The {\em linking polynomial} of the matroid pair $(M_1,M_2)$ is defined by
\begin{equation}\label{eq:J}Q_{M_1,M_2}(x,y,u,v)=\sum_{A\subseteq E}(x-1)^{r_1(E)-r_1(A)}(y-1)^{|A|-r_1(A)}(u-1)^{r_2(E)-r_2(A)}(v-1)^{|A|-r_2(A)}.\end{equation}
\end{definition}

\noindent
The linking polynomial of a matroid pair $(M_1,M_2)$ is in fact a specialization of the trivariate Tutte polynomial as there are just three independent parameters in the exponents -- size, rank in $M_1$, and rank in $M_2$:
\begin{equation}\label{eq:Q_from_T}
Q_{M_1,M_2}(x,y,u,v)=(u-\!1)^{r_2(E)}(y-\!1)^{-r_1(E)}S_{M_1,M_2}\Big(1+(x-\!1)(u-\!1),1+(y-\!1)(v-\!1), 1+\frac{y-\!1}{u-\!1}\Big).
\end{equation}

The trivariate Tutte polynomial in the form of the four-variable linking polynomial is evidently the canonical Tutte polynomial of matroid pairs in the sense of~\cite{dupont18}: in the language of that paper, the universal norm of a pair of matroids on a common ground set is $$(M_1,M_2)\mapsto x^{r_1(E)-r_1(A)}y^{|A|-r_1(A)}u^{r_2(E)-r_2(A)}v^{|A|-r_2(A)}$$
and the universal Tutte character (in eight variables) boils down by homogeneity to the four-variable linking polynomial.   

By setting one of the variables $x,y,u,v$ to be constant in equation~\eqref{eq:Q_from_T} and relabelling variables we obtain the trivariate Tutte polynomial as a specialization of the linking polynomial. For example, setting in turn $x,y,u$ and $v$ to equal $2$ we obtain:
$$S_{M_1,M_2}(X,Y,Z) =  :\sum_{A\subseteq E}(X-1)^{r_1(E)-r_1(A)}(Y-1)^{|A|-r_2(A)}(Z-1)^{r_2(A)-r_1(A)} \qquad\qquad\qquad$$
\begin{align*}& = (X-1)^{r_1(E)-r_2(E)}(Z-1)^{r_1(E)}Q_{M_1,M_2}\left(2,1+(X\!-\!1)(Z\!-\!1),X,1+\frac{Y-1}{(X\!-\!1)(Z\!-\!1)}\right)\\
& = (Z\!-\!1)^{r_2(E)}Q_{M_1,M_2}\left(1+(X\!-\!1)(Z\!-\!1),2,\frac{Z}{Z-1},Y\right)\\
& = (Z\!-\!1)^{r_1(E)} Q_{M_1,M_2}\Big(X,Z,2,1+\frac{Y-1}{Z-1}\Big)\\
&= (Y\!-\!1)^{r_1(E)-r_2(E)} (Z\!-\!1)^{r_2(E)}Q_{M_1,M_2}\left(1+\frac{(X\!-\!1)(Z\!-\!1)}{Y-1}, Y, 1+\frac{Y-1}{Z-1},2\right)
 \end{align*}

\noindent Three easily verified properties of the linking polynomial yield the corresponding properties for the trivariate Tutte polynomial. 
 Firstly,
the Tutte polynomials of $M_1$ and of $M_2$ are the bivariate specializations
$$Q_{M_1,M_2}(x,y,2,2)=T_{M_1}(x,y),\quad Q_{M_1,M_2}(2,2,u,v)=T_{M_2}(u,v).$$
Secondly, swapping matroids gives
$$Q_{M_2,M_1}(x,y,u,v)=Q_{M_1,M_2}(u,v,x,y),$$
Lastly, taking dual matroids gives
$$Q_{M_1^*,M_2^*}(x,y,u,v)=Q_{M_1,M_2}(y,x,v,u).$$
(The dual of a matroid $M=(E,r)$ is the matroid $M^*=(E,r^*)$ with rank function
$r^*(A)=r(A^c)+|A|-r(E)$,
for $A \subseteq E$. The dual matroid is alternatively specified by its bases being the complements of the bases of $M$. 
The independent sets of $M^*$ are those sets $A$ such that there is a basis of $M$ contained in $A^c$. 
If $M=M(\Gamma)$ is the cycle matroid of a planar graph $\Gamma$, then $M^*$ is the cycle matroid of the (surface) dual planar graph $\Gamma^*$.) 



We may use these  properties of the linking polynomial to first obtain that
the trivariate Tutte polynomial $S_{M_1,M_2}(X,Y,Z)$ 
includes the Tutte polynomial of $M_1$ and the Tutte polynomial of $M_2$ as specializations:
\begin{equation}\label{eq:Tutte_M1}
T_{M_1}(X,Y)=(Y-1)^{-r_1(E)}S_{M_1,M_2}(X,Y,Y)
\end{equation}
and
\begin{equation}\label{eq:Tutte_M2}
T_{M_2}(X,Y)=(X-1)^{r_2(E)}S_{M_1,M_2}\left(X, Y,\frac{X}{X-1}\right),
\end{equation}
where the right-hand side of either identity is well-defined for $Y=1$ (respectively $X=1$) once expanded as a polynomial in $Y$ (respectively $X$).
The polynomial $S_{M_1,M_2}(X,Y,Z)$ thus includes among its evaluations enumerations of combinatorial objects associated with $M_1$ or with $M_2$ taken separately, 
such as the number of connected spanning sets of $M_1$ (evaluating at $(1,2,2)$).

Secondly, 
\vspace{0.2cm}

$S_{M_2,M_1}(X,Y,Z)=$
\begin{equation}\label{eq:swapM1M2}(X-1)^{r_2(E)}(Y-1)^{-r_1(E)}(Z-1)^{r_1(E)+r_2(E)}S_{M_1,M_2}\left(X,Y,1+\frac{Y-1}{(X-1)(Z-1)}\right).
\end{equation}

Lastly,
\vspace{0.2cm}

$S_{M_1^*, M_2^*}(X,Y,Z)=$
\begin{equation}\label{eq:dualityS}\hspace{1.25cm}(X-1)^{-r_1(E)}(Y-1)^{r_2(E)}(Z-1)^{|E|-r_1(E)-r_2(E)}S_{M_1,M_2}\left(Y,X,1\!+\frac{(X\!-1)(Z\!-1)}{Y-1}\right).\end{equation}
The combination of swapping matroids and taking duals gives
$$S_{M_2^*,M_1^*}(X,Y,Z) = (Z-1)^{|E|}S_{M_1,M_2}\left(Y,X,\frac{Z}{Z-1}\right).$$


As noted by Welsh and Kayibi~\cite{WK04}, when $M_1$ and $M_2$ form a matroid perspective $M_2 \rightarrow M_1$, which is to say that every circuit of $M_2$ is a union of circuits of~$M_1$, the linking polynomial of a pair of matroids $(M_1,M_2)$ on a common ground set, and hence the trivariate Tutte polynomial of $(M_1,M_2)$,  coincides with the Las Vergnas polynomial. 
Explicitly, if $M_2 \rightarrow M_1$ is a matroid perspective then
\begin{equation}\label{equation:lasvergnas}
S_{M_1,M_2}(X,Y,Z) = (Z-1)^{r_2(E)}T_{M_2\to M_1}\left(X,Y,\frac{1}{Z-1}\right),
\end{equation}
where $T_{M_2\to M_1}$ is the Tutte polynomial of $M_2 \rightarrow M_1$, as given in~\cite[Equation~(5.2)]{vergnas1999} and first defined by Las Vergnas in~\cite{vergnas75} (for perspectives of arbitrary length). Every $\Delta$-matroid $D$ (see~\cite{CMNR16} for a definition and extensive treatment of $\Delta$-matroids) gives rise to a matroid perspective $D_{\text{max}} \rightarrow D_{\text{min}}$ between the upper and lower matroid of $D$. Hence, equation (\ref{equation:lasvergnas}) also relates the polynomial $S_{M_1,M_2}$ to the Las Vergnas polynomial of $D$ (see Definition 6.1 in~\cite{CMNR16}). In turn, the Las Vergnas polynomial of the $\Delta$-matroid $D(G)$, where $G$ is an embedded graph, specializes to the Las Vergnas polynomial of $G$ as defined in~\cite{vergnas75}, since in that case $D(G)_{\text{min}}$ is the cycle matroid of $G$, and $D(G)_{\text{max}}$ is the bond matroid of the surface dual of $G$ (Corollary 5.4 in~\cite{CMNR16}).

%
%
\paragraph{Deletion-contraction and the Recipe Theorem for matroid pairs} \label{sec:del_cont_recipe_B}

A deletion-contraction recurrence formula for the trivariate Tutte polynomial $S_{M_1,M_2}$ follows directly from that given for the linking polynomial $Q_{M_1,M_2}$ in~\cite[Theorem 1]{WK04}. 

\begin{theorem}\label{theorem:delcontr}
Let $M_1$ and $M_2$ be matroids on a set $E$, and let $e \in E$.
Then the trivariate Tutte polynomial $S_{M_1, M_2}=S_{M_1, M_2}(X,Y,Z)$ satisfies the following recurrence:
\begin{align*}
&S_{M_1,M_2}=\\
&\begin{cases}
X(Z\!-\!1)\, S_{M_1/e,M_2/e} & \mbox{if $e$ is a coloop in both $M_1$ and $M_2$;}\\
Y\, S_{M_1\backslash  e,M_2\backslash  e} & \mbox{if $e$ is a loop in both $M_1$ and $M_2$;}\\
 Z\, S_{M_1\backslash  e,M_2\backslash  e} & \mbox{if $e$ is a loop in $M_1$ and a coloop in $M_2$;}\\
 (XZ-\!X-\!Z+\!Y)\, S_{M_1/e,M_2/e} & \mbox{if $e$ is a coloop in $M_1$ and a loop in $M_2$;}\\
(Z\!-\!1) \left[S_{M_1/e,M_2/e}\!+\!(X\!-\!1)\, S_{M_1\backslash  e,M_2\backslash  e}\right] & \mbox{if $e$ is a coloop in $M_1$ and ordinary in $M_2$;}\\
 (Y\!-\!1)\,S_{M_1/e,M_2/e}+ S_{M_1\backslash  e, M_2\backslash  e} & \mbox{if $e$ is ordinary in $M_1$ and a loop in $M_2$;}\\
(Z\!-\!1)\, S_{M_1/e,M_2/e}+ S_{M_1\backslash  e,M_2\backslash  e} &
\begin{cases}
 \mbox{if (i) $e$ is a loop in $M_1$ and ordinary in  $M_2$, or}\\
 \mbox{(ii)  $e$ is ordinary in $M_1$ and a coloop in $M_2$, or}\\
 \mbox{(iii) $e$ is ordinary in both $M_1$ and $M_2$.}
\end{cases}\\
%
%
%
\end{cases}
\end{align*}

\end{theorem}

\noindent 
The invariant 
\begin{equation}\label{eq:universal_matroidpair} \alpha^{r_1(E)}\beta^{|E|-r_2(E)}\gamma^{r_2(E)-r_1(E)}S_{M_1,M_2}(\frac{x}{\alpha},\frac{y}{\beta}, \frac{z}{\gamma}),\end{equation}
is easily checked from the subset expansion~\eqref{equation:sm1m2_2} for $S_{M_1,M_2}(X,Y,Z)$ to be a polynomial in $x,y,z,\alpha,\beta$ and $\gamma$  (by using $0\leq r(A)\leq r(E)$ for $A\subseteq E$ in a matroid $(M,r)$ on ground set $E$). 
By considering the effect of edge deletion and contraction on the prefactor of $S_{M_1,M_2}(\frac{x}{\alpha},\frac{y}{\beta}, \frac{z}{\gamma})$ in~\eqref{eq:universal_matroidpair} and multiplying the recurrence for $S_{M_1,M_2}(X,Y,Z)$ given in Theorem~\ref{theorem:delcontr} by this prefactor, and taking $X=\frac{x}{\alpha}, Y=\frac{y}{\beta}$ and $Z=\frac{z}{\gamma}$, 
 we find that this scaled version of the trivariate Tutte polynomial satisfies a recurrence giving a ``Recipe Theorem" for deletion-contraction invariants of matroid pairs. 
\begin{theorem}\label{theorem:recipe_matroidpairs_WK}
Let $R$ be an invariant of matroid pairs $(M_1,M_2)$, in which $M_1=(E,r_1)$ and $M_2=(E,r_2)$, which is multiplicative over direct sums and 
satisfies the following deletion-contraction recurrence:
\begin{align*}
&R_{M_1,M_2}=\\
&\begin{cases}
x\, R_{M_1/e,M_2/e} & \mbox{if $e$ is a coloop in both $M_1$ and $M_2$;}\\
y\, R_{M_1\backslash e,M_2\backslash  e} & \mbox{if $e$ is a loop in both $M_1$ and $M_2$;}\\
z\, R_{M_1\backslash  e,M_2\backslash  e} & \mbox{if $e$ is a loop in $M_1$ and a coloop in $M_2$;}\\
\left(\frac{\alpha(y-\beta)}{\gamma}+\frac{\beta(x-\alpha)}{\gamma}\right)\, R_{M_1/e,M_2/e} & \mbox{if $e$ is a coloop in $M_1$ and a loop in $M_2$;}\\
\alpha\, R_{M_1/e,M_2/e}\!+\frac{\beta(x\!-\!\alpha)}{\gamma}\, R_{M_1\backslash  e,M_2\backslash  e} & \mbox{if $e$ is a coloop in $M_1$ and ordinary in $M_2$;}\\
 \frac{\alpha(y\!-\!\beta)}{\gamma}\,R_{M_1/e,M_2/e}+ \beta R_{M_1\backslash  e, M_2\backslash  e} & \mbox{if $e$ is ordinary in $M_1$ and a loop in $M_2$;}\\
(z\!-\!\gamma)\, R_{M_1/e,M_2/e}+ \beta R_{M_1\backslash  e,M_2\backslash  e} & \mbox{if $e$ is a loop in $M_1$ and ordinary in  $M_2$,}\\
\alpha\, R_{M_1/e,M_2/e}+ \gamma R_{M_1\backslash  e,M_2\backslash  e} & \mbox{if $e$ is ordinary in $M_1$ and a coloop in $M_2$,}\\
\alpha\, R_{M_1/e,M_2/e}+ \beta R_{M_1\backslash  e,M_2\backslash  e} &  \mbox{if $e$ is ordinary in both $M_1$ and $M_2$,}\\
\end{cases}\end{align*}

and $R_{M_1,M_2}=1$ if $E=\emptyset$. 

Then 
$$R_{M_1,M_2}=\alpha^{r_1(E)}\beta^{|E|-r_2(E)}\gamma^{r_2(E)-r_1(E)}S_{M_1,M_2}(\frac{x}{\alpha}, \frac{y}{\beta}, \frac{z}{\gamma}).$$ 
\end{theorem}

\noindent
Theorem~\ref{theorem:recipe_matroidpairs_WK} can be derived from the ``Recipe Theorem"~\cite[Theorem 5]{WK04} for the four-variable linking polynomial $Q_{M_1,M_2}$; the parameters \begin{equation}\label{eq:WKpar}(m,n,A,B,C,D,a,b,c,d,p,q,z,t)\end{equation} in~\cite{WK04} 
are the parameters $$(\beta, \alpha,x, \frac{\alpha(y-\beta)}{\gamma}+\frac{\beta(x-\alpha)}{\gamma}, z,y,\beta,z-\gamma,\frac{\beta(x-\alpha)}{\gamma},\alpha,\beta,\frac{\alpha(y-\beta)}{\gamma},\gamma,\alpha)$$
in our notation. 
Subject to the condition $AD\neq BC$ among the parameters~\eqref{eq:WKpar}, Welsh and Kayibi further show that if an invariant of matroid pairs on a common ground set satisfies a general deletion-contraction formula according to the nine possible edge types, as itemized in~\cite[Equations~(1) to (13)]{WK04} and featuring the parameters~\eqref{eq:WKpar}, then the dependencies among the parameters~\eqref{eq:WKpar} are those seen among the coefficients in the recurrence of Theorem~\ref{theorem:recipe_matroidpairs_WK}. When $AD=BC$ among the parameters~\eqref{eq:WKpar} 
there is ``a multitude of very special cases"\cite{WK04}, in each case the invariant taking the form of a single monomial (see~\cite[Section 7.5]{WK04} for an example). 

Theorem~\ref{theorem:recipe2} (the ``Recipe Theorem" for  deletion-contraction invariants of signed graphs $\Sigma=(\Gamma,\sigma)$ up to  vertex-switching) may be obtained as a corollary of Theorem~\ref{theorem:recipe_matroidpairs_WK} upon taking $M_1=F(\Sigma)$, $M_2=M(\Gamma)$, discarding the edge type combinations of a loop in $F(\Sigma)$ and coloop or ordinary edge in $M(\Gamma)$ as these cannot occur for this particular choice of matroids, and verifying that the reduction of a bouquet of negative loops by deletion and contraction of edges leads to the  boundary condition for bouquets of negative loops given in the statement of Theorem~\ref{theorem:recipe2}.  (A negative loop is a loop in $M(\Gamma)$ that is ordinary in $F(\Sigma)$ if there is more than one negative loop in the bouquet and a coloop  in $F(\Sigma)$ if there is just one negative loop. Deletion and contraction of a negative loop is defined for signed graphs in Zaslavsky~\cite{zas82} and requires enlarging the domain of signed graphs to include half-arcs and free loops.)

\section{Proof of Lemma~\ref{lemma:specialbasis_2}}
\label{sec:app_fig_lemma}

In this appendix we prove Lemma~\ref{lemma:specialbasis_2}, which involves a detailed case analysis; after giving this proof, we then explain how the same result may be derived in the linear algebra framework of~\cite{chen09}.

\begin{proof}[Proof of Lemma~\ref{lemma:specialbasis_2}]
We begin by establishing uniqueness. A $G$-tension $f:E\to G$ of a signed graph $\Sigma=(\Gamma,\sigma)$ gives by restriction a $G$-valued function $f'=f_{|B}$ on the edges of a connected basis $B$ of $\Sigma$. Thus it suffices to prove that any function $f':B\to G$ can be uniquely extended to a $G$-tension~$f$ of~$\Sigma$. 
	Consider two tensions, $g$ and $h$, that coincide on $B$. Then $f=g-h$ is also a tension and is zero on $B$.
	For each $e \in E\backslash  B$, there is a unique circuit walk in $\Sigma$ consisting of the edges $B \cup \{e\}$ and the signature~$\sigma$ restricted to these edges. The edge $e$ is not a circuit path edge of this circuit walk as $B$ is a connected basis. Hence, given the values of $f$ on $B$, the tension condition in equation~\eqref{eq:def_ten_2} uniquely determines the value on $e$ of the extension of $f'$ to the tension $f$ of $\Sigma$, and we see that $f(e)=0$ as $f$ is identically zero on~$B$. Thus $f(e)=0$ for each $e\in E\backslash  B$. This finishes the proof of uniqueness.

	The procedure described above gives, for an arbitrary function $f':B\to G$ and edge $e\in E\backslash B$, a unique assignment $f(e)$ determined by a walk around the circuit contained in~$B\cup \{e\}$. It remains to see that this procedure gives a tension, regardless of the basis $B$ and the function $f'$. 
	
For this purpose, consider the connected graph~$T$ with the connected bases as vertices, two connected bases joined by an edge if there is an edge exchange transforming one into the other. 
Every circuit in a connected signed graph can be obtained as a subgraph after adding a unique edge to some connected basis. For a connected basis~$B'$, let $\mathcal{C}_{B'}$ be the set of circuit walks that arise by adding an edge to~$B'$.
By construction, $f:E\to G$ is defined in such a way that the equations~\eqref{eq:def_ten_2} for circuit walks from~$\mathcal{C}_{B}$ are satisfied. 
The equations~\eqref{eq:def_ten_2} given by a circuit walk in~$\mathcal{C}_{B'}$ for any connected basis $B'$ are all satisfied by $f$ provided the following claim holds:
\begin{quote}
\label{eq.claim}
If edges $e_1$ and $e_2$ are added to a generic connected basis $B$, 
and the value $f(e_i)$ is determined by equation~\eqref{eq:def_ten_2} for the unique circuit walk in~$B\cup\{e_i\}$, $i=1,2$, then equation~\eqref{eq:def_ten_2} is satisfied by $f:B\cup\{ e_1,e_2\}\to G$ for each circuit walk in~$B\cup\{e_1,e_2\}$.
\end{quote}
 Indeed, for each change of connected basis $B'=B\cup \{e_1\}\setminus \{e_2\}$, when extending $f':B\to G$ to $f:B\cup\{e_1\}\to G$ the value $f(e_1)$ determined by equation~\eqref{eq:def_ten_2} for a walk around the unique circuit in $B\cup\{e_1\}$ is such that, when extending $f':B'\to G$ to $f:B'\cup\{e_2\}\to G$, the value $f(e_2)$ determined by equation~\eqref{eq:def_ten_2} for a walk around the unique circuit in $B'\cup\{e_2\}$ coincides with the value $f'(e_2)$. Then, for each connected basis $B$, not only are the equations given by circuit walks in $\mathcal{C}_B$ satisfied, but also those in $\mathcal{C}_{B'}$. 

The above claim 
is established by a case analysis detailed in Figures~\ref{fig:t1af},\ref{fig:t1bf},\ref{fig:t1cf},\ref{fig:t1df},\ref{fig:t2bf} and associated Tables~\ref{tab:t1af},\ref{tab:t1bf},\ref{tab:t1cf},\ref{tab:t1df},\ref{tab:t2bf} to see that all the equations given by circuit walks in $B\cup\{e_1,e_2\}$ are determined by those given by the circuit walk in $B\cup \{e_1\}$ and the circuit walk in~$B\cup \{e_2\}$.
This finishes the proof of Lemma~\ref{lemma:specialbasis_2}.
%
%
%
%
%
%
%
\end{proof}

\paragraph{Explanation of the figures.}
In the tables and figures given below, the relevant cases occur when the signed graph is connected and unbalanced; however, also included are the balanced cases (remove an edge within one of the cycles that is forced to be negative to turn an unbalanced connected basis into a balanced connected basis; in balanced cases we may furthermore assume that all the signs of the edges are positive, and the coefficients $2$ in the given equations become $0$). The various cases of the case analysis are obtained by considering a generic tree with an additional edge closing a negative cycle (to make a connected basis of an unbalanced signed graph), adding an edge $e_1$, and then adding an additional edge $e_2$: for instance, Figure~\ref{fig:t1af} arises by considering that the first edge~$e_1$ closes another cycle, independent of the basis cycle, and then~$e_2$ is a chord of one of the cycles. All the figures of type~I are those in which the first edge closes a cycle independent of the basis cycle (that is, it is not a chord of the negative cycle of the given connected basis), while the figures of type~II are those in which the first edge is a chord of the basis cycle. Some of these figures can be obtained in different ways.

Leaves are not relevant in equations~\eqref{eq:def_ten_2}; therefore 
Figures \ref{fig:t1af},\ref{fig:t1bf}, \ref{fig:t1cf}, \ref{fig:t1df} and \ref{fig:t2bf} only give the different schematic structure of the circuits when two edges are added to a connected basis of an unbalanced signed graph (which contains a negative cycle). In order for Figures \ref{fig:t1af},\ref{fig:t1bf}, \ref{fig:t1cf}, \ref{fig:t1df} and \ref{fig:t2bf} to be exhaustive, some paths labelled $t_i$ may be empty.

The first two columns of a table indicate in which path labelled $t_j$ the edge $e_1$ or $e_2$ is located within the corresponding figure (symmetric cases have been omitted).
The heading ``$\text{eq}i$'' indicates the equation for a walk around the unique circuit contained in $B\cup \{e_i\}$. The $s_i$ indicates the sign of the path labelled $t_i$ (with $\pm$ in the column with heading $(s_1,\ldots,s_6)$ indicating that the path can be positive or negative), and $s_{ij\dots k}$ is short for the product $s_is_j\cdots s_k$. A path can only be empty if its sign is positive. 
In the column ``other cycles'' the additional circuits (containing both $e_1$ and $e_2$) are given in terms of the equation~\eqref{eq:def_ten_2} they define; the equality indicates how the equations for the circuit walks in $B\cup \{e_1\}$ and $B\cup \{e_2\}$ determine the equation given by each of these other circuit walks.
Absence of a horizontal separator in a given column indicates that the assumption stated in the cell for the column is the same for the whole horizontal block.

Each expression $t_i$ is related to an oriented arc in the figure. If the arc is oriented from the vertex $v_1$ to the vertex $v_k$ then it is associated to a walk $W=(v_1,e_1,v_2,\ldots,e_{k-1},v_k)$. The edges $e_1,\ldots,e_{k-1}$ are then oriented following the orientation of the arc in such a way that
\[
f(W)=t_i=\sum_{j=1}^{k-1} f(e_j)
\]
(notice that all the signs of this expression are positive).
In other words, the half-edge corresponding to $e_1$ attached to $v_1$ is always pointing from $v_1$ to $v_k$. The remaining half-edges (which have a natural ordering from $v_1$ to $v_k$) point in a direction consistent with the previous half-edge: 
\begin{itemize}
	\item If $h$ is the second half-edge of an edge, then $h$ points in the same direction as the first half-edge if the edge is positive, and in the opposite direction if the edge is negative.
	\item If $h$ is the first half-edge of an edge, then it points in the same direction as the second half-edge of the previous edge (thus, at each internal vertex of the path the two incident half-edges are aligned).
\end{itemize}
Therefore, if $W_1,W_2,\ldots,W_k$ are paths in which the terminal vertex of one path is the initial vertex of the next path, then the height of their concatenation is given by 
\[f(W_1*W_2*\cdots*W_k)=r_1t_1+\ldots+r_kt_k,\] where
\begin{align*}
r_{i}&=\prod_{j = 1}^{i-1} \sigma(W_j)=\prod_{j = 1}^{i-1} s_j \quad \text{if the direction of $W_i$ aligns with that of the concatenation,}\\
r_{i}&=-\prod_{j = 1}^{i} \sigma(W_j)=-\prod_{j = 1}^{i} s_j\qquad \text{if the direction is opposite.}
\end{align*} 

\paragraph{Connection with Chen and Wang~\cite{chen09}.}
The coefficients in front of the expression $t_i$  in Tables~\ref{tab:t1af}-\ref{tab:t2bf} are integers and give, up to a sign, the coefficients of the indicator vectors of the circuits in \cite{chen09} (denoted by $I_C$ in \cite[(3.2)]{chen09}). Then the equations given in \eqref{eq:def_ten_2} are those in \cite[(4.4)]{chen09}. By introducing cuts (\cite[Section~2]{chen09}) and bonds (following Zaslavsky's \cite{zas82,zas82a}), which are minimal cuts, they show that, as subspaces of $\mathbb{R}^{|E(\Sigma)|}$, the cut space is orthogonal to the circuit space~\cite[Lemma~3.3]{chen09}, and giving a dimensional argument they show that the $\mathbb{R}$-vector subspace of $\mathbb{R}^{|E(\Sigma)|}$ generated by the indicator vectors of the circuits is of dimension $|E|-(|V|-k_b(\Sigma))$ \cite[Theorem~3.5]{chen09}. A basis  is given by the indicator vectors of the circuit walks $C_i$ in $B\cup \{e_i\}$, $e_i\in E\setminus B$, where $B$ is a connected basis; the expression of any other circuit walk $C$ as a linear combination of the $\{C_i\}_{e_i\in E\setminus B}$ can be thus obtained using the coefficients of $C$ in the coordinates given by the edges in $E\setminus e_i$. This shows that the coefficients of the linear combination are integers (as, restricted to the coordinates given by $e_i\in E\setminus B$, these indicator vectors of the $C_i$ give a diagonal matrix with $\pm 1$ as diagonal entries), and therefore the equations~\eqref{eq:def_ten_2} for circuit walks can be deduced from those for the circuit walks $\{C_i\}_{e_i\in E\setminus B}$ alone. This gives a proof of Lemma~\ref{lemma:specialbasis_2}; alternatively, the proof of Lemma~\ref{lemma:specialbasis_2} given earlier yields a proof of the result on the dimension of the circuit space from~\cite{chen09} that does not require the introduction of bonds.

%


%
%

\begin{figure}[h!]
	\centering
	\begin{tikzpicture}[scale=0.48,decoration={
		markings,
		mark=at position 0.4 with {\arrow[line width=1.5mm]{to}}}]
	\tikzstyle{vertex}=[circle,fill=black!100,minimum size=7pt,inner sep=0pt]
	
	\tikzstyle{vertex2}=[circle,fill=black!100,minimum size=6pt,inner sep=0pt]
	%
	
	
	\clip (-10.5,-4) rectangle (12, 4);

	\node[vertex2] (v1) at (-6,-3) {};
	\node[vertex2] (v2) at (-6,3) {};
	\node[vertex2] (v3) at (0,0) {};
	\node[vertex2] (v4) at (6,0) {};
	
	
	
	\draw[line width=2pt,postaction={decorate}] (v2.center) to[out=180, in=180,looseness=2] node [midway,left] {$t_1$} (v1.center);
	
	\draw[line width=2pt,postaction={decorate}] (v2.center) to[out=270, in=90]  node [midway,left] (v212) {$t_2$} (v1.center);
	

	\draw[line width=2pt,postaction={decorate}] (v1.center) to[out=0, in=270]  node [midway,below] (v13) {$t_4$} (v3.center);

	\draw[line width=2pt,postaction={decorate}] (v3.center) to[out=90, in=0]  node [midway,above] (v32) {$t_3$} (v2.center);
	\draw[line width=2pt,postaction={decorate}] (v4.center) to[out=180, in=0]  node [midway,above] (v43) {$t_5$} (v3.center);
	\draw[line width=2pt,postaction={decorate}] (v4.south) to[out=300, in=60,looseness=70]  node [midway,right] (v44) {$t_6$} (v4.north);

	\end{tikzpicture}
	\caption{Figure type IA.} \label{fig:t1af}
\end{figure}
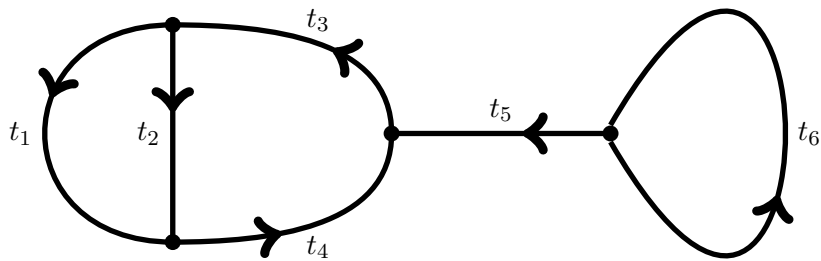

\begin{table} 
	\centering
	\small{
		\begin{tabular}{|@{}p{0.7cm}@{}|@{}p{0.9cm}@{}|@{}p{3cm}@{}|@{}p{3cm}@{}|@{}p{3.6cm}@{}|@{}p{4cm}@{}|}
			\hline
			$e_1\in$ & $e_2\in$& eq1 & eq2 & $(s_1,s_2,s_3,s_4,s_5,s_6)$& other cycles  \\
			\hline
			\multirow{8}{1.5em}{$t_6$} & \multirow{5}{1.5em}{$t_1$ \newline ($t_2$ sym.)} &\multirow{2}{3cm}{$t_6$}& $t_1-s_{12} t_2$&$(s_2,\pm,\pm,-s_{2 3},\pm,+)$ & NA \\
			\cline{4-6}
			&& & $t_1+s_1 t_4+s_{14}t_3$ &$(s_{34},\pm,\pm,-s_{23},\pm,+)$& NA\\
			\cline{3-6}
			
			& &  \multirow{2}{3cm}{$t_6+s_6 t_5+s_{65}t_3+s_{653}t_2+s_{6532}t_4-s_{653245}t_5$} 
			&$t_1-s_{12}t_2$
			&$(s_2,\pm,\pm,-s_{23},\pm,-)$
			& $t_6+s_{6}t_5+s_{65}t_3+s_{653}t_1+s_{6531}t_4-s_{653145}t_5=
			\text{eq1}+s_{653}\text{eq2}$\\
			\cline{4-6}
			
			&&&\multirow{2}{3cm}{$t_1+s_1 t_4+s_{14} t_3$}
			& \multirow{2}{1.5em}{$(s_{34},\pm,\pm,-s_{23},\pm,-)$} 
			& 
			$t_6+2s_6t_5-2s_{654}t_4-s_{6542}t_2+s_{6542}t_1=
			\text{eq2}+s_{6542}\text{eq1}$\\
			\cline{6-6}
			&&&& & 
			$t_6+2s_6t_5+2s_{65}t_3+s_{653}t_2-s_{65321}t_1=
			\text{eq2}-s_{65321}\text{eq1}$\\
			\cline{2-6}
			& \multirow{3}{10em}{$t_3$} 
			& $t_6$& \multirow{1}{10em}{$t_3+s_3t_1+s_{31}t_4$} 
			&$(\pm,-s_1,s_{14},\pm,\pm,+)$
			& NA \\
			\cline{3-6}
			&& \multirow{2}{3cm}{$t_6+2s_6 t_5-2s_{654}t_4-s_{6541}t_1+s_{6541}t_2$}
			& \multirow{2}{10em}{$t_3+s_3t_1+s_{31}t_4$} 
			&\multirow{2}{10em}{$(\pm,-s_1,s_{14},\pm,\pm,-)$}
			& $t_1-s_{12}t_2-2s_{123}t_3-2s_{1235}t_5+s_{1235}t_6=
			s_{1235}\text{eq1}-2s_{123}\text{eq2}$\\
			\cline{6-6} 
			&&&&& $t_3+s_3t_2+s_{32}t_4-2s_{3245}t_5+s_{3245}t_6=
			s_{3245}\text{eq1}+\text{eq2}$ \\
			\cline{1-6}
			
			\multirow{9}{1.5em}{$t_1$} & \multirow{3}{1em}{$t_2$} 
			& $t_1+s_1t_4+s_{14}t_3$ 
			& $t_2+s_2 t_4+s_{24} t_3$ 
			& $(s_{34},s_{34},\pm,\pm,\pm,-)$ 
			& $t_1-s_{12}t_2=\text{eq1}-s_{12}\text{eq2}$\\
			\cline{3-6}
			&
			&\multirow{2}{3cm}{$t_1+s_1t_4-2s_{145}t_5+s_{145}t_6+s_{14565}t_3$}
			&$t_2+s_2t_4-2s_{245}t_5+s_{245}t_6+s_{24565}t_3$ 
			& $(-s_{34},-s_{34},\pm,\pm,\pm,-)$ 
			& $t_1-s_{12}t_2=\text{eq1}-s_{12}\text{eq2}$\\
			\cline{4-6}
			&& 
			& $t_2+s_2t_4+s_{24}t_3$ 
			& $(-s_{34},s_{34},\pm,\pm,\pm,-)$
			& $t_1-s_{12}t_2-2s_{123}t_3-2s_{1235}t_5+s_{1235}t_6=
			\text{eq1}-s_{12}\text{eq2}$\\
			\cline{2-6}
			
			&\multirow{6}{2em}{$t_3$ \newline
				($t_4$ sym.)}
			& \multirow{2}{3cm}{$t_1-s_{12}t_2$} 
			& $t_3+s_{3}t_2+s_{32}t_4$ 
			&$(s_2,\pm,s_{34},\pm,\pm,-)$
			& $t_3+s_3t_1+s_{31}t_4=
			s_3\text{eq1}+\text{eq2}$\\
			\cline{4-6}
			&& & $t_3+s_3t_2+s_{32}t_4-2s_{3245}t_5+s_{3245}t_6$
			&$(s_2,\pm,-s_{24},\pm,\pm,-)$ &$t_3+s_3t_1+s_{31}t_4-2s_{3145}t_5+s_{3145}t_6=s_3\text{eq1}+\text{eq2}$ \\
			\cline{3-6}
			&& \multirow{4}{3cm}{
				$t_1+2s_1t_4-2s_{145}t_5+s_{145}t_6-s_{1456542}t_2$}
			& \multirow{2}{3cm}{$t_3+s_{3}t_2+s_{32}t_4$} 
			&
			\multirow{2}{3cm}{$(-s_2,\pm,s_{24},\pm,\pm,-)$}
			&	$t_3+s_3t_1+s_{31}t_4-2s_{3145}t_5+s_{3145}t_6=
			s_3\text{eq1}+\text{eq2}$\\
			\cline{6-6}
			&&&&& $t_1-s_{12}t_2-2s_{123}t_3-2s_{1235}t_5+s_{1235}t_6 = \text{eq1}-2s_{123} \text{eq2}$\\
			\cline{4-6}
			&& 
			& \multirow{2}{3cm}{$t_3+s_3t_2+s_{32}t_4-2s_{3245}t_5+s_{3245}t_6$}
			&\multirow{2}{3cm}{$(-s_2,\pm,-s_{24},\pm,\pm,-)$}
			
			&
			$t_3+s_3t_1+s_{31}t_4=s_3\text{eq1}+\text{eq2}$
			\\
			\cline{6-6}
			&&&&&$t_1-s_{12}t_2-2s_{123}t_3-2s_{1235}t_5+s_{1235}t_6=\text{eq1}-2s_{123}\text{eq2}$\\
			\cline{1-6}
			
	\end{tabular}}
	\caption{Type IA. In conjunction with Figure~\ref{fig:t1af}. $s_{ij}=s_{i}s_j$ is the sign of $t_it_j$.}\label{tab:t1af}
\end{table}


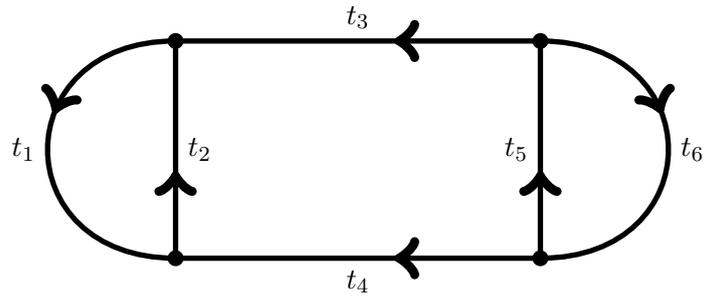
\begin{figure}[h!]
	\centering
	\begin{tikzpicture}[scale=0.48,decoration={
		markings,
		mark=at position 0.4 with {\arrow[line width=1.5mm]{to}}}]
	\tikzstyle{vertex}=[circle,fill=black!100,minimum size=7pt,inner sep=0pt]
	
	\tikzstyle{vertex2}=[circle,fill=black!100,minimum size=6pt,inner sep=0pt]
	%
	
	
	\clip (-10,-4) rectangle (10, 4);

	\node[vertex2] (v1) at (-5,-3) {};
	\node[vertex2] (v2) at (-5,3) {};
	\node[vertex2] (v3) at (5,-3) {};
	\node[vertex2] (v4) at (5,3) {};
	
	
	
	\draw[line width=2pt,postaction={decorate}] (v2.center) to[out=180, in=180,looseness=2] node [midway,left] {$t_1$} (v1.center);
	
	\draw[line width=2pt,postaction={decorate}] (v1.center) to[out=90, in=270]  node [midway,right] (v212) {$t_2$} (v2.center);
	

	\draw[line width=2pt,postaction={decorate}] (v3.center) to[out=180, in=0]  node [midway,below] (v13) {$t_4$} (v1.center);

	\draw[line width=2pt,postaction={decorate}] (v4.center) to[out=180, in=0]  node [midway,above] (v32) {$t_3$} (v2.center);
	\draw[line width=2pt,postaction={decorate}] (v3.center) to[out=90, in=270]  node [midway,left] (v43) {$t_5$} (v4.center);
	\draw[line width=2pt,postaction={decorate}] (v4.center) to[out=0, in=0,looseness=2]  node [midway,right] (v44) {$t_6$} (v3.center);

	\end{tikzpicture}
	
	\caption{Figure type IB.} \label{fig:t1bf}
\end{figure}

\begin{table} 
	\centering
	\small{
		\begin{tabular}{|@{}p{1cm}@{}|@{}p{1cm}@{}|@{}p{3cm}@{}|@{}p{3cm}@{}|@{}p{4cm}@{}|@{}p{4cm}@{}|}
			\hline
			$e_1\in$ & $e_2\in$& eq1 & eq2 & $(s_1,s_2,s_3,s_4,s_5,s_6)$& other cycles  \\
			\hline
			$t_3$\newline ($t_4$ sym.) & \multirow{5}{1.5em}{$t_1$ \newline ($t_2$ sym.)}  &\multirow{3}{3cm}{$t_3-s_{32}t_2-s_{324}t_4+s_{324}t_5$}& \multirow{2}{3cm}{$t_1-s_{14}t_4+s_{14}t_5+s_{145}t_6+s_{1456}t_4+s_{14564}t_2$}
			& \multirow{2}{3cm}{$(-s_2,\pm,\pm,s_{235},\pm,-s_5)$}
			& $t_1+s_1t_2-s_{123}t_3+s_{123}t_6+s_{1236}t_5+s_{12365}t_3=\text{eq2}+2s_{3}\text{eq1}$ \\
			\cline{6-6}
			&&&&  & $t_1-s_{14}t_4-s_{146}t_6+s_{146}t_3=\text{eq2}+s_{146}\text{eq1}$\\
			\cline{4-6}
			
			& & 
			&$t_1+s_{1}t_2$
			
			& \multirow{2}{1.5em}{$(s_2,\pm,\pm,s_{235},\pm,-s_5)$}
			
			& $t_1-s_{14}t_4+s_{14}t_5+s_{145}t_3=s_{145}\text{eq1}+\text{eq2}$\\
			\cline{3-6}
			
			& &\multirow{2}{3cm}{$t_3-s_{32}t_2-s_{324}t_4-s_{3246}t_6$}& \multirow{2}{3cm}{$t_1-s_{14}t_4+s_{14}t_5+s_{145}t_6+s_{1456}t_4+s_{14564}t_2$}
			& \multirow{2}{3cm}{$(-s_2,\pm,\pm,s_{236},\pm,-s_5)$}
			& $t_1+s_1t_2-s_{123}t_3+s_{123}t_6+s_{1236}t_5+s_{12365}t_3=\text{eq2}+2s_{3}\text{eq1}$ \\
			\cline{6-6}
			&&&&  & $t_1-s_{14}t_4+s_{14}t_5+s_{145}t_3=\text{eq2}+s_{145}\text{eq1}$\\
			\cline{1-6}
			
			\multirow{8}{1.5em}{$t_1$}&\multirow{3}{4cm}{$t_2$}
			&\multirow{3}{3cm}{$t_1-s_{14}t_4+s_{14}t_5+s_{145}t_3$}& \multirow{2}{3cm}{$t_2-s_{23}t_3+s_{23}t_6+s_{236}t_4$} 
			&\multirow{2}{4cm}{$(-s_2,\pm,\pm,s_{236},-s_6,\pm)$}
			& $t_1+s_1t_2-s_{123}t_3+s_{123}t_6+s_{1236}t_5+s_{12365}t_3=
			\text{eq1}+s_1\text{eq2}$\\
			\cline{6-6}
			&&&& & 
			$t_1-s_{14}t_4+s_{14}t_5+s_{145}t_6+s_{1456}t_4+s_{14564}t_2=\text{eq1}+s_{14564}\text{eq2}$\\
			\cline{4-6}
			&
			&  
			& $t_2-s_{23}t_3-s_{235}t_5+s_{235}t_4$
			&\multirow{1}{4cm}{$(s_2,\pm,\pm,s_{235},-s_6,\pm)$}
			& $t_1+s_1t_2=\text{eq1}+s_1\text{eq2}$\\
			\cline{2-6}
			&\multirow{5}{1.5em}{$t_5$}
			& \multirow{3}{3cm}{$t_1-s_{14}t_4-s_{146}t_6+s_{146}t_3$}
			&$t_5+s_{5}t_6$
			&\multirow{1}{4cm}{$(\pm,-s_{346},\pm,s_{136},s_6,\pm)$}
			&$t_1-s_{14}t_4+s_{14}t_5+s_{145}t_3=\text{eq1}+s_{14}\text{eq2}$\\
			\cline{4-6}
			& 
			& 
			&\multirow{2}{3cm}{$t_5+s_5t_3-s_{532}t_2-s_{5324}t_4$}
			&
			\multirow{2}{4cm}{$(-s_2,\pm,\pm,-s_{236},-s_6,\pm)$} 
			&$t_2+s_2t_1-s_{214}t_4+s_{214}t_5+s_{2145}t_6+s_{21456}t_4=
			s_2\text{eq1}+s_{214}\text{eq2}$ \\
			\cline{6-6}
			&&&&  & $t_1+s_1t_2-s_{123}t_3+s_{123}t_6+s_{1236}t_5+s_{12365}t_3
			=\text{eq1}+s_{1236}\text{eq2}$\\
			\cline{3-6}
			&
			& \multirow{2}{3cm}{$t_1+s_1 t_2$}
			&$t_5+s_5t_6$
			& \multirow{1}{4cm}{$(s_2,\pm,\pm,s_{236},s_6,\pm)$}
			& NA\\
			\cline{4-6}
			&
			& 
			&$t_5+s_5t_3-s_{532}t_2-s_{5324}t_4$
			&\multirow{1}{4cm}{$(s_2,\pm,\pm,-s_{236},-s_6,\pm)$}
			&$t_1-s_{14}t_4+s_{14}t_5+s_{145}t_3=\text{eq1}+s_{14}\text{eq2}$\\
			\cline{1-6}
			
	\end{tabular}}
	\caption{Type IB. In conjunction with Figure~\ref{fig:t1bf}. $s_{ij}=s_i s_j$ is the sign of $t_it_j$.} \label{tab:t1bf}
\end{table}

\begin{table} 
	\centering
	\small{
		\begin{tabular}{|@{}p{0.7cm}@{}|@{}p{0.7cm}@{}|@{}p{3cm}@{}|@{}p{3cm}@{}|@{}p{3.2cm}@{}|@{}p{4cm}@{}|}
			\hline
			$e_1\in$ & $e_2\in$&  eq1 & eq2 &$(s_1,s_2,s_3,s_4,s_5,s_6)$& other cycles  \\
			\hline
			\multirow{8}{1em}{$t_1$} 
			& \multirow{4}{1.5em}{$t_3$} 
			&\multirow{2}{3cm}{$t_1$}& $t_3+s_{3}t_4-s_{345}t_5+s_{345}t_6+s_{3456}t_4$
			& \multirow{1}{3cm}{$(+,\pm,\pm,-s_{3},\pm,-)$}
			& NA \\
			\cline{4-6}
			& &
			& \multirow{1}{3cm}{$t_3+s_3t_4$}
			&\multirow{1}{3cm}{$(+,\pm,\pm,s_{3},\pm,-)$}& NA \\
			\cline{3-6}
			& &\multirow{2}{3cm}{$t_1-2s_{12}t_2+2s_{12}t_4-2s_{1245}t_5+s_{1245}t_6$}& \multirow{1}{3cm}{$t_3+s_3t_4$}
			&\multirow{1}{3cm}{$(-,\pm,\pm,s_{3},\pm,-)$}
			& $t_1-2s_{12}t_2-2s_{123}t_3-2s_{1235}t_5+s_{1235}t_6=\text{eq1}
			-2s_{123}\text{eq2}$ \\
			\cline{4-6}
			& &
			& $t_3+s_3t_4-2s_{345}t_5+s_{345}t_6$
			&\multirow{1}{3cm}{$(-,\pm,\pm,-s_{3},\pm,-)$}
			& $t_1-2s_{12}t_2+s_{12}t_4+s_{124}t_3=\text{eq1}
			+s_{124}\text{eq2}$ \\
			\cline{2-6}
			&\multirow{5}{1.5em}{$t_6$} &\multirow{2}{3cm}{$t_1$}& 
			$t_6$
			&$(+,\pm,\pm,-s_{3},\pm,+)$
			& 
			NA\\
			\cline{4-6}
			& &
			& $t_6+2s_6t_5+s_{65}t_3+s_{653}t_4$
			&$(+,\pm,\pm,-s_{3},\pm,-)$
			& NA\\
			\cline{3-6}	
			& &\multirow{3}{3cm}{$t_1-2s_{12}t_2+s_{12}t_4+s_{124}t_3$}
			& $t_6$
			&$(-,\pm,\pm,-s_{3},\pm,+)$
			& 
			NA\\
			\cline{4-6}
			& &
			&$t_6+2s_6t_5+s_{65}t_3+s_{653}t_4$
			&
			\multirow{2}{3cm}{$(-,\pm,\pm,-s_{3},\pm,-)$} 
			& $t_1-2s_{12}t_2-2s_{123}t_3-2s_{1235}t_5+s_{1235}t_6=\text{eq1}+s_{1235}\text{eq2}$\\
			\cline{6-6}
			&&&&&$t_1-2s_{12}t_2+2s_{12}t_4-2s_{1245}t_5+s_{1245}t_6=\text{eq1}+s_{1245}\text{eq2}$\\
			\cline{1-6}	
	\end{tabular}}
	\caption{Type IC. In conjunction with Figure~\ref{fig:t1cf}. $s_i$ is the sign of $t_i$. $s_{ij}$ is the sign of $t_it_j$. The case $e_1\in t_3$ and $e_2\in t_6$ is symmetric to the $e_1\in t_1$, $e_2\in t_3$.} \label{tab:t1cf}
\end{table}


\begin{figure}[h!]
	\centering
	\begin{tikzpicture}[scale=0.48,decoration={
		markings,
		mark=at position 0.4 with {\arrow[line width=1.5mm]{to}}}]
	\tikzstyle{vertex}=[circle,fill=black!100,minimum size=7pt,inner sep=0pt]
	
	\tikzstyle{vertex2}=[circle,fill=black!100,minimum size=6pt,inner sep=0pt]
	%
	
	
	\clip (-10,-2.5) rectangle (10, 2.5);

	\node[vertex2] (v1) at (-6,0) {};
	\node[vertex2] (v2) at (-2,0) {};
	\node[vertex2] (v3) at (2,0) {};
	\node[vertex2] (v4) at (6,0) {};
	
	
	
	\draw[line width=2pt,postaction={decorate}] (v1.north) to[out=120, in=240,looseness=40] node [midway,left] {$t_1$} (v1.south);
	
	\draw[line width=2pt,postaction={decorate}] (v2.center) to[out=180, in=0]  node [midway,above] (v212) {$t_2$} (v1.center);
	

	\draw[line width=2pt,postaction={decorate}] (v2.center) to[out=270, in=270]  node [midway,below] (v13) {$t_4$} (v3.center);

	\draw[line width=2pt,postaction={decorate}] (v3.center) to[out=90, in=90]  node [midway,above] (v32) {$t_3$} (v2.center);
	\draw[line width=2pt,postaction={decorate}] (v4.center) to[out=180, in=0]  node [midway,above] (v43) {$t_5$} (v3.center);
	\draw[line width=2pt,postaction={decorate}] (v4.south) to[out=300, in=60,looseness=40]  node [midway,right] (v44) {$t_6$} (v4.north);

	\end{tikzpicture}
	
	\caption{Figure type IC.} \label{fig:t1cf}
\end{figure}
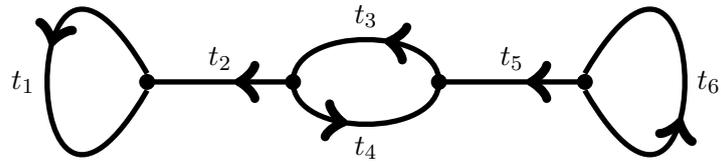

\begin{table}[ht!] 
	\centering
	\small{
		\begin{tabular}{|@{}p{0.7cm}@{}|@{}p{0.7cm}@{}|@{}p{3cm}@{}|@{}p{3cm}@{}|@{}p{3.2cm}@{}|@{}p{4cm}@{}|}
			\hline
			$e_1\in$ & $e_2\in$& eq1 & eq2& $(s_1,s_2,s_3,s_4,s_5,s_6)$ & other cycles  \\
			\hline
			\multirow{4}{1em}{$t_1$} & \multirow{4}{1.5em}{$t_6$} 
			&\multirow{2}{3cm}{$t_1$}& $t_6$
			& $(+,\pm,\pm,\pm,-,+)$
			& NA \\
			\cline{4-6}
			& &
			&$t_6+2s_{6}t_4+2s_{64}t_3+s_{643}t_5$ 
			&$(+,\pm,\pm,\pm,-,-)$
			& NA \\
			\cline{3-6}
			& &\multirow{2}{3cm}{$t_1+2s_1t_2+2s_{12}t_3+s_{123}t_5$}& 
			$t_6$
			&$(-,\pm,\pm,\pm,-,+)$
			& NA \\
			\cline{4-6}
			&
			&
			& $t_6+2s_{6}t_4+2s_{64}t_3+s_{643}t_5$ 
			&$(-,\pm,\pm,\pm,-,-)$ 
			& $t_1+2s_1t_2-2s_{124}t_4+s_{124}t_6=\text{eq1}+s_{124}\text{eq2}$ \\
			\cline{1-6}
	\end{tabular}}
	\caption{Type ID. In conjunction with Figure~\ref{fig:t1df}. $s_i$ is the sign of $t_i$. $s_{ij}$ is the sign of $t_it_j$. Other cases are symmetric.} \label{tab:t1df}
\end{table}


\begin{figure}[ht!]
	\centering
	\begin{tikzpicture}[scale=0.48,decoration={
		markings,
		mark=at position 0.4 with {\arrow[line width=1.5mm]{to}}}]
	\tikzstyle{vertex}=[circle,fill=black!100,minimum size=7pt,inner sep=0pt]
	
	\tikzstyle{vertex2}=[circle,fill=black!100,minimum size=6pt,inner sep=0pt]
	%
	
	
	\clip (-7,-2.5) rectangle (7, 6.5);

	\node[vertex2] (v1) at (-3,0) {};
	\node[vertex2] (v2) at (0,0) {};
	\node[vertex2] (v3) at (0,3) {};
	\node[vertex2] (v4) at (3,0) {};
	
	
	
	\draw[line width=2pt,postaction={decorate}] (v1.north) to[out=120, in=240,looseness=40] node [midway,left] {$t_5$} (v1.south);
	
	\draw[line width=2pt,postaction={decorate}] (v2.center) to[out=180, in=0]  node [midway,below] (v212) {$t_3$} (v1.center);
	

	\draw[line width=2pt,postaction={decorate}] (v3.center) to[out=270, in=90]  node [midway,right] (v13) {$t_2$} (v2.center);

	\draw[line width=2pt,postaction={decorate}] (v3.east) to[out=30, in=150,looseness=40]  node [midway,below] (v32) {$t_1$} (v3.west);
	\draw[line width=2pt,postaction={decorate}] (v4.center) to[out=180, in=0]  node [midway,below] (v43) {$t_5$} (v2.center);
	\draw[line width=2pt,postaction={decorate}] (v4.south) to[out=300, in=60,looseness=40]  node [midway,right] (v44) {$t_6$} (v4.north);

	\end{tikzpicture}
	
	\caption{Figure type ID.} \label{fig:t1df}
\end{figure}
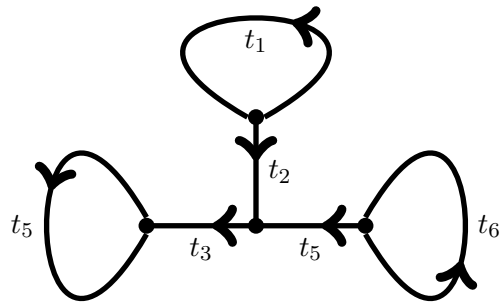

\begin{table}[ht!]
	\centering
	\small{
		\begin{tabular}{|@{}p{0.7cm}@{}|@{}p{0.7cm}@{}|@{}p{3cm}@{}|@{}p{3cm}@{}|@{}p{3.9cm}@{}|@{}p{4cm}@{}|}
			\hline
			$e_1\in$ & $e_2\in$& eq1 & eq2 & $(s_1,s_2,s_3,s_4,s_5,s_6)$& other cycles  \\
			\hline
			\multirow{8}{1em}{$t_1$} & \multirow{4}{1.5em}{$t_2$}  
			&\multirow{2}{3cm}{$t_1+s_1t_6-s_{163}t_3$}
			& $t_2-s_{24}t_4-s_{243}t_3$ 
			& $(s_{36},s_{34},\pm,\pm,-s_{46},+)$
			&
			$t_1+s_1t_6+s_{16}t_4-s_{1642}t_2=\text{eq1}-s_{1642}\text{eq2}$  \\
			\cline{4-6}
			&
			&
			&$t_2+s_2t_5+s_{25}t_6-s_{2563}t_3$
			&$(s_{36},s_{563},\pm,\pm,-s_{46},\pm)$ 
			& $t_1-s_{15}t_5-s_{152}t_2=\text{eq1}-s_{152}\text{eq2}$
			\\
			\cline{3-6}
			&
			&\multirow{2}{3cm}{$t_1-s_{15}t_5-s_{154}t_4-s_{1543}t_3$}
			&$t_2-s_{24}t_4-s_{243}t_3$
			&$(s_{345},s_{34},\pm,\pm,-s_{46},\pm)$ 
			& $t_1-s_{15}t_5-s_{152}t_2=\text{eq1}-s_{152}\text{eq2}$
			\\
			\cline{4-6}
			&
			&
			&$t_2+s_2t_5+s_{25}t_6-s_{2563}t_3$
			&$(s_{345},s_{563},\pm,\pm,-s_{46},\pm)$ 
			& $t_1+s_1t_6+s_{16}t_4-s_{1642}t_2=\text{eq1}-s_{1642}\text{eq2}$
			\\
			\cline{2-6}
			&\multirow{4}{1.5em}{$t_4$}
			&\multirow{2}{3cm}{$t_1-s_{15}t_5-s_{152}t_2$}
			&$t_4-s_{42}t_2-s_{423}t_3$
			&$(s_{25},\pm,\pm,s_{23},\pm,-s_{235})$  
			& $t_1-s_{15}t_5-s_{154}t_4-s_{1543}t_3=\text{eq1}-s_{154}\text{eq2}$
			\\
			\cline{4-6}
			&
			&
			& $t_4+s_4t_5+s_{45}t_6$
			&$(s_{25},-s_{356},\pm,s_{56},\pm,\pm)$ 
			& $t_1+s_1t_6+s_{16}t_4-s_{1642}t_2=\text{eq1}+s_{16}\text{eq2}$ \\
			\cline{3-6}
			& 
			&\multirow{2}{3cm}{$t_1+s_1t_6-s_{163}t_3$}
			&$t_4-s_{42}t_2+s_{42}t_3$
			&$(s_{36},\pm,\pm,s_{23},-s_{236},\pm)$
			& $t_1+s_1t_6+s_{16}t_4-s_{1642}t_2=\text{eq1}+s_{16}\text{eq2}$ \\
			\cline{4-6}
			&
			&
			&$t_4+s_4t_5+s_{45}t_6$
			&$(s_{36},-s_{356},\pm,s_{56},\pm,\pm)$ 
			&$t_1-s_{15}t_5-s_{154}t_4-s_{1543}t_3=\text{eq1}-s_{154}\text{eq2}$  \\
			\cline{1-6}
	\end{tabular}}
	\caption{Type IIB. In conjunction with Figure~\ref{fig:t2bf}. $s_{ij}$ is the sign of $t_it_j$.} \label{tab:t2bf}
\end{table}


\begin{figure}[ht!]
	\centering
	\begin{tikzpicture}[scale=0.48,decoration={
		markings,
		mark=at position 0.4 with {\arrow[line width=1.5mm]{to}}}]
	\tikzstyle{vertex}=[circle,fill=black!100,minimum size=7pt,inner sep=0pt]
	
	\tikzstyle{vertex2}=[circle,fill=black!100,minimum size=6pt,inner sep=0pt]
	%
	
	
	\clip (-4.,-4.5) rectangle (4., 3.5);

	\node[vertex2] (v1) at (-3,-2) {};
	\node[vertex2] (v2) at (3,-2) {};
	\node[vertex2] (v3) at (0,3) {};
	\node[vertex2] (v4) at (0,0) {};
	
	
	
	\draw[line width=2pt,postaction={decorate}] (v4.center) to[out=90, in=270,] node [midway,right] {$t_1$} (v3.center);
	
	\draw[line width=2pt,postaction={decorate}] (v4.center) to[out=330, in=140]  node [midway,below] (v212) {$t_2$} (v2.center);
	

	\draw[line width=2pt,postaction={decorate}] (v4.center) to[out=210, in=30]  node [midway,above] (v13) {$t_3$} (v1.center);

	\draw[line width=2pt,postaction={decorate}] (v1.center) to[out=300, in=240,looseness=1]  node [midway,below] (v32) {$t_4$} (v2.center);
	\draw[line width=2pt,postaction={decorate}] (v2.center) to[out=60, in=0]  node [midway,above] (v43) {$t_5$} (v3.center);
	\draw[line width=2pt,postaction={decorate}] (v3.center) to[out=180, in=120,looseness=1]  node [midway,left] (v44) {$t_6$} (v1.center);

	\end{tikzpicture}
	\caption{Figure type IIB.} \label{fig:t2bf}
\end{figure}
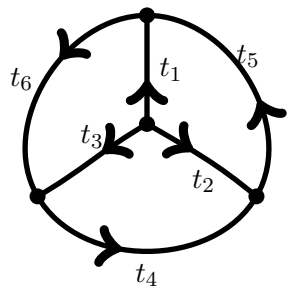

\section{Dictionary of polynomials}\label{app:dictionary}


Table~\ref{tab.1} below gives details of the graph polynomials that feature or are mentioned in this paper, and Table~\ref{fig:M4} describes some relations between them.
\small{
\begin{table}[htp]
	\begin{tabular}{|p{3.8cm}|p{9cm} |@{}c@{}|}
	\hline 
	Name & Expression & Reference\\
	\hline
	Tutte poly. graph $\Gamma$ \newline $T_{\Gamma}(X,Y)$ & \(\displaystyle \sum_{A \subseteq E}(X-1)^{k(\Gamma\backslash A^c)-k(\Gamma)}(Y-1)^{|A|-|V|+k(\Gamma\backslash A^c)}\)& \eqref{eq:Tutte_subset}\\
	\hline
	Tutte poly. \newline matroid~$M$\newline $T_{M}(X,Y)$ & \(\displaystyle \sum_{A \subseteq E}(X-1)^{r(E)-r(A)}(Y-1)^{|A|-r(A)}\)& \eqref{eq:tutte_matroid}\\
	\hline
	trivar. Tutte poly.  \newline 
	pair of matroids $(M_1,M_2)$\newline 
	$S_{M_1,M_2}(X,Y,Z)$ &
	\smash[b]{
		\begin{tabular}[t]{c}
	 	\( \displaystyle
	 	 \sum_{A \subseteq E} (X-1)^{r_1(E)-r_1(A)}(Y-1)^{|A|-r_2(A)}\cdot\)
	 	 \\
	 	\(\displaystyle 
	 	\qquad\cdot (Z-1)^{r_2(A)+r_1(E)-r_1(A)} \)
	 \end{tabular}}
	 &\eqref{equation:sm1m2_2}\\
	 \hline
	 linking poly.\newline 
	 pair of matroids $(M_1,M_2)$\newline
	 $Q_{M_1,M_2}(x,y,u,v)$ &
	 \smash[b]{
	 	\begin{tabular}[t]{c}
	 \(\displaystyle \sum_{A\subseteq E} (x-1)^{r_1(E)-r_1(A)}(y-1)^{|A|-r_1(A)}\cdot\)
	 \\
	 \(\displaystyle
	 \qquad \cdot (u-1)^{r_2(E)-r_2(A)}
	 (v-1)^{|A|-r_2(A)} \)
	 \end{tabular}}
	  &\smash[b]{
	  	\begin{tabular}[t]{l} App.~\ref{app:matroids}\\
	  \eqref{eq:J}\\
	\end{tabular}
}\\
	 \hline
	 
	 Las Vergnas poly.\newline
	 mat. persp. $M_2\to M_1$  \newline
	 $T_{M_2\to M_1}(X,Y,Z)$
	 &\smash[b]{
	 	\begin{tabular}[t]{l}
	 \(\displaystyle
 \sum_{A \subset E}  (X-1)^{r_{1}(E)-r_{1}(A)} (Y-1)^{|A|-r_{2}(A)} \cdot\)
 \\
	 \(\displaystyle \qquad  \cdot Z^{r_{2}(E)-r_{1}(E)-(r_{2}(A)-r_{1}(A))} \) \end{tabular}}
	 	& \cite[(5.2)]{vergnas1999}\\
	 	\hline
	 	trivar. Tutte  \newline
	 	signed graph $\Sigma$ \newline
	 		$T_{\Sigma}(X,Y,Z)$
	 	&
	 	\smash[b]{
	 		\begin{tabular}[t]{l}
	 		\(\displaystyle	
	 \sum_{A \subseteq E}(X-1)^{k(\Sigma\backslash A^c)-k(\Sigma)}(Y-1)^{|A|-|V|+k_{b}(\Sigma\backslash A^c)}\cdot\)
	        \\
	  \(\displaystyle
	  \qquad \cdot (Z-1)^{k_{u}(\Sigma\backslash A^c)} \)
  \end{tabular}
  }
	 	& \eqref{eq:signed_Tutte_subset} \\
	 	\hline
	 	flow polynomial \newline
	 	signed graph $\Sigma$\newline
	 	finite abelian group $G$\newline
	 	$q_\Sigma(G)$ 
	 	&
	 	\smash[b]{
	 		\begin{tabular}[t]{c}
	\( \displaystyle \sum_{A \subseteq E}(-1)^{|A^c|}|G|^{|A|-|V|+k_b(\Sigma\backslash A^c)}\left( |G|/|2G|\right)^{k_u(\Sigma\backslash A^c)}\)
	\end{tabular}}
	 	& \eqref{equation:flowcountsigned}\\
	 	\hline
	 	chromatic poly. \newline
	 		 signed graph $\Sigma$\newline
	 	$\{0,\pm1,\ldots,\pm n\}$\newline
	 	$\chi_\Sigma(2n\!+\!1)$ &
	 \( \displaystyle (-1)^{|V|-k(\Gamma)}(2n\!+\!1)^{k(\Gamma)}T_{\Sigma}\left(-2n,0,2n/(2n+1)\right)\) &\eqref{equation:chromatic_poly} \\
	 \hline
	 nz pot. diff. \newline
	 signed graph  $\Sigma$ \newline
	 finite abelian group $G$\newline 
	 $p_\Sigma(G)$ 
	 &
	 \(\displaystyle (-1)^{r(\Gamma)}|2G|^{k_u(\Sigma)}T_{\Sigma}\Big(1-|G|,0,1-1/|2G|\Big)\)
	 &
\eqref{equation:potential_differences}\\
\hline
surface Tutte poly. \newline
map $M$ \newline
$\T(M;\x,\y)$
&
\(\displaystyle \sum_{A \subseteq E}x^{n^*(M/A)}y^{n(M\backslash A^c)} \prod_{\substack{\mathrm{conn. }\text{ }\mathrm{cpts}\\ M_i \text{ }\mathrm{of}\text{ } M/A}}x_{\bar{g}(M_i)} \prod_{\substack{\mathrm{conn. }\text{ }\mathrm{cpts}\\ M_j \text{ }\mathrm{of}\text{ } M\backslash A^c}}y_{\bar{g}(M_j)}\)
& \cite[(4)]{goodall2020tutte}\\
\hline
Krieger--O'Connor poly. \newline
signed graph $\Sigma$ \newline
$C_\Sigma(x,y,z)$
&
\smash[b]{
	\begin{tabular}[t]{l}
\(\displaystyle 
\sum_{A\subset E}(x-1)^{k_b(\Sigma\setminus A^c)-k_b(\Sigma)} (y-1)^{|A|-|V|+k_b(\Sigma\setminus A^c)}\cdot\)\\
\(\displaystyle\qquad \cdot (z-1)^{k_u(\Sigma\setminus A^c)}\)
\end{tabular}} &\cite{kriocon13} \\
\hline
\end{tabular}
\caption{Dictionary of polynomials.}\label{tab.1}
	\end{table}
}

\noindent
Notation:
\begin{itemize}
	\item $k(\cdot)$: number of connected components
	\item $r_{\cdot}(\cdot)$:  matroid rank; $r_i(A)$, rank of the set $A$ in matroid $M_i$
		\item $V$: vertex set of graph or signed graph
\item $E$: edge set of graph or signed graph, ground set of matroid 
	\item $A^c$, complement of $A$ in $E$,  or $A^c=E\setminus  A$

\item $\Sigma\setminus A^c$: spanning graph $(V(\Sigma),A)$.
	\item $k_u(A)$/$k_b(A)$: number of unbalanced/balanced connected components of signed graph $(V(\Sigma),A)$
	\item $\mathbf{x}=(x;\ldots,x_{-(i+1)},x_{-i},\ldots,x_{-1},x_0,x_1,\ldots,x_i,x_{i+1},\ldots)$, and similarly $\mathbf{y}$,  infinite vector of variables.
	\item $M\backslash A^c$: map obtained from $M$ by deleting edges not in $A$  (see~\cite{goodall2017tutte} or \cite[Definition~A.1]{goodall2020tutte} for edge deletion in maps); the surface in which $M\backslash A^c$ is embedded may differ from that of $M$. 
	\item $M/A$:  map obtained from $M$ by contracting edges in $A$ (see~\cite[Definition A.3]{goodall2020tutte} for edge contraction in maps -- notably, contracting a loop may create an extra vertex); the surface in which $M/A$ is embedded may differ from that of $M$).
	\item $\bar{g}(M_i)$:  signed genus of connected component $M_i$ of given submap of $M$, equal to the number of handles if $M_i$ orientably embedded, and minus the number of cross-caps if $M_i$ is non-orientably embedded.
	
	\item $n(M\setminus A^c)=e(M\setminus A^c)-v(M\setminus A^c)+k(M\setminus A^c)=|A|-|V(M)|+k(M\setminus A^c)$
	\item $n^{\ast}(M/ A)=f(M/A)-v(M/A)+k(M/A)$ where $f(M/A)$ is the number of faces of the map $M/A$, and $v(M/A)$ is the number of vertices of the map $M/A$.
\end{itemize}

\small{
	\begin{table}[ht]
	\begin{tabular}{|@{}p{3.2cm}@{}|@{}p{3.2cm}@{}|@{}c@{}|@{}c@{}|}
		\hline 
		Name & Name & Relation & Ref.\\
		\hline
		trivar. Tutte poly.\newline pair of matroids \newline $(M_1,M_2)$\newline $S_{M_1,M_2}(X,Y,Z)$ & 
			 Las Vergnas poly.\newline
		matroid persp. \newline  $M_2\to M_1$  \newline
		$T_{M_2,M_1}(X,Y,Z)$
		&
\(\displaystyle		S_{M_1,M_2}(X,Y,Z) = (Z-1)^{r_2(E)}T_{M_2,M_1}\left(X,Y,1/(Z-1)\right)\)
		& \eqref{equation:lasvergnas}\\
		\hline
		trivar. Tutte poly.\newline 
		signed graph $\Sigma$ \newline
		$T_\Sigma (X,Y,Z)$
		&
		trivar. Tutte poly.\newline
		pair of matroids $(M_1,M_2)$\newline
		$S_{M_1,M_2}(X,Y,Z)$
		&
		\(\displaystyle
	T_{\Sigma}(X,Y,Z) = (Z-1)^{-r_M(E)}S_{M(\Gamma), F(\Sigma)}(X,Y,Z).
		\)
		& 	\smash[b]{\begin{tabular}[t]{l}
				\eqref{def_tutte_from_mat_2} \\
		\eqref{def_tutte_from_mat} \\
	\end{tabular}}
\\
		\hline
		 linking poly.\newline 
	pair of matroids \newline $(M_1,M_2)$\newline
	$Q_{M_1,M_2}(x,y,u,v)$ &
	trivar. Tutte\newline 
		pair of matroids $(M_1,M_2)$\newline
	$S_{M_1,M_2}(x,y,u,v)$ &
			\smash[b]{\begin{tabular}[t]{l}
					\(\displaystyle Q_{M_1,M_2}(x,y,u,v)=(u-1)^{r_2(E)-r_1(E)}\cdot \)\\
					\(\displaystyle \:\cdot S_{M_1,M_2} \Big(1+(x\!-\!1)(u\!-\!1),1+(y\!-\!1)(v\!-\!1), 1+\frac{y\!-\!1}{u\!-\!1}\Big)\)
					\end{tabular}}
& \eqref{eq:Q_from_T}\\
		\hline
			 	flow polynomial \newline
		signed graph $\Sigma$\newline
		fin. ab. group $G$\newline
		$q_\Sigma(G)$ 
		&
				trivar. Tutte poly.\newline 
		signed graph $\Sigma$ \newline
		$T_\Sigma (X,Y,Z)$
		&
		\smash[b]{\begin{tabular}[t]{@{}c@{}}
		\(\displaystyle q_\Sigma(G) = (-1)^{|E|-|V|+k(\Gamma)}T_{\Sigma}\left(0,1-|G|,1-|G|/|2G|\right) \)
		\end{tabular}}
		& \eqref{eq.flow-eval-3var}\\
		\hline 
			 trivar. Tutte  poly.\newline
		signed graph $\Sigma$ \newline
		$T_{\Sigma}(X,Y,Z)$
		&
		surface Tutte poly.\newline
		map $M$ \newline
		$\T(M;\x,\y)$
		&
		 
\smash[b]{\begin{tabular}[t]{@{}c@{}}\(\displaystyle T_{\Sigma}(X,Y,Z)=(X-1)^{-k(M)}\T(M;\x,\y)\) \\
$\Sigma$ underlying the map $M$
\end{tabular}}
 
		&
	\eqref{eq.signed_eval_surf}\\
	\hline

	Krieger--O'Connor poly. \newline
	signed graph $\Sigma$ \newline
	$C_\Sigma(x,y,z)$
	&
		trivar. Tutte poly. \newline
	signed graph $\Sigma$ \newline
	$T_{\Sigma}(X,Y,Z)$
	&
	\(\displaystyle C_\Sigma (x,y,z)= (z-1)^{k_u(\Sigma)}T_{\Sigma}\left(x,y,1+(z-1)/(x-1)\right) \)
	&
	Tab~\ref{tab.1}\\
	\hline
	
	\end{tabular}
\caption{Relations between some of the polynomials in Table~\ref{tab.1}.}\label{fig:M4}
\end{table}
}

 \bibliographystyle{abbrv} 
 \bibliography{biblio}

\begin{thebibliography}{10}

\bibitem{beck06}
M.~Beck and T.~Zaslavsky.
\newblock The number of nowhere-zero flows on graphs and signed graphs.
\newblock {\em J. Combin. Theory Ser. B}, 96(6):901--918, 2006.

\bibitem{biggs97}
N.~Biggs.
\newblock Algebraic potential theory on graphs.
\newblock {\em Bull. London Math. Soc.}, 29(6):641--682, 1997.

\bibitem{bollobas99}
B.~Bollob\'{a}s and O.~Riordan.
\newblock A {T}utte polynomial for coloured graphs.
\newblock {\em Combin. Probab. Comput.}, 8(1-2):45--93, 1999.

\bibitem{bollobas01}
B.~Bollob\'as and O.~Riordan.
\newblock A polynomial invariant of graphs on orientable surfaces.
\newblock {\em Proc. London Math. Soc. (3)}, 83(3):513--531, 2001.

\bibitem{bollobas02}
B.~Bollob{\'a}s and O.~Riordan.
\newblock A polynomial of graphs on surfaces.
\newblock {\em Math. Ann.}, 323(1):81--96, 2002.

\bibitem{bouchet83}
A.~Bouchet.
\newblock Nowhere-zero integral flows on a bidirected graph.
\newblock {\em J. Combin. Theory Ser. B}, 34(3):279--292, 1983.

\bibitem{bryl92}
T.~Brylawski and J.~Oxley.
\newblock The {T}utte polynomial and its applications.
\newblock In {\em Matroid {A}pplications}, volume~40 of {\em Encyclopedia Math.
  Appl.}, pages 123--225. Cambridge Univ. Press, Cambridge, 1992.

\bibitem{butler12}
C.~Butler.
\newblock A quasi-tree expansion of the {K}rushkal polynomial.
\newblock {\em Adv. Appl. Math.}, 94:3--22, 2018.

\bibitem{chen09}
B.~Chen and J.~Wang.
\newblock The flow and tension spaces and lattices of signed graphs.
\newblock {\em European J. Combin.}, 30(1):263--279, 2009.

\bibitem{Chm}
S.~Chmutov.
\newblock Topological {T}utte polynomial.
\newblock {\em arXiv:1708.08132}, 2017.

\bibitem{CMNR16}
C.~Chun, I.~Moffatt, S.~Noble, and R.~Rueckeriemen.
\newblock Matroids, delta-matroids and embedded graphs.
\newblock {\em arXiv:1403.0920}, 2016.

\bibitem{crapo69}
H.~Crapo.
\newblock The {T}utte polynomial.
\newblock {\em Aequationes Math.}, 3(3):211--229, 1969.

\bibitem{devos17}
M.~DeVos, E.~Rollov{\'a}, and R.~{\v{S}}{\'a}mal.
\newblock A note on counting flows in signed graphs.
\newblock {\em arXiv:1701.07369}, 2017.

\bibitem{dupont18}
C.~Dupont, A.~Fink, and L.~Moci.
\newblock Universal {T}utte characters via combinatorial coalgebras.
\newblock {\em Algebr. Comb.}, 1(5):603--651, 2018.

\bibitem{EMM13}
J.~Ellis-Monaghan and I.~Moffatt.
\newblock {\em Graphs on {S}urfaces: {D}ualities, {P}olynomials, and {K}nots},
  volume~84.
\newblock Springer, 2013.

\bibitem{FZ}
D.~Forge and T.~Zaslavsky.
\newblock Lattice points in orthotopes and a huge polynomial {T}utte invariant
  of weighted gain graphs.
\newblock {\em Journal of Combinatorial Theory, Series B}, 118:186--227, 2016.

\bibitem{GR01}
C.~Godsil and G.~Royle.
\newblock {\em Algebraic {G}raph {T}heory}, volume 207 of {\em Graduate Texts
  in Mathematics}.
\newblock Springer-Verlag, New York, 2001.

\bibitem{goodall16}
A.~Goodall, T.~Krajewski, G.~Regts, and L.~Vena.
\newblock A {T}utte polynomial for maps.
\newblock {\em Combin. Probab. Comput.}, 27(6):913--945, 2018.

\bibitem{goodall2017tutte}
A.~Goodall, B.~Litjens, G.~Regts, and L.~Vena.
\newblock A {T}utte polynomial for non-orientable maps.
\newblock {\em Electron. Notes Discrete Math.}, 61:513--519, 2017.

\bibitem{goodall2020tutte}
A.~Goodall, B.~Litjens, G.~Regts, and L.~Vena.
\newblock A {T}utte polynomial for maps {II}: the non-orientable case.
\newblock {\em European J. Combin.}, 86:103095, 2020.

\bibitem{har53}
F.~Harary.
\newblock On the notion of balance of a signed graph.
\newblock {\em Michigan Math. J.}, 2(2):143--146, 1953--1954.

\bibitem{helmeguizon05}
L.~Helme-Guizon and Y.~Rong.
\newblock A categorification of the chromatic polynomial.
\newblock {\em Algebr. Geom. Topol.}, (5):1365--1388, 2005.

\bibitem{jassohernandez06}
E.~Jasso-Hernandez and Y.~Rong.
\newblock A categorification of the {T}utte polynomial.
\newblock {\em Algebr. Geom. Topol.}, (6):2031--2049, 2006.

\bibitem{kauf89}
L.~Kauffman.
\newblock A {T}utte polynomial for signed graphs.
\newblock {\em Discrete Appl. Math.}, 25(1):105--127, 1989.

\bibitem{khovanov00}
M.~Khovanov.
\newblock A categorification of the {J}ones polynomial.
\newblock {\em Duke Math. J.}, (101):359--426, 2000.

\bibitem{krajewski18}
T.~Krajewski, I.~Moffatt, and A.~Tanasa.
\newblock Hopf algebras and {T}utte polynomials.
\newblock {\em Adv. Applied Math.}, 95:271--330, 2018.

\bibitem{kriocon13}
A.~Krieger and B.~O'Connor.
\newblock Tutte polynomial of signed graphs and its categorification,
  {\tt\url{https://people.math.osu.edu/chmutov.1/wor-gr-su13/pres.pdf}}.

\bibitem{krushkal11}
V.~Krushkal.
\newblock Graphs, links, and duality on surfaces.
\newblock {\em Combin. Probab. Comput.}, 20(2):267--287, 2011.

\bibitem{vergnas75}
M.~Las~Vergnas.
\newblock Extensions normales d'un matro\"{\i}de, polyn\^{o}me de {T}utte d'un
  morphisme.
\newblock {\em C. R. Acad. Sci. Paris S\'{e}r. A-B}, 280(22):Ai, A1479--A1482,
  1975.

\bibitem{vergnas80}
M.~Las~Vergnas.
\newblock On the {T}utte polynomial of a morphism of matroids.
\newblock {\em Ann. Discrete Math.}, 8:7--20, 1980.

\bibitem{vergnas1999}
M.~Las~Vergnas.
\newblock The {T}utte polynomial of a morphism of matroids {I}: {S}et-pointed
  matroids and matroid perspectives.
\newblock {\em Ann. Inst. Fourier (Grenoble)}, 49(3):973--1015, 1999.

\bibitem{litjens18}
B.~Litjens and B.~Sevenster.
\newblock Partition functions and a generalized coloring-flow duality for
  embedded graphs.
\newblock {\em J. Graph Theory}, 88(2):271--283, 2018.

\bibitem{mrs16}
E.~M{\'a}{\v{c}}ajov{\'a}, A.~Raspaud, and M.~{\v{S}}koviera.
\newblock The chromatic number of a signed graph.
\newblock {\em Electron. J. Combin.}, 23(1):1--14, 2016.

\bibitem{mohar01}
B.~Mohar and C.~Thomassen.
\newblock {\em Graphs on {S}urfaces}.
\newblock J. Hopkins Univ. Press, 2001.

\bibitem{oxley06}
J.~Oxley.
\newblock {\em Matroid {T}heory}, volume~21 of {\em Oxford Graduate Texts in
  Mathematics}.
\newblock Oxford University Press, Oxford, second edition, 2011.

\bibitem{qian18}
J.~Qian and X.~Ren.
\newblock Flow polynomials of a signed graph.
\newblock {\em Electron. J. Combin.}, 26(3):Paper 3.37, 16, 2019.

\bibitem{tutte49}
W.~T. Tutte.
\newblock On the imbedding of linear graphs in surfaces.
\newblock {\em Proc. London Math. Soc. (2)}, 51:474--483, 1949.

\bibitem{tutte54}
W.~T. Tutte.
\newblock A contribution to the theory of chromatic polynomials.
\newblock {\em Canad. J. Math}, 6(80-91):3--4, 1954.

\bibitem{WK04}
D.~Welsh and K.~Kayibi.
\newblock A linking polynomial of two matroids.
\newblock {\em Adv. Appl. Math.}, 32(1):391 -- 419, 2004.

\bibitem{zas82b}
T.~Zaslavsky.
\newblock Chromatic invariants of signed graphs.
\newblock {\em Discrete Math.}, 42(2-3):287--312, 1982.

\bibitem{zas82a}
T.~Zaslavsky.
\newblock Signed graph coloring.
\newblock {\em Discrete Math.}, 39(2):215--228, 1982.

\bibitem{zas82}
T.~Zaslavsky.
\newblock Signed graphs.
\newblock {\em Discrete Appl. Math.}, 4(1):47--74, 1982.

\bibitem{zas92}
T.~Zaslavsky.
\newblock Orientation embedding of signed graphs.
\newblock {\em J. Graph Theory}, 16(5):399--422, 1992.

\bibitem{zas95}
T.~Zaslavsky.
\newblock Biased graphs {III}: chromatic and dichromatic invariants.
\newblock {\em J. Combin. Theory Ser. B}, 64(1):17--88, 1995.

\bibitem{zas96}
T.~Zaslavsky.
\newblock {\em A {M}athematical {B}ibliography of {S}igned and {G}ain {G}raphs
  and {A}llied {A}reas}.
\newblock Dpt. of Mathematical Sciences Binghampton University, 1996.

\end{thebibliography}

\end{document}